\documentclass[12pt,reqno,oneside]{amsart}
\makeindex
\usepackage{amsmath,amsthm,amssymb}
\usepackage{graphicx,xcolor}

\usepackage{etoolbox}
\numberwithin{figure}{section}
\numberwithin{equation}{section}
 
\DeclareFontFamily{U}{mathb}{\hyphenchar\font45}
\DeclareFontShape{U}{mathb}{m}{n}{
	<-6> mathb5 <6-7> mathb6 <7-8> mathb7
	<8-9> mathb8 <9-10> mathb9
	<10-12> mathb10 <12-> mathb12
}{}
\DeclareSymbolFont{mathb}{U}{mathb}{m}{n}
\DeclareMathSymbol{\llcurly}{\mathrel}{mathb}{"CE}
\DeclareMathSymbol{\ggcurly}{\mathrel}{mathb}{"CF}

\definecolor{Red}{cmyk}{0,1,1,0}

\definecolor{Blue}{cmyk}{1,1,0,0}

\theoremstyle{plain}

\newtheorem{theorem}{Theorem}[section]
\newtheorem{proposition}[theorem]{Proposition}
\newtheorem{lemma}[theorem]{Lemma}
\newtheorem{corollary}[theorem]{Corollary}
\theoremstyle{definition} \theoremstyle{remark}
\newtheorem{remark}[theorem]{Remark}

\newtheorem{definition}[theorem]{Definition}

\newcommand{\Leb}{\operatorname{Leb}}

\renewcommand{\ge}{\geqslant}

\renewcommand{\le}{\leqslant}

\usepackage{mathrsfs}
\usepackage[colorlinks=true,linkcolor=blue, urlcolor=blue, citecolor=blue,backref]{hyperref}	

\makeatletter
\@namedef{subjclassname@2020}{%
  \textup{2020} Mathematics Subject Classification}
\makeatother

\newcommand{\norm}[1]{\left\Vert#1\right\Vert}

\newcommand{\lquot}[2]{#1\text{\textbackslash}\hspace{.1em}#2}

\newcommand{\R}{\mathbb R}
\newcommand{\HP}{\mathbb H}

\newcommand{\N}{\mathbb N}
\newcommand{\Z}{\mathbb Z}

\newcommand{\CCC}{{\mathcal C}}

\newcommand{\PPP}{{\mathcal P}}

\newcommand{\Ph}{{\widehat{P}}}
\newcommand{\Pt}{{\widetilde{P}}}
\newcommand{\Rh}{{\widehat{R}}}

\newcommand{\rh}{{\widehat{r}}}
\newcommand{\rt}{{\widetilde{r}}}
\newcommand{\Fh}{{\widehat{F}}}
\newcommand{\Ft}{{\widetilde{F}}}
\newcommand{\Gh}{{\widehat{G}}}
\newcommand{\Gt}{{\widetilde{G}}}
\newcommand{\Dt}{{\widetilde{\Delta}}}

\newcommand{\Isq}[1]{I_{s_0 \dots s_{#1}}^{q_0  \dots  q_{#1}}}

\usepackage[font={scriptsize},labelfont={bf}, labelsep={period}]{caption}
\usepackage{mathdots}
\usepackage[figuresright]{rotating}
\usepackage{pdflscape}

\usepackage{relsize}

\usepackage{fullpage}
\usepackage{graphicx}
\usepackage{fancyvrb}
\usepackage[latin1]{inputenc}
\usepackage[T1]{fontenc}
\usepackage[main=english]{babel}
\usepackage{wrapfig}
\usepackage{amsfonts}
\usepackage{amssymb}
\usepackage{amsmath}
\usepackage{eucal}
\usepackage{xcolor}
\usepackage{color}
\usepackage{float}
\usepackage{amsthm}
\usepackage{enumitem}
\usepackage{graphics} 
\usepackage{lmodern}
\usepackage[parfill]{parskip}
\usepackage{caption}
\usepackage{setspace}
\usepackage{listings}
\usepackage{verbatim}
\usepackage{hyperref}
\usepackage{dsfont}

\usepackage{multicol}

\usepackage{booktabs}

\allowdisplaybreaks

\newcommand\textline[4][t]{%
  \par\smallskip\noindent\parbox[#1]{.333\textwidth}{\raggedright\texttt{}#2}%
  \parbox[#1]{.333\textwidth}{\centering#3}%
  \parbox[#1]{.333\textwidth}{\raggedleft\texttt{#4}}\par\smallskip%
}

\begin{document}

\author{Nicola Bertozzi}%
\address{Dipartimento di Matematica, Universit\`a di Pisa, Largo Bruno
  Pontecorvo 5, 56127 Pisa, Italy} \email{nicola.bertozzi@phd.unipi.it}
  
\author{Claudio Bonanno}%
\address{Dipartimento di Matematica, Universit\`a di Pisa, Largo Bruno
  Pontecorvo 5, 56127 Pisa, Italy} \email{claudio.bonanno@unipi.it}
  
\author{Paulo Varandas}%
\address{CIDMA and Departamento de Matem\'atica, Universidade de Aveiro, Campus Universit\'ario de Santiago 3810-193, Aveiro, Portugal} \email{paulo.varandas@ua.pt}


\title{Exponential Mixing for Hyperbolic Flows on Non-Compact Spaces}

\begin{abstract}
We introduce a family of hyperbolic flows on non-compact phase spaces that includes the geodesic flow on the modular surface. For these systems we prove exponential decay of correlations for sufficiently regular observables with respect to its SRB measure. Our approach follows the dynamical method of Dolgopyat and subsequent developments for suspension flows with uniformly hyperbolic Poincar\'e maps satisfying a uniform non-integrability condition.
To fit this framework, we construct a suspension model via a triple inducing scheme that yields a uniformly hyperbolic Poincar\'e map with a countable Markov partition. We show that the resulting roof function is cohomologous to one that is constant along stable leaves and satisfies the required non-integrability and tail conditions. As an application, we recover a dynamical proof on Ratner's exponential mixing for the geodesic flow on the modular surface.
\end{abstract}

\maketitle
\section{Introduction} \label{sec:intro}

One of the major cores of the modern theory of mixing flows is the study of hyperbolic, geodesic, and horocycle flows on negatively curved surfaces. In particular, geodesic flows on compact surfaces with negative curvature exhibit a rich dynamics which translates into exponential rate of decay of correlations \cite{dolgopyat}. In some remarkable cases, compactness turns out to be non-necessary: in 1987, Ratner proved that the geodesic flow on the modular surface mixes exponentially fast (cf. \cite{ratner}). She exploited the nice algebraic structure of the phase space through an argument based on harmonic analysis and representation theory. In the perspective of confirming the conjecture that, in general, negative curvature is sufficient for exponential decay of correlations, in 1998 Chernov explored the possibility of showing the same result through a more classical and general dynamical argument, involving Markov partitions (cf. \cite{chernov}). Taking into account some additional geometric properties, namely uniform non-integrability of the stable and unstable manifolds, he recovered a stretched-exponential estimate for the correlation functions of contact three-dimensional Anosov flows and geodesic flows on compact surfaces of variable negative curvature. These results were then enhanced by the works of Dolgopyat in the late nineties (cf. \cite{dolgopyat, dolgopyatII}), who set up a general framework for proving exponential decay of correlations for sufficiently regular Anosov flows. His method relies on the uniform non-integrability and the strong spectral properties of the transfer operator, which are used to estimate the Laplace transform of the correlation function.
The exponential decay of correlations for contact Anosov flows was studied by Liverani~\cite{li}.

In this paper we introduce a family of hyperbolic flows, which includes the geodesic flow on the modular surface, and we show that each of these systems exhibits exponential decay of correlations with respect to sufficiently regular observables and a physical measure which, in the modular surface case, coincides with the Liouville measure. In particular, we employ the dynamical method outlined by Dolgopyat and then extended by Baladi and Vall\'ee in \cite{baladivallee}, Avila, Gou\"ezel and Yoccoz in \cite{avilaGouezelyoccoz},
and Ara\'ujo and Melbourne \cite{araujomelbourne}. These papers provide criteria for proving exponential decay of correlations for the Sinai-Ruelle-Bowen (SRB) measures of flows that can be modelled as suspension flows over Poincar\'e maps with a good hyperbolic structure and whose roof function satisfy some uniform non-integrability condition (UNI).  
These results have been applied in the study of decay of correlations for both hyperbolic and non-uniformly hyperbolic flows, including the Lorenz attractor (see \cite{ABV,ArVar,BW} and references therein). There are also results on the exponential decay of correlations for flows associated with other equilibrium states (see \cite{daltrovarandas,Pollicott,Stoyanov,TZ} and references therein). Related interesting results concerning geodesic flows include \cite{bonthonneauweich}.

While all the previous contexts above deal with flows on a compact phase space, with the notable exception of \cite{bonthonneauweich}, here we are able to show that the latter approach can also be used to tackle geodesic flows in non-compact phase spaces, under suitable assumptions. More precisely, Bonanno, Del Vigna, and Isola \cite{bonanno_delvigna_isola} produced an isomorphism between the geodesic flow on the modular surface and a suspension flow over an appropriate Poincar\'e section
which consists of a skew-product on an unbounded domain on the plane and, in the first component, the dynamics assumes a neutral behavior similar to the Maneville-Pommeau map close to the origin, has a singularity at $x=1$ and has derivative equal to 1 in the unbounded connected component of its domain (see equation~\eqref{eq:defbase}). Our approach is to model the geodesic flow on the modular surface by a suspension flow that fits the assumptions of the known criteria for exponential decay of correlations. In order to do so, we proceed with a triple inducing scheme to guarantee that the resulting Poincar\'e map is uniformly hyperbolic defined over a countable Markov partition. The drawback is that the triple inducing generates an extremely more complicate roof function which could be non-constant along stable manifolds for the Poincar\'e map. In case of the geodesic flow (and the models that we propose here), the resulting roof function can be proved to be cohomologous to a roof function which is constant along the stable foliation, which allows to reduce the correlation function from the suspension flow to the expanding semiflow case and to apply the results from Avila, Gou\"ezel and Yoccoz \cite{avilaGouezelyoccoz}. 

This paper is organized as follows. 
In Section \ref{sec:family} we describe the family of suspension flows that we will consider and we state the main results. In Section \ref{sec:standard} we will recall the criterion for exponential decay of correlations for the SRB measure, due to Avila, Gou\"ezel, and Yoccoz \cite{avilaGouezelyoccoz}.
Section \ref{sec:base} is devoted to the study of the Poincar\'e maps, by inducing it on a subset of the original phase space (which is bounded on the predominantly unstable direction but remains unbounded in the predominantly contractive component) and by considering an accelerated version of the system
(taking the second iterate of such map) that leads to uniform expansion. We emphasize that the domains are still unbounded, hence have non-compact closure.  Section \ref{sec:roof function} deals with the roof function and its induced counterpart. We prove it is cohomologous to a much more convenient roof function, allowing to verify that 
it satisfies uniform non-integrability condition and has exponential tails.
 In Section \ref{sec:expdecay} we establish exponential decay of correlations for the accelerated version of the flow, extending some arguments presented in \cite{avilaGouezelyoccoz} to our non-compact domains, and prove the main result of the paper (Theorem~\ref{thm:main}). In Section \ref{sec:geodesic} we verify that the geodesic flow on the modular surface can be modeled by a suspension flow in the family we are considering, hence recover an alternative dynamical proof  for its exponential mixing rate with respect to Liouville measure. Finally, in Appendices A and B we provide the proofs of a couple of technical results that we need in the previous sections.

\section{Statement of the main results} \label{sec:family}

 Our main reslts concern a family of hyperbolic flows on non-compact phase spaces. We first describe the one-dimensional maps, Poincar\'e maps and suspension flows (cf. Subsections~\ref{sec:setting}-~\ref{sec:susprf}), postponing the statement of the main results to Subsection~\ref{subsec:statements}. 

\subsection{One-dimensional maps} \label{sec:setting} 
Inspired by \cite{bonanno_delvigna_isola}, let us consider a map $f_0: (0,1) \rightarrow \R^+$ that is $C^2$-smooth and satisfies the following set of assumptions:
\begin{multicols}{2}
\begin{itemize}
\item[\textbf{(A1)}] $\lim\limits_{x \to 0^+} f_0 (x) = 0$, \\ and $\lim\limits_{x \to 1^-} f_0(x) = +\infty$.
\item[\textbf{(A2)}] $f_0'(x) > 1$, for any $ x \in (0,1)$.
\item[\textbf{(A3)}] $\lim\limits_{x \to 0^+} f_0'(x)=1$.
\end{itemize}
\columnbreak
\begin{itemize}
\item[\textbf{(A4)}] $f_0''(x) > 0$, for any $x \in (0,1)$.
\item[\textbf{(A5)}] 
$\dfrac{f_0''(x)}{f_0'(x)^2}$ is decreasing in $x$.
\item[\textbf{(A6)}] $C_A = \sup\limits_{x \in (0,1)} \dfrac{f_0''(x)}{[f_0'(x)]^2} < +\infty$.
\end{itemize}
\end{multicols}

Let us comment on the assumptions. The map $f_0$ is strictly increasing, expanding, convex and unbounded (as a consequence of (A1), (A2) and (A4)). However, it combines the neutral behavior of a map with an indifferent fixed point (as a consequence of (A3))  
and a singularity (it has a singularity at $x=1$ as assumption (A1) implies that $\lim\limits_{x \to 1^-} f_0'(x) = +\infty$). Assumption \textbf{(A5)}, satisfied for geodesic flow, is of technical nature and it is natural condition that is useful to verify the UNI condition (cf. proof of Proposition~\ref{thm:rstandard2}).
Finally assumption \textbf{(A6)}, often referred to as \textit{Adler's property}, implies the \textit{bounded distortion property}, i.e. given $x,y \in (0,1)$ with $x\neq y$, we have:
\begin{align}
\left| \ln \dfrac{f_0'(x)}{f_0'(y)}\right| \le C_A |f_0(x)-f_0(y)|. \label{boundeddistortionf0}
\end{align}


\begin{figure}[htb]
    \centering
\includegraphics[width=0.44\linewidth]{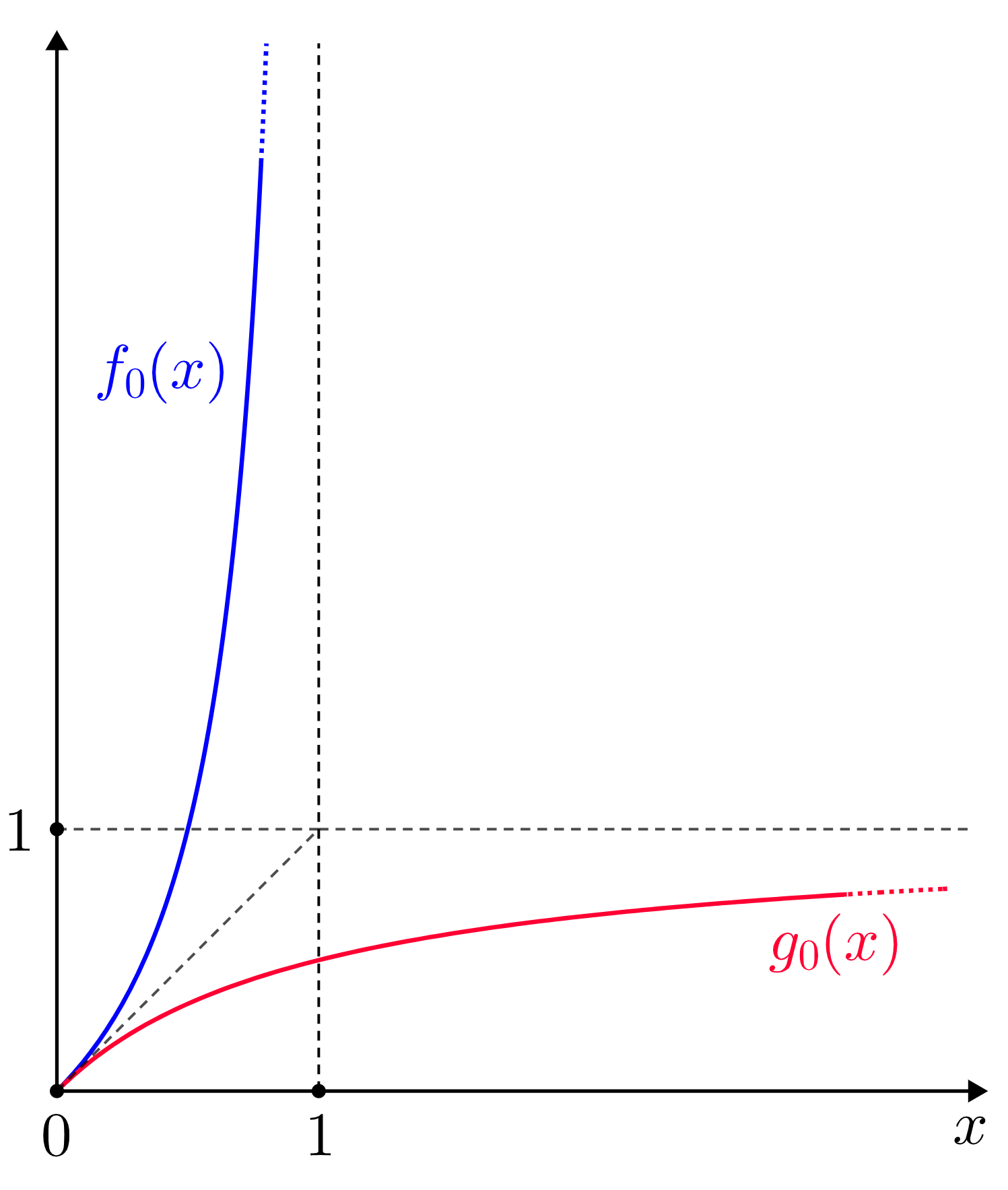}
    \caption{An illustration of the maps $f_0$ and its inverse $g_0$.}
    \label{fig:map_f0_g0}
\end{figure}

Since the map $f_0$ is a bijection, we denote by $g_0: \R_+ \rightarrow (0,1)$ its inverse. 
Note that $g_0$ is increasing and contracting, since $g_0'(x) \in (0,1)$ for every $x \in \R$. However, the contraction does not have a uniform rate, because $\lim\limits_{x \to 0^+} g_0'(x) = 1$. We assume the following properties on $g_0$:

\textbf{(B)} There exist monotone decreasing sequences $\left\{ \omega_n^{(i)} \right\}_{n \ge 1}$, for $i \in \lbrace 1,2 \rbrace$ on the interval $(0,1)$ converging to $0$, 
and constants $C_I^{(1)},C_I^{(2)}>0$ and
$\sigma_1,\sigma_2 \in (0,1)$ such that the following properties hold:
\begin{itemize}[leftmargin=1.5cm]
\item[(i)] $\sum\limits_{n =1}^{+\infty} \left[ \omega_n^{(i)} \right]^{1-\sigma_i} < +\infty$;
\item[(ii)] $g_0'(n-1) \le C_I^{(1)} \cdot \omega_n^{(1)}$ for every $n\ge 2$;
\item[(iii)] $\prod\limits_{j = 1}^{n-1} g_0'(g_0^j(1)) = \dfrac{(g_0^n)'(1)}{g_0'(1)} \le C_I^{(2)} \cdot \omega_n^{(2)}$ for every $n\ge 1$;
\end{itemize}
These assumptions, which appear naturally in the context of the geodesic flow on the modular surface (cf. Section~\ref{sec:geodesic}), could be formulated in terms of the derivatives of the map $f_0$ along a sequence of points that tend to $0$, where the derivative $f_0'$ presents loses hyperbolicity.
We will always assume functions $f_0$ and $g_0$ satisfy the previous set of assumptions and use the notation
$\prod\limits_{l=n}^m x_l = 1$ for any sequence $\lbrace x_l \rbrace_{l \ge 1}$ and $n>m$.

\subsection{Poincar\'e maps}
Let us introduce the dynamical system that will be used to construct the suspension flows, hence will be the Poincar\'e first return map to the base. Consider the \textit{base map} $\PPP$ defined on the non-compact and unbounded metric space $X=\R^+ \times \R^+$ by
\begin{align}
\label{eq:defbase}
\PPP (x,y) := \begin{cases}
\left( f_0(x),g_0(y)\right), & \text{if}\ x<1; \\[5pt]
(x-1,y+1), & \text{if} \ x>1; \\[5pt]
(0, +\infty ), & \text{if}\ x=1
\end{cases}
\end{align}
for any $(x,y) \in \R^+ \times \R^+$. This is not a product map. In fact, writting $f_1(x) = x-1$ and $g_1(x) = x+1$ for every $x \in \R$, and defining  $f:\R^+ \rightarrow \R^+$ and $g_x:\R^+ \rightarrow \R^+$ as
\begin{align*}
f(x) = \begin{cases}
f_0(x), & \text{if}\ x<1; \\[5pt]
f_1(x), & \text{if} \ x>1; \\[5pt]
0, & \text{if}\ x=1.
\end{cases} \ \ \ \ \ \ \ \ \ \ g_x(y) = \begin{cases}
g_0(y), & \text{if}\ x<1; \\[5pt]
g_1(y), & \text{if} \ x>1; \\[5pt]
+\infty, & \text{if}\ x=1.
\end{cases}
\end{align*}
for any $(x,y) \in \R^+ \times \R^+$, one can observe that $\PPP$ is the skew-product 
\begin{equation}
    \label{eq:Passkew}
    \PPP (x,y) = (f(x),g_x(y)).
\end{equation}
We further assume that:
\medskip

\textbf{(C)} $\PPP$ preserves an ergodic measure  $m \ll Leb$ (where $Leb$ stands for the Lebesgue measure) and that, for every $a>0$ 
$$m ((a,b) \times \R^+ ) <+\infty.
$$ 

\medskip
\begin{remark}
Observe that the measure $m$ need not be finite. Moreover, since $m \ll Leb$, one can neglect the point $x=1$ in the definitions of the maps above, hence  focusing on the sets $(0,1)$ and $(1,+\infty)$.    
\end{remark}

\subsection{Roof functions and suspension flows}
\label{sec:susprf}
Let us consider strictly positive \textit{roof functions}  $\rho$ which are defined for every $(x,y) \in \left( \R^+ \setminus \lbrace 1 \rbrace \right) \times \R^+$ by:
\begin{align}
\label{eq:defroof functiongen}
\rho (x,y) := \rho_0 \ln \left[ \dfrac{x}{y} \cdot \dfrac{g_x(y)}{f(x)}\cdot \dfrac{f'(x)}{\partial_y g_x(y)}\right] = \begin{cases}
\rho_0\ln \left[ \dfrac{x}{y}\cdot \dfrac{g_0(y)}{f_0(x)} \cdot \dfrac{f_0'(x)}{g_0'(y)} \right], & \text{if}\ x<1; \\
\\
\rho_0 \ln \left[ \dfrac{x}{y}\cdot \dfrac{g_1(y)}{f_1(x)} \cdot \dfrac{f_1'(x)}{g_1'(y)} \right], & \text{if} \ x>1,
\end{cases}
\end{align}
for some $\rho_0 >0$, and satisfying the integrability condition 
\textline[t]{\textbf{(D)}}{
$\int_{\R^+ \times \R^+} \rho \, dm < +\infty
$.}{}

\medskip
Now, we proceed to define the suspension flow over the basis $( (\R^+ \setminus\{1\}) \times \R^+ , \PPP , m)$ with roof function $\rho$. Its phase space is\
\begin{align}
\Sigma_\rho = \left\{ [(x,y),s ] \in \left( \left( \R^+\setminus \lbrace 1 \rbrace \right) \times \R^+ \right) \times \R \, | \,0 \le s \le \rho (x,y) \right\} / \sim, \label{phasespceflow}
\end{align}
where the equivalence relation $\sim$ identifies the points $[(x,y), \rho (x,y)]$ and $[\PPP (x,y) , 0]$.
The semiflow $\varphi^t : \Sigma_\rho \rightarrow \Sigma_\rho$ is  defined in local coordinates as
\begin{align}
\varphi^t [(x,y), s ] =  [(x,y), s +t], \label{suspensionflowdef}
\end{align}
for every $t \in [0,+\infty)$, taking into account the identifications. In other words,
\begin{align}
\label{eq:suspension-t}
\varphi^t [(x,y),s] = \left( \PPP^n (x,y) , s+t-\sum\limits_{i=0}^{n-1}\rho \circ \PPP^i (x,y) \right),
\end{align}
where $n\ge 0$ is the unique integer such that
\begin{align*}
\sum\limits_{i=0}^{n-1}\rho \circ \PPP^i (x,y) \le t+s < \sum\limits_{i=0}^{n}\rho \circ \PPP^i (x,y).
\end{align*}
Since the map $\PPP$ is invertible then $\varphi^t$ 
defines a flow.
Observe that the flow $\varphi^t$ preserves the product measure $m \otimes Leb$ on $\Sigma_\rho$ given by 
$$
\int \psi \; d(m \otimes Leb) =  
\int\limits_{\R^+\times \R^+} \left[ \int_0^{\rho (x,y)} \, \psi((x,y),t) \, dt \right] \, dm (x,y) 
$$
for every continuous observable $\psi: \mathbb R^+ \times \mathbb R^+ \to \mathbb R$ with compact support.
As $\rho$ is $m$-integrable then the  measure $m \otimes Leb$ is finite and,
in particular, $\varphi^t$ preserves the probability measure $m_\rho = \dfrac{m \otimes Leb}{m \otimes Leb (\Sigma_\rho)}$.\\

\subsection{Statements} \label{subsec:statements}
In order to state our main results, let us introduce the class of observables that we will consider.
Given a smooth Riemannian manifold $M$, let us denote by $C_b^1(M)$ the space of $C^1$-smooth and bounded functions $u : M \rightarrow \R$ such that $\sup_{p \in M} \| Du(p) \| < +\infty$, endowed with the norm
\begin{align}
\label{normfunction}
\| u \|_{C_b^1 (M )}  = \sup\limits_{p \in M} |u(p)| + \sup\limits_{p \in M} \| Du(p ) \|. 
\end{align}
In Section \ref{subsec:exponentialDOCoriginalflow} we will define a map $\Pi: \Sigma_r \rightarrow \Sigma_\rho$ which induces a conjugacy between a suspension flow $\Pt_t$, defined on its phase space $\Sigma_r$, and the flow $\varphi^t$ defined on $\Sigma_\rho$ in the previous section. $\Pt_t$ is obtained through an inducing scheme, and, in some sense, is an \textit{accelerated} counterpart of the original flow $\varphi^t$, defined over a base space $\Sigma \subset \left( g_0(1) , 1 \right) \times \R^+$, and with a roof function $r$ (see the discussion below and Section \ref{sec:exponentialDOCinducedflow} for more details).

We denote by $C_{b,*}^1(\Sigma_r )$ the space of observables in 
$C_b^1(\Sigma_r )$ which have uniformly continuous derivatives in the base component $(x,y) \in \Sigma$. Then, the functional space for the observables that we consider for the decay of correlations of the flow $\varphi^t$ on $\Sigma_\rho$ will be the push-forward through $\Pi$ of $C^1_{b,*} ( \Sigma_r)$, that is,
\begin{align}
\CCC_* = \Pi_* \left( C^1_{b,*} ( \Sigma_r) \right) = \bigg\{ u : \Sigma_\rho \rightarrow \R \  | \  u \circ \Pi \in C^1_{b,*} (\Sigma_r) \bigg\}, \label{def:classofobservables}
\end{align}
equipped with the norm $\norm{\cdot}_*$ defined, for every $u \in \CCC_*$, as:
\begin{align}
\norm{u}_* = \norm{u \circ \Pi}_{C_b^1(\Sigma_r)}. \label{def:normofobservables}
\end{align}

We are now in a position to state our main result concerning the decay of correlations for the class of suspension flows described above.

\begin{theorem}
\label{thm:main}
Let $\varphi^t$  be a suspension flow over the base $\bigg(\big( \R^+\setminus \lbrace 1 \rbrace \big) \times \R^+ , \PPP , m\bigg)$ with roof function $\rho$, satisfying  assumptions \textbf{(A1)} - \textbf{(A6)}, \textbf{(B)}, \textbf{(C)}, and \textbf{(D)}. Then, there exist constants $C, \delta\in \R^+$ such that, for all $u,v \in \CCC_*$, and for all $ t \in [0,+\infty )$, we have that
\begin{align*}
\left| \,\int\limits_{\Sigma_\rho} u \cdot v \circ \varphi^t \, dm_\rho -  \int\limits_{\Sigma_\rho} u \, dm_\rho \cdot \int\limits_{\Sigma_\rho} v \, dm_\rho  \right| \le C \cdot \norm{u}_*  \norm{v}_* \cdot e^{-\delta t}.
\end{align*}
\end{theorem}

\begin{remark}
For an explicit example of a subclass of observables included in $\CCC_*$, we may consider the space of compactly supported functions in $C_b^1(\Sigma_\rho)$ whose support does not intersect the set of points of the form $[(x,y),\,\rho(x,y)]$ (i.e. those points for which we have the identification between $[(x,y),\,\rho(x,y)]$ and $[\PPP(x,y),0]$). In particular, we have that the conjugacy $\Pi: \Sigma_r \rightarrow \Sigma_\rho$ is piecewise $C^1$, hence it breaks the $C^1-$regularity when crossing the singularities, which are represented by the vertical boundaries 
\begin{align*}
\left\{ [(x,y),s] \in \bigg( \big( \R^+\setminus \lbrace 1 \rbrace \big) \times \R^+ \bigg) \times \R \ \bigg| \ s=0 \ \text{or} \ s=\rho(x,y) \right\},
\end{align*}
in the lift of the phase space $\Sigma_\rho$ of the flow $\varphi^t$ to $\bigg( \big( \R^+\setminus \lbrace 1 \rbrace \big) \times \R^+ \bigg) \times \R$.
\end{remark}

\color{black}
Some comments are in order. The previous theorem cannot be obtained directly from the criteria in \cite{araujomelbourne, avilaGouezelyoccoz,baladivallee} due to the possible lack of uniform hyperbolic behavior of the Poincar\'e map coming from the unbounded domains. Our strategy here is to proceed with a triple inducing process. The first one consists of reducing the Poincar\'e section and to keep track of the relation between the roof functions of both suspension flows. The resulting Poincar\'e map is still not uniformly hyperbolic, as it presents points whose derivative is not a hyperbolic matrix. This requires an extra inducing process, keeping the cross-section but taking the second return time to it, in a way that the Poincar\'e map becomes hyperbolic. Such triple inducing scheme pushes the complexity from the dynamics to the roof function, which is written as a double Birkhoff sum over the original roof function.

Now we note that Theorem~\ref{thm:main} can be used to provide a purely dynamical proof  that the geodesic flow on the modular surface has exponential decay of correlations with respect to its physical measure. 
Bonanno, Del Vigna and Isola in \cite{bonanno_delvigna_isola} give an isomorphic representation of the geodesic flow in terms of a suspension flow where the base map $\PPP$ is a skew-product 
$\PPP(x,y) = (f(x),g_x(y))$
on $((\R^+ \smallsetminus \lbrace 1 \rbrace) \times \R^+$, where
\begin{align}
\label{eq:deffuncgeo}
f(x) = \begin{cases}  \dfrac{x}{1-x}, & \text{if} \ x<1;\\[8pt]   x-1, & \text{if} \ x>1;\end{cases} \ \ \ \ \ \ \ \ \ g_x(y) = \begin{cases}   \dfrac{y}{1+y}, & \text{if} \ x<1,\\[8pt]   y+1, & \text{if} \ x>1;\end{cases}
\end{align}
and the roof function is given by
\begin{equation}
    \label{eq:roof functiongeo}
    \rho (x,y) = 
    \begin{cases} \dfrac{1}{2}\ln \left[ \dfrac{1+y}{1-x}\right], & \text{if} \ x<1; \\[16pt] \dfrac{1}{2}\ln \left[ \dfrac{1+\dfrac{1}{y}}{1-\dfrac{1}{x}}\right], & \text{if} \ x>1. 
    \end{cases}
\end{equation}

Call $\mathcal{G}$ the phase space of the geodesic flow on the modular surface, and $\Sigma^G_\rho$ the phase space of its suspension flow model. The latter is obtained from the original \textit{orbifold} model on $\mathcal{G}$ through a $C^{\infty}-$smooth change of variables $\Gamma: \mathcal{G} \rightarrow \Sigma^G_\rho$ (cf. \cite{bonanno_delvigna_isola}, Appendix B). We define $\widetilde{\CCC}_*$ to be the pull-back via $\Gamma$ of the functional space $\CCC_*$ constructed as in \eqref{def:classofobservables}, i.e.
\begin{align}
\widetilde{\CCC}_* = \Gamma^* \left( \CCC_* \right)  = \left\{ \widetilde{u} : \mathcal{G} \rightarrow \R \, | \, \widetilde{u} \circ \Gamma^{-1} \in \CCC_* \right\} = \left\{ u \circ \Gamma \, | \, u \in \CCC_* \right\}. \label{def:classofobservables2}
\end{align}

In Section~\ref{sec:geodesic} we show that the geodesic flow on the modular surface fits in the framework of Theorem~\ref{thm:main}, and hence we provide a dynamical proof  for the following:

\begin{theorem}
\label{thm:geodesic}
    The geodesic flow on the modular surface has exponential decay of correlations with respect to the Liouville measure and observables in $\widetilde{\CCC}_*$.
\end{theorem}

\begin{remark}
The space $\widetilde{\CCC}_*$ contains the compactly supported functions in $C_b^1(\mathcal{G})$ whose support does not intersect the cross-section which induces the isomorphism $\Gamma$ between the geodesic flow on the modular surface and its suspension flow representation.
\end{remark}

\section{A criterion for exponential decay of correlations} \label{sec:standard}

In this section we recall a general framework from \cite[Section~2]{avilaGouezelyoccoz} in which exponential decay of correlations for SRB measures can be deduced for suspension flows over piecewise hyperbolic skew-products whose roof functions satisfy a set of good properties, to be described below. Whenever we refer to this context, we will call it the standard setting and the standard assumptions.

Let $\Delta$ be an open subset of a Riemannian manifold, with compact closure and whose boundary is a finite union of smooth hypersurfaces.
Let $\{\Delta^{(l)}\}_\ell$ define an open, measurable and at most countable partition of $\Delta$ so that 
$\bigcup\limits_{l} \Delta^{(l)}$ 
is a full $m$-measure subset of $\Delta$.
Assume that 
$$
T: \bigcup\limits_{l} \Delta^{(l)} \to \Delta
$$ 
is such that: (i) $T\mid_{\Delta^{(l)}}: \Delta^{(l)} \to \Delta$ a $C^1$ diffeomorphism each $l$ (i.e. T is \textit{full-branch}); 
(ii) $T$  is a uniformly expanding Gibbs-Markov map (in particular, it satisfies a generalization of Adler's property in higher dimension on each $\Delta^{(l)}$, uniformly with respect to $l$). 
Such a map preserves a unique $T$-invariant probability measure $\mu \ll Leb$, which is mixing and has a $C^1$ density bounded from above and from below (see e.g. \cite{aaronson, alves, avilaGouezelyoccoz}).

Following \cite{avilaGouezelyoccoz}, we say that a skew-product $\widehat{T}: \widehat{\Delta} \rightarrow \widehat{\Delta}$ 
is a $C^{2}$-\emph{piecewise hyperbolic skew-product} over $T$ if there exists a continuous surjective map $\pi : \widehat{\Delta} \rightarrow \Delta$ such that $T \circ \pi = \pi \circ \widehat{T}$ and the following properties hold:
\begin{enumerate}
\item[(i)] there exists a $\widehat{T}$-invariant probability measure $\nu$ on $\widehat{\Delta}$ and $\pi_*\nu=\mu$;
\item[(ii)] ($C_b^1$-regular disintegration) 
there exists a measurable function
$\Delta\ni x\mapsto \nu_x$ so that $\nu_x$ is a probability on $\widehat{\Delta}$ supported on $\pi^{-1}(x)$ and, for any measurable set $A \subset \widehat{\Delta}$,
\begin{align*}
\nu (A) = \int\limits_{\Delta} \nu_x (A) \, d\mu(x);
\end{align*}
moreover, the disintegration is $C_b^1$-regular, meaning that there exists $C'>0$ such that, for any open set $O\subset \bigcup\limits_l \Delta^{(l)}$, and any $u \in C^1_b(\pi^{-1}(O))$ with uniformly continuous derivatives, 
 the function 
$$
\overline{u}: O \rightarrow \R
\quad \text{given by} \quad \overline{u}(x) = \int\limits_{\widehat{\Delta}} u \, d\nu_x
$$ 
belongs to $C_b^1(O)$ and satisfies the inequality
\begin{align*}
\sup\limits_{x \in O} \| D\overline{u} (x) \| \le C' \!\!\!\! \sup\limits_{z \in \pi^{-1}(O)} \| Du (z) \|;
\end{align*}
\item[(iii)] ($\widehat{T}$ contracts the $\pi$-fibres)  there exists $k \in (0,1)$ so that $d(\widehat{T}(z_1) , \widehat{T}(z_2)) \le k d(z_1,z_2)$ for all $z_1,z_2 \in \widehat{\Delta}$ with $\pi(z_1)=\pi(z_2)$, where $d$ is the metric induced on $\widehat{\Delta}$ by the Riemannian structure.
\end{enumerate}

\medskip
We now proceed to specify the set of properties that the roof function is required to satisfy in \cite{avilaGouezelyoccoz}.
Denote by $\mathcal{H}$ the set of inverse branches of $T$. Assume that the roof function $r: \bigcup\limits_l \Delta^{(l)} \rightarrow \R^+$ satisfies:
\begin{itemize}
\item[(i)] $\inf\limits_{x \in \bigcup\limits_l \Delta^{(l)}} r(x) > 0$;
\item[(ii)] there exists $C''>0$ such that, for every $h \in \mathcal{H}$, we have: $\|D(r \circ h )\|_{C^0} \le C''$;
\item[(iii)] (\textit{Uniform Non-Integrability}) it is not possible to write $r= \psi +\phi \circ T-\phi$ on $\bigcup\limits_l \Delta^{(l)}$, where $\psi: \Delta \rightarrow \R$ is constant on each set $\Delta^{(l)}$ and $\phi: \Delta \rightarrow \R$ is measurable;
\item[(iv)] (\textit{exponential tails}) there exists $\sigma>0$ such that $\int\limits_{\Delta} e^{\sigma r} \, dLeb < +\infty$. 
\end{itemize}

Under the standard assumptions, Avila, Gou\"ezel and Yoccoz \cite{avilaGouezelyoccoz} proved that the suspension flow built over the basis $\widehat{T}$ with roof function $r \circ \pi$ has exponential decay of correlations for $C_b^1(\widehat{\Delta})$ observables. 
More precisely:

\begin{theorem}
\cite{avilaGouezelyoccoz}
\label{thm:AGY}
Let $\varphi^t$ be a suspension flow over a $C^2$-piecewise hyperbolic skew-product and a roof function $r$ satisfying the standard assumptions.
There exist constants $C, \delta>0$ such that, for all $U,V \in C_b^1(\Sigma_r )$ with uniformly continuous derivatives in the component $(x,y) \in \Sigma$, and for all $ t \in [0,+\infty )$, we have that
\begin{align*}
\left| \,\int\limits_{\Sigma_r} U \cdot V \circ \varphi^t \, d\nu_r - \int\limits_{\Sigma_r} U \, d\nu_r \cdot \int\limits_{\Sigma_r} U \, d\nu_r  \right| \le C \cdot \| U \|_{C_b^1(\Sigma_r )} \| V \|_{C_b^1 (\Sigma_r )} \cdot e^{-\delta t}.
\end{align*}
\end{theorem}

\begin{remark}
    One can not use the directly the previous approach to prove exponential decay of correlations 
    for the family of suspension flows, described in the previous Section \ref{sec:family}. In fact, the latter does not fit into the standard conditions listed above, as:
\begin{itemize}
\item[(a)] the phase space for the base is an unbounded subset of $\mathbb R^2$ (both in the first and the second component);
\item[(b)] the measure $m$ preserved by the skew product may be infinite;
\item[(c)] the map $f$ is not uniformly expanding, and the maps $g_x$ are not uniformly contracting;
\item[(d)] the roof function $\rho$ may not be bounded away from zero on $\R^+\times \R^+$.
\item[(e)] over the set $(1,+\infty)\times \R^+$, the roof function is cohomologous to $0$: if $(x,y) \in (1,+\infty) \times \R^+$ then
\begin{align*}
u(x,y) = \rho_0 \ln \left(\dfrac{y}{x} \right),
\end{align*}
and
\begin{align*}
\rho(x,y) = \rho_0\ln \left[ \dfrac{x}{y}\cdot \dfrac{g_1(y)}{f_1(x)}\cdot \dfrac{f_1'(x)}{g_1'(y)}\right] = \rho_0\ln \left(\dfrac{g_1(y)}{f_1(x)} \right) - \rho_0\ln \left( \dfrac{y}{x} \right) = u(\PPP(x,y))-u(x,y).
\end{align*}
\end{itemize}

\end{remark}

Despite all these obstructions, we outline a strategy that eludes them and allows us to use the full power of the techniques described in \cite{avilaGouezelyoccoz}. In view of the previous discussion, our approach here is to consider a suitable inducing scheme that will allow one to observe the original suspension flow as a suspension flows (on a reduced cross-section and second Poincar\'e return time) that verifies the standard assumptions of Avila, Gou\"ezel and Yoccoz.


\section{An inducing scheme for the base transformation} \label{sec:base}
We are interested in the asymptotical statistical mixing properties of the suspension flow $\varphi^t$. In order to explore them, we induce the system on a subset of the phase space and redirect our investigation to the behaviour of a suitable induced map.

The idea is to induce on a set of finite measure so that we can work within a probability space framework. We then need to verify whether the original assumptions \textbf{(A)}, \textbf{(B)}, \textbf{(C)}, and \textbf{(D)} on $f_0,g_0$, and $\rho$ from Section \ref{sec:family} carry over to the induced map, ensuring that it satisfies the standard conditions. Let us take advantage of the fact that $\mathcal{P}$ is a skew-product (recall ~\eqref{eq:Passkew}) in order to initially study the first ({unstable}) component.

\subsection{Inducing scheme for $f$} \label{sec:firstcomponent}
Let $\pi: \R^2 \rightarrow \R$ be the canonical projection on the first component. Since $f\circ \pi = \pi \circ \PPP$ and $m$ is $\PPP$-invariant, one knows that $f$ preserves the measure $\pi_* m$.
Note that the dynamics of $f$ combines a branch with complex behavior in the interval $(0,1)$ (with the presence of indifferent behavior combined with a singularity at $x=1$)  and a branch defined on $(1,+\infty)$ (which is responsible for the recurrence, by bringing the orbits of points back to $(0,1)$). In this way, we induce the map $f$ by its first Poincar\'e return time to the interval $(0,1)$. For every $s \ge 1$, set 
\begin{align}
I_s =  (a_s , b_s ) = \left(g_0(s-1) , g_0(s) \right), \label{Is}
\end{align}
and observe that $I_s = \left( g_0 \circ g_1^{s-2}(1), g_0 \circ g_1^{s-1} ( 1) \right)$. 

Note that $a_{s+1}=b_s$ for every $s\ge 1$. 
Hence, 
$\left\{ {I_s} \right\}_{s \ge 1}$ forms a partition of $(0,1)$ up to a zero Lebesgue measure subset.
Moreover, since $f_0$ is monotonically increasing, $f_0 \left(I_s\right)=(s-1,s)$, hence  in particular $f_0 (a_s) = s-1$ and $f_0 (b_s) = s$.
Now let $\tau : (0,1) \rightarrow \mathbb{N}$ be the first return time map of $f$ to $(0,1)$ defined as
\begin{align}
\tau (x) := \inf \left\{n \ge 1 \, | \, f^n (x) \in (0,1) \right\}.\label{firstreturn1}
\end{align}

\begin{figure}[htb]
    \centering
    \begin{minipage}{0.45\textwidth}
        \centering
        \vspace{-1.3cm}
        \includegraphics[width=1.2\linewidth]{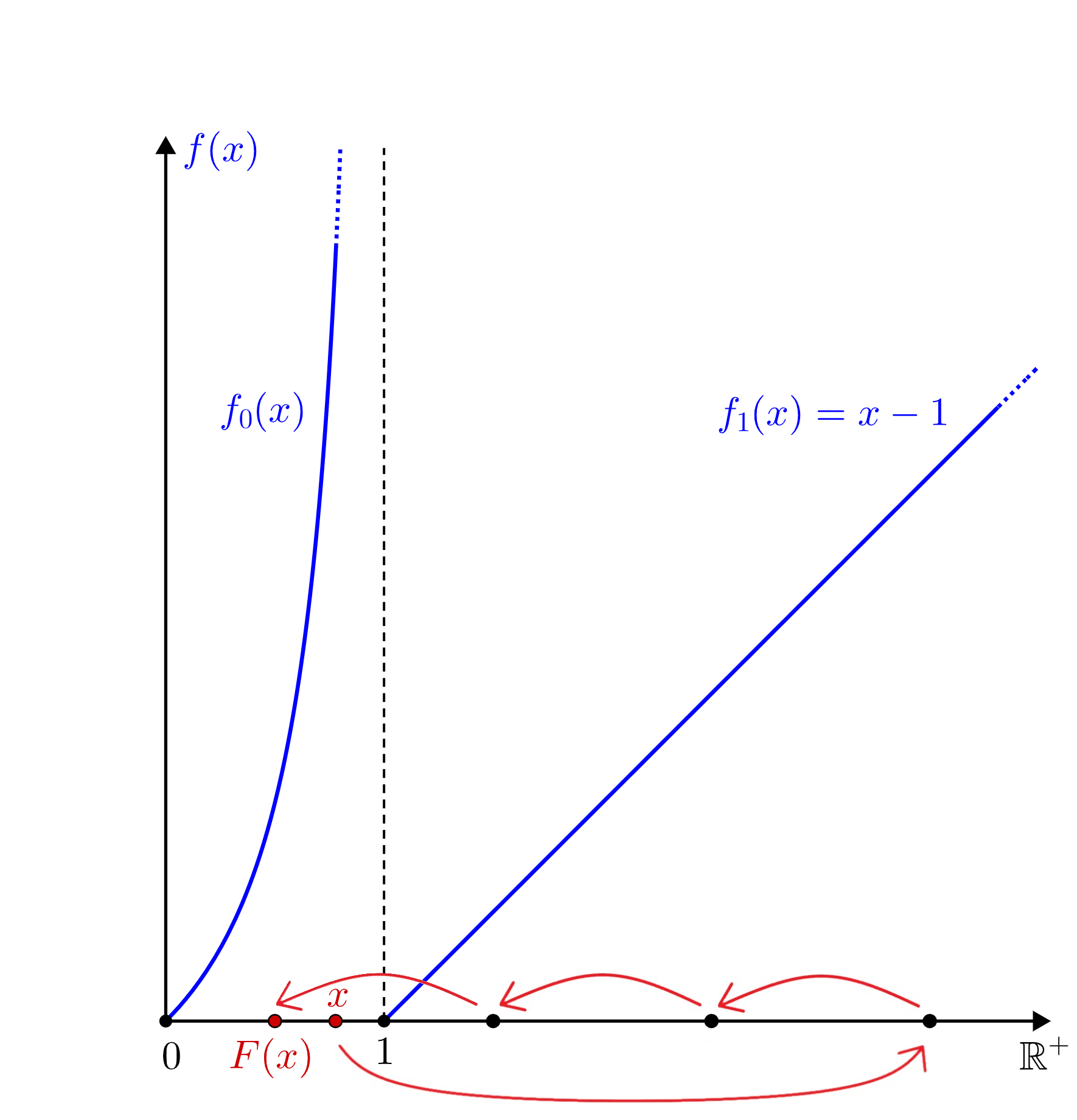}
    \end{minipage}\hfill
    \begin{minipage}{0.45\textwidth}
        \centering
        \includegraphics[width=1.1\linewidth]{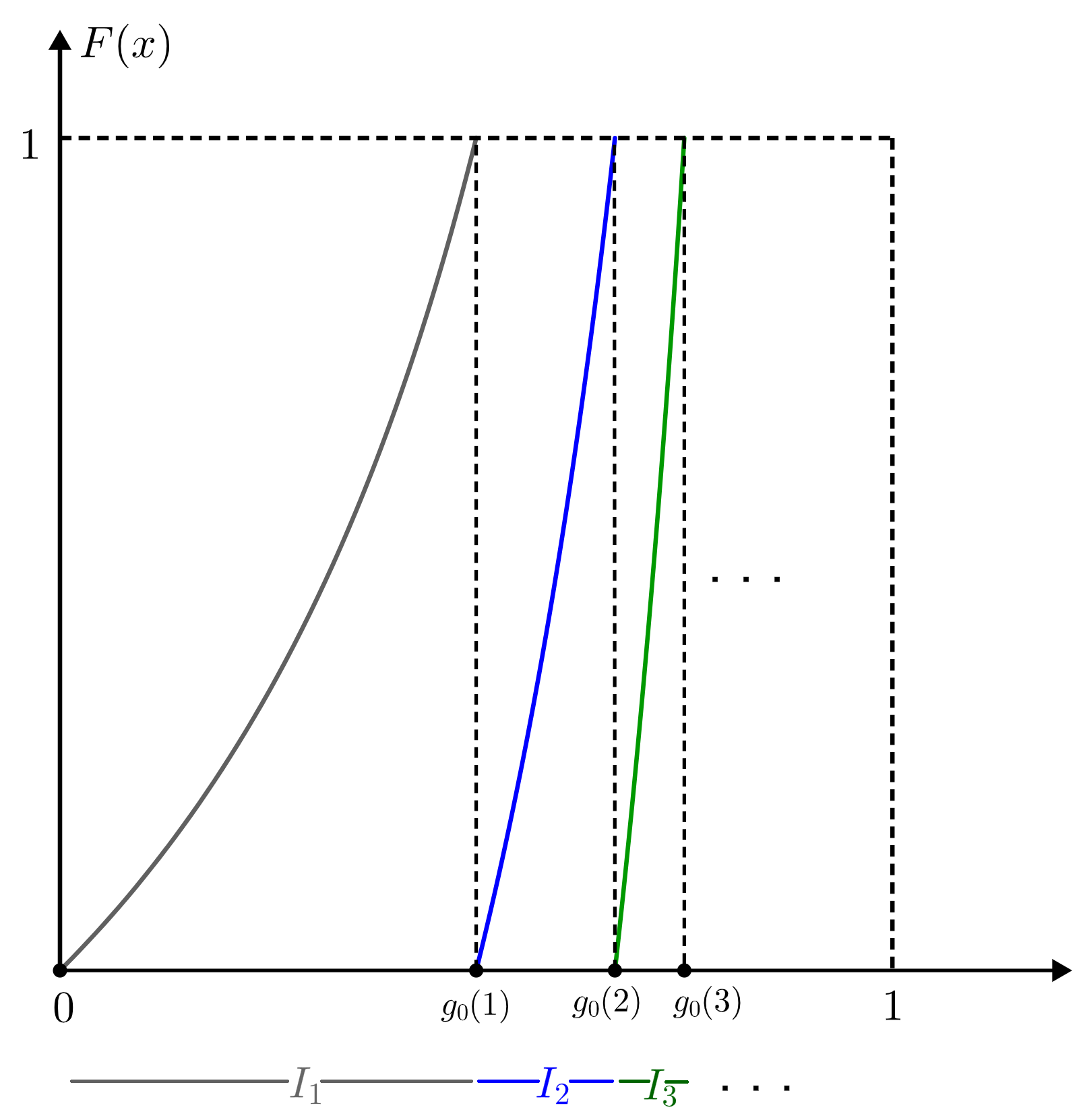}
    \end{minipage}
    \caption{Construction of the induced map $F$ and intervals $I_s$.}
    \label{fig:firstinducing}
\end{figure}

\begin{lemma}
\label{lem:tau}
For every $s \ge 1$ we have that $\tau|_{I_s} \equiv s$. In particular $\tau$ is finite on $\bigcup\limits_{s \ge 1} I_s$.
\end{lemma}

\begin{proof}

Take $s \ge 1$ and observe that
\begin{align*}
f^k(f(I_s))= f^k (f_0 (I_s)) = f^k ((s-1,s))=f_1^k((s-1,s))=(s-1-k,s-k)
\end{align*}
for any $0 \le k \le s-1$. This implies that $f^k(f (I_s)) = (0,1)$ if and only if $k=s-1$.\\
Therefore, $f^s (I_s)=f_1^{s-1}(f_0 (I_s))=(0,1)$, so that $\tau|_{I_s} \equiv s$. This proves teh lemma.
\end{proof}

In view of the previous lemma we can take the induced map
\begin{align}
F= f(\cdot)^{\tau(\cdot )} : \bigcup\limits_{s \ge 1}I_s \rightarrow (0,1). \label{inducedmap1}
\end{align}
By construction and Lemma \ref{lem:tau}
one knows that 
$F$ is a piecewise $C^2$ full-branch piecewise map on the measurable partition $\lbrace I_s \rbrace_{s \ge 1}$, and $F|_{I_s}$ is a strictly increasing $C^2$ function with $C^2$ extension to the closure $\overline{I_s}$, and 
\begin{align}
F(x) = f_1^{s-1}\circ f_0 (x) = f_0(x) - (s-1), \quad \text{for every $x\in {I_s}$}. \label{expressionforF}
\end{align}
Hence  $F'(x) = f_0'(x)$ for every $x\in \bigcup\limits_{s \ge 1}I_s$.

\medskip
The induced map $F$ preserves the measure $\pi_*m\mid_{(0,1)}$ which may be infinite. In order to obtain a Gibbs-Markov map over a finite measure set we perform a second inducing scheme. We start by noting that,
since $\bigcup\limits_{s \ge 2} I_s$ has infimum $a_2=b_1=g_0(1)>0$, assumption 
{\bf (C)} ensures that
\begin{align*}
\pi_*m \big( \bigcup\limits_{s \ge 2} I_s \big) =  \pi_*m ((g_0(1),1)) = m ((g_0(1),1) \times \R^+)<+\infty.
\end{align*}
We proceed to consider the second induced map of $f$, obtained as the first return time map $\hat F$ of $F$ to the interval $\bigcup\limits_{s \ge 2} I_s$. For our convenience, let us first define the following objects, for every $s,q \ge 1$, with $s \ge 2$:
\begin{itemize}
\item $J^q =  \left[ F^{-1} \left( g_0^q (1) , g_0^{q-1}(1) \right)\right] \cap (g_0(1),1)$.
\item $J_s^q = J^q \cap I_s = \left[ F|_{I_s} \right]^{-1} \left( g_0^q (1) , g_0^{q-1}(1) \right)$.
\item $\Delta = \bigcup\limits_{q \ge 1} J^q = \bigcup\limits_{\substack{s \ge 2 \\ q \ge 1}}J_s^q$.
\item $c_s^q = g_0 \circ g_1^{s-1}\circ g_0^q (1)$; and $d_s^q = g_0 \circ g_1^{s-1}\circ g_0^{q-1} (1)$.
\end{itemize}


Observe that, by the definitions, $J^q$ is the disjoint union of intervals $J_s^q$, for $s \ge 2$, each of which is respectively contained in $I_s$.

\begin{lemma}
\label{lemma:J}

Consider $s,q \ge 1$, $s \ge 2$. Then, $J_s^q = (c_s^q , d_s^q )$.
\end{lemma}
\begin{proof}

Since $F|_{I_s} = f_1^{s-1} \circ f_0$ is continuous and increasing, 
\begin{align*}
(c_s^q,d_s^q) & = \left[ F|_{I_s} \right]^{-1} \left( g_0^q (1) , g_0^{q-1}(1) \right) = g_0 \circ g_1^{s-1} \left( g_0^q (1) , g_0^{q-1}(1) \right) \\
&= \left(g_0 \circ g_1^{s-1}\circ g_0^q (1) , g_0 \circ g_1^{s-1}\circ g_0^{q-1} (1) \right)
\end{align*}
(recall that $f_0^{-1}=g_0$ and $f_1^{-1} = g_1$). This proves the lemma.
\end{proof}

Note that both $\left\{ {J^q} \right\}_{q \ge 1}$ and $\left\{ {J_s^q} \right\}_{\substack{q \ge 1\\ s \ge 2}}$ are partitions of the interval $\left( g_0(1), 1 \right)$ ($\Leb$ mod 0) and that,
for each fixed $s$, $\left\{ \overline{J_s^q} \right\}_{q \ge 1}$ is a partition of the interval $(a_s, b_s]$
($\Leb$ mod 0).
Now, define the first return time $\kappa : \left(g_0(1),1\right) \rightarrow \mathbb{N}$  of $F$ to $\left(g_0(1),1\right)$ as
\smallskip
\begin{align}
\kappa (x) := \inf \left\{n \ge 1 \, | \, F^n (x) \in \left(g_0(1),1\right) \right\}. \label{firstreturntime2}
\end{align}

\begin{lemma}
\label{lem:J}
For every $q \ge 1$, we have that $\kappa|_{J^q} \equiv q$.
\end{lemma}

\begin{proof}
Fix $q \ge 1$.
As $J^q = F^{-1} \left( g_0^q (1) , g_0^{q-1}(1) \right)\cap (g_0(1),1)$, the map $F$ maps $J^q$ to $\left( g_0^q (1) , g_0^{q-1} (1)\right)$.
If $q=1$, we are done, as $\kappa (J^1) = 1$, by definition. Otherwise, if $0\le j \le q-2$ then $\left( g_0^{q-j} (1) , g_0^{q-1-j}(1) \right) \subset I_1$, while $F$ coincides with $f_0$ (recall \eqref{expressionforF}), thus
\begin{align*}
F^j \left( g_0^{q} (1) , g_0^{q-1}(1) \right) &= f_0^j \left( g_0^{q} (1) , g_0^{q}(1) \right) = \left( f_0^j \circ g_0^{q} (1) , f_0^j \circ g_0^{q-1}(1) \right)=  \left( g_0^{q-j} (1) , g_0^{q-1-j}(1) \right).
\end{align*}
Note that 
$F^{q-2} \left( g_0^{q} (1) , g_0^{q-1}(1) \right)=\left( g_0^{2} (1) , g_0(1) \right)$ and, consequently,
$$
F \left( g_0^{2} (1) , g_0(1) \right) = f_0 \left( g_0^{2} (1) , g_0(1) \right) = \left( g_0 (1) , 1 \right).
$$
This proves that $F^q (J^q)=\left( g_0 (1) , 1 \right)$, as desired.
\end{proof}

The previous lemma ensures that $\kappa$ is well-defined and finite on $\Delta$.
%
Lemmas \ref{lem:tau} and \ref{lem:J} show that the return times $\tau$ and $\kappa$ are piecewise constant and that 
\begin{equation}
    \label{eq:collectiontauk}
    \tau(x) = s \quad \text{and} \quad \kappa(x)=q, \qquad \text{for every $x \in J_s^q$, $s\ge2$, $q\ge1$}
\end{equation}
Finally, let us define the induced map 
\begin{align}
\Fh = F(\cdot)^{\kappa(\cdot )} : \Delta \rightarrow \left(g_0(1),1\right). \label{inducedmap2}
\end{align}

\begin{figure}[htb]
    \centering
    \includegraphics[width=0.55\linewidth]{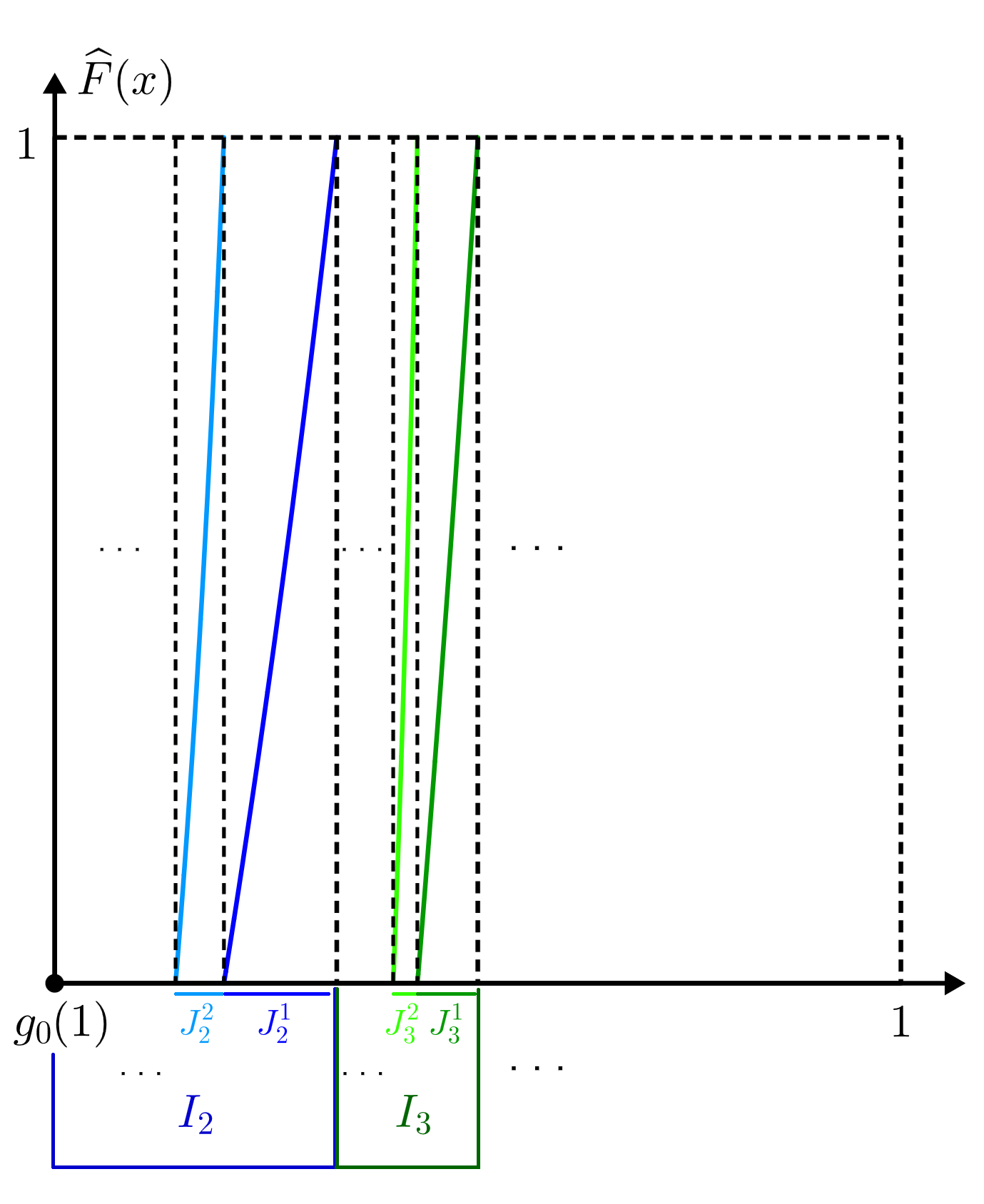}
    \caption{An illustration of the second induced map $\Fh$ and intervals $J_s^q$.}
    \label{fig:fhat}
\end{figure}

Recall that $J_s^q=(c_s^q,d_s^q)$
for each $s,q\ge 1$, with $s\ge 2$.
For any function $\varphi: \Delta \rightarrow \R$, it is convenient to adopt the notation that, for $c_s^q$ and $d_s^q$,
\begin{align*}
\varphi (c_s^q):=\lim\limits_{x \to (c_s^q)^+} \varphi(x), \ \ \ \ \ \ \ \ \varphi (d_s^q) := \lim\limits_{x \to (d_s^q)^-}\varphi(x),
\end{align*}
whenever the limits are well-defined.
From now on, if not specified, we intend $s$ and $q$ as positive natural numbers, with $s \ge 2$.
Furthermore, given $x \in J_s^q$, throughout the paper we will make large use of the following formulas, obtained by straightforward computations:
\begin{align}
&\widehat{F}(x) = F(x ) ^{\kappa (x )}|_{J_s^q}= F|_{I_1}^{q-1}\circ F|_{I_s}(x) = f_0^{q-1}\circ f_1^{s-1} \circ f_0(x). \label{expressionFhat} \\
&\Fh'(x) = \prod\limits_{l=0}^{q-1}f_0'(F^l(x)) = \left[ \prod\limits_{l=0}^{q-2} f_0' (f_0^l \circ f_1^{s-1} \circ f_0 (x)) \right] \cdot f_0'(x). \label{derivativeFhat}\\
&\Fh''(x) = \sum\limits_{j=0}^{q-1} \left[\prod\limits_{l=0}^{j-1} f_0'(F^l (x))^2 \right] \cdot f_0''(F^j(x)) \cdot \left[ \prod\limits_{l=j+1}^{q-1} f_0'(F^l(x)) \right]. \label{secondderivativeFhat}
\end{align}

While it is clear that $\Fh|_{J_s^q}$ and $\Fh'|_{J_s^q}$ are strictly increasing we need to check that the map $\Fh$ exhibit an expanding behavior.

\begin{proposition}
\label{prop:MarkovGibbs}
The following properties hold:
\begin{itemize}
\item[$\boldsymbol{(1)}$] (full-branch) $\Fh: J_s^q \rightarrow (g_0(1),1)$ is a $C^2$ diffeomorphism;
\item[$\boldsymbol{(2)}$] 
(uniform expansion)
$\Fh'(x) > f_0'(g_0(1)) > 1, \ \forall \, x \in \Delta$; 
\item[$\boldsymbol{(3)}$] 
(Adler's property)
$\widehat{C}_A = \sup\limits_{\substack{s \ge 2\\q \ge 1}} \left( \sup\limits_{x \in J_s^q} \dfrac{\widehat{F}''(x)}{(\widehat{F}'(x))^2} \right) < +\infty$;
\item[$\boldsymbol{(4)}$]  (bounded distortion)
for $x,y \in J_s^q$, $x \neq y$,
\begin{align*}
\bigg| \ln  \left[ \dfrac{\widehat{F}'(x)}{\widehat{F}'(y)} \right] \bigg| \le \widehat{C}_A \left| \widehat{F}(x)-\widehat{F}(y)\right|. \label{boundeddistortionFhat}
\end{align*}
\item[$\boldsymbol{(5)}$] 
(double induction)
$\Fh=f^{\tau^k}$, where $\tau^k=\sum_{j=0}^{k-1} \tau \circ F^j$.
\end{itemize}
\end{proposition}

\begin{proof}
Let us prove each item separately. 
First, as $\Fh|_{J_s^q} = f_0^{q-1}\circ f_1^{s-1} \circ f_0|_{J_s^q}$ is a $C^2$ diffeomorphism between $J_s^q$ and its image, 
one needs only to check that $\Fh (J_s^q) = \left( g_0 (1) , 1 \right)$. We know that $J_s^q = \left( g_0\circ g_1^{s-1} \circ g_0^q(1) ,  g_0\circ g_1^{s-1} \circ g_0^{q-1}(1) \right)$. Hence:
\begin{align*}
\Fh (J_s^q ) &= f_0^{q-1}\circ f_1^{s-1} \circ f_0 (J_s^q) = f_0^{q-1}\circ f_1^{s-1} \left( g_1^{s-1} \circ g_0^q(1) ,   g_1^{s-1} \circ g_0^{q-1}(1) \right) = \\
&= f_0^{q-1} \left( g_0^q (1) , g_0^{q-1}(1) \right)= \left( g_0 (1) , 1 \right).
\end{align*}
This proves item (1). 

Second, if $x \in J^q$ then $x > g_0(1)$. Then, combining equation \eqref{derivativeFhat} with the fact that $f_0' >1$ and $f_0'$ is increasing one obtains
\begin{align*}
\Fh'(x) = \prod\limits_{l=0}^{q-1}f_0'(F^l(x)) \ge f_0'(x) > f_0'(g_0(1))>1,
\end{align*}
which proves item (2). 

We proceed to prove item (3).
Fix $x \in J_s^q$. Using the expressions \eqref{derivativeFhat} and \eqref{secondderivativeFhat},
\begin{align*}
\begin{array}{l l}
\dfrac{\widehat{F}''(x)}{\widehat{F}'(x)^2} & = \dfrac{\sum\limits_{j=0}^{q-1} \left[\prod\limits_{l=0}^{j-1} f_0'(F^l (x))^2 \right] \cdot f_0''(F^j(x)) \cdot \left[ \prod\limits_{l=j+1}^{q-1} f_0'(F^l(x))\right]}{\left[ \prod\limits_{l=0}^{q-1} f_0'(F^l (x)) \right]^2} \\
\\
\ & = \sum\limits_{j=0}^{q-1} \left[\dfrac{ \left[\prod\limits_{l=0}^{j-1} f_0'(F^l (x))^2 \right] \cdot f_0''(F^j(x)) \cdot \left[ \prod\limits_{l=j+1}^{q-1} f_0'(F^l(x))\right]}{\left[ \prod\limits_{l=0}^{j-1} f_0'(F^l (x))^2 \right] \cdot f_0'(F^j (x))^2 \cdot \prod\limits_{l=j+1}^{q-1} f_0'(F^l(x))^2}\right]\\
\\
\ & =\sum\limits_{j=0}^{q-1} \dfrac{f_0''(F^j(x))}{f_0'(F^j(x))^2} \cdot \prod\limits_{l = j+1}^{q-1}\dfrac{1}{f_0'(F^l(x))}.
\end{array}
\end{align*}
As $f_0$ satisfies Adler's property   (recall Assumption {\bf (A6)}), there exists $C_A>0$ so that  $\dfrac{f_0''(F^j(x))}{f_0'(F^j(x))^2} \le C_A < +\infty$  for every  $0\le j\le q-1$.
This, combined with the fact that  $g_0'$ is strictly decreasing and $x \in J_s^q = (c_s^q , d_s^q)$ ensures that
\begin{align*}
\dfrac{\Fh''(x)}{\Fh'(x)^2} \le C_A \sum\limits_{j=0}^{q-1} \prod\limits_{l=j+1}^{q-1} \dfrac{1}{f_0'(F^l(x))} = C_A \sum\limits_{j=0}^{q-1} \prod\limits_{l=j+1}^{q-1} g_0'(f_0 \circ F^l(x)) < C_A \sum\limits_{j=0}^{q-1} \prod\limits_{l=j+1}^{q-1} g_0'(f_0 \circ F^l(c_s^q)).
\end{align*}
Since $g_0=f_0^{-1}$ and $g_1=f_1^{-1}$,
one has that 
\begin{align*}
f_0 \circ F^l(c_s^q) & = f_0 \circ f_0^{l-1} \circ f_1^{s-1}\circ f_0 \circ g_0 \circ g_1^{s-1} \circ g_0^{q}(1) \\
& = f_0^l \circ f_1^{s-1}\circ g_1^{s-1} \circ g_0^q(1)=
  f_0^l \circ g_0^q (1) = g_0^{q-l}(1).
\end{align*}
for each $j+1\le l \le q-1$.
Hence, using assumption \textbf{(B)},
\begin{align*}
\dfrac{\Fh''(x)}{\Fh'(x)^2} & \le C_A \sum\limits_{j=0}^{q-1} \prod\limits_{l=j+1}^{q-1} g_0' (g_0^{q-l}(1))= C_A\sum\limits_{j=1}^q \prod\limits_{l=1}^{j-1}g_0'(g_0^l(1))
\\
& \le 
 C_A C_I^{(2)} \sum\limits_{j=1}^q \omega_j^{(2)} < C_A C_I^{(2)} \sum\limits_{j=1}^{+\infty} \omega_j^{(2)},
\end{align*}
which is summable because $\prod\limits_{l=1}^{j-1} g_0'(g_0^l (1)) < C_I^{(2)} \cdot \omega_j^{(2)}$, and the constants $\omega_j^{(2)} \in (0,1)$ satisfy
$\sum\limits_{j=1}^{+\infty} \left[ \omega_j^{(2)} \right]^{1-\sigma_2} < +\infty$.
Since the estimate does not depend on $s,q$ and $x$, the proof  of item~(3) is finished.
Item (4) is a direct consequence of item (3) as Adler's property of $\Fh$ implies the bounded distortion property, 
similarly as for $f_0$. 

We are left to prove that 
$\Fh$ can be written as $f^{\tau^k}$, with  $\tau^k=\sum_{j=0}^{k-1} \tau \circ F^j$.
Denote by $\theta: \left(g_0(1),1 \right) \rightarrow \N^*$ the first return time of a point in $\left(g_0(1),1 \right)$
under the action of $f$. By construction $\theta|_{J_s^q} \equiv s+q-1$, as for any point $x\in J_s^q$, the iterate $f^s(x)$ is first at which the point returns to the interval $(0,1)$, and exactly other $q-1$ compositions by $F$ in order to arrive in $\left(g_0(1),1 \right)$. This finishes the proof  of the proposition.
\end{proof}

\begin{corollary}
\label{cor:prob}
The Gibbs-Markov map $\Fh$ has a unique probability measure $\widehat{\mu}$, which is exact and  absolutely continuous with respect to the Lebesgue measure. Moreover, $\frac{d\hat{\mu}}{d\Leb}$ is a piecewise $C^1$ bounded function and  $\widehat{\mu} = \dfrac{1}{\pi_* m ( \Delta )}\cdot (\pi_* m)|_{\Delta}$.
\end{corollary}

\begin{proof}
The existence, uniqueness and exactness of $\widehat{\mu} \ll \Leb$ and regularity of its density are well-known  (see e.g. \cite{alves}, Theorem 3.13).
Now, since $f$ preserves $\pi_* m$, and $\pi_* m ( \Delta )$ is finite, then $\Fh$ preserves $(\pi_*m)|_{\Delta}$ (see \cite[Proposition 4]{zweimuller}).
On the other hand, $\widehat{\mu}$ is the unique invariant measure for $\Fh$ which is absolutely continuous with respect to Lebesgue, meaning that $\widehat{\mu} = \dfrac{1}{\pi_* m ( \Delta )}\cdot (\pi_* m)|_{\Delta}$.
\end{proof}

\subsection{Inducing scheme for $\mathcal P$} \label{sec:secondcomponent}

Let us now extend the inducing procedure to the map $\PPP$: we obtain a skew-product on $\Delta \times \R^+$, with $\Fh$ acting on the first component, while the second one is ruled by non-autonomous compositions of maps $g_x$. The new system does not fit into the standard assumptions, but we show how to work around this obstacle. We also find an invariant ergodic probability measure for the new map, and show that it admits a regular disintegration along the unstable fibres.
\color{black}
For every $n \ge 1$, let us define 
$$
g_x^n= g_{f^{n-1}(x)} \circ g_{f^{n-2}(x)} \circ \dots \circ g_{f(x)} \circ g_x,
$$ when it is well-defined.
We induce $\PPP$ on $\big[ \bigcup\limits_{s \ge 1} I_s \big] \times \mathbb{R}^+$ using the inducing time $\tau$ (which depends only on the first coordinate) and define
$$
P(x,y)= \PPP^{\tau(x)}(x,y)= \left( F(x) , G_x (y) \right)
$$
where $G_x(y)= g_x^{\tau (x)}(y)$.
In order to recover the second inducing on first coordinate we induce $P$ on $\left( g_0(1) , 1 \right) \times \R^+$ and obtain the map
\begin{equation}
\label{eq:doubleinduceP}
    \Ph(x,y):= P^{\kappa(x)}(x,y) = \left(\Fh(x), \Gh_x(y)\right)
\end{equation}
where $\Gh_x (y) = G_x^{\kappa(x)}(y)$ and  $G_x^n(y)= G_{F^{n-1}(x)} \circ \dots \circ G_{F(x)} \circ G_x (y)$, for every $n \ge 1$. We proceed to describe the dynamics of the fiber compositions.

\begin{lemma}
\label{lem:G}
Consider $(x,y) \in \Delta \times \R^+$. Then:
\begin{itemize}
\item[$\boldsymbol{(1)}$] For each $1\le j \le \kappa(x)$ one has
\begin{align}
G_x^j(y)& = g_0^{j-1}\circ g_1^{\tau (x)-1} \circ g_0 (y); \label{powerofG}
\end{align}
\item[$\boldsymbol{(2)}$] $\widehat{G}_x(\R^+) = \left( g_0^{\kappa (x) -1} (\tau(x)-1), g_0^{\kappa (x) -1} (\tau (x)) \right)$;
\item[$\boldsymbol{(3)}$] $\Gh_x$ maps the vertical fibre $\lbrace x \rbrace \times \R^+$ onto an interval of length at most $1$.
\end{itemize}
\end{lemma}

\begin{proof}
Let us prove item (1). If $j=1$, this is an immediate consequence of the definition of $G_x$.
If $2 \le j \le \kappa(x)$, then $F(x), \dots , F^{j-1}(x) \in I_1$, 
hence
\begin{align*}
G_x^j(y)& =G_{F^{j-1}(x)} \circ \dots \circ G_{F(x)} \circ G_x (y) \\
& = g_1^{\tau(F^{j-1}(x))-1} \circ g_0 \circ \dots \circ g_1^{\tau(F(x))-1} \circ g_0 \circ g_1^{\tau (x)-1} \circ g_0 (y)\\
& =g_0^{j-1}\circ g_1^{\tau (x)-1} \circ g_0 (y).
\end{align*}
Items (2) and (3) are direct consequences of the latter.
\end{proof}

Lemma \ref{lem:G} implies that $\Gh_x(y)$ is constant along the $x$ component
on each component $J_s^q \times \R^+$.
Thus, for simplicity, we express the derivative along $y$ as $\Gh'_x(y) = \partial_y \Gh_x(y)$.
Observe that $\Gh'_x(y) < 1$ for every $(x,y) \in \Delta \times \R^+$. 
However, contrary to the standard assumptions of \cite{avilaGouezelyoccoz}, the contracting behavior is not uniform, because in case $x\in J_s^q$ with $q=1$,
\begin{align*}
\lim\limits_{y \to 0^+} \Gh'_x(y)= \lim\limits_{y \to 0^+} \dfrac{d}{dy} \left[ g_1^{s-1} \circ g_0 \right] (y) = \lim\limits_{y \to 0^+} g_0'(y)=1.
\end{align*}
A way to overcome this obstacle is to accelerate the dynamics by considering the second iterate of $\Ph$, which corresponds to a third inducing scheme of constant time.
In order to simplify the notations, let us consider $s_0,s_1 \ge 1$, $s_0,s_1 \ge 2$, and $q_0,q_1 \ge 1$, and define $J_{s_0s_1}^{q_0q_1} = J_{s_0}^{q_0} \cap \Fh^{-1} \left( J_{s_1}^{q_1} \right)$ and $\Dt = \bigcup\limits_{\substack{s_0,s_1 \ge 2\\q_0,q_1 \ge 1}} J_{s_0s_1}^{q_0q_1}$. Then, for every $(x,y) \in \Dt \times \R^+$, take
\begin{equation}
    \label{eq:thirdinduce}
    \Pt (x,y) = \Ph^2 (x,y) = \left( \Ft (x), \Gt_x(y) \right)
\end{equation}
where $\Ft (x) = \Fh^2(x)$ and $\Gt_x(y) = \Gh_{\Fh (x)} \circ \Gh_x(y)$. 
It is straightforward to check that the Adler's property of $\Fh$ implies the Adler's property for $\Ft$, i.e.
\begin{align*}
\widetilde{C}_A = \sup\limits_{\substack{s_0,s_1 \ge 2\\q_0,q_1 \ge 1}} \sup\limits_{x \in J_{s_0s_1}^{q_0q_1}} \dfrac{\Ft''(x)}{(\Ft'(x))^2} < +\infty. \label{adlerFtilde}
\end{align*}
Consequently, $\Ft$ satisfies the bounded distortion property with constant $\widetilde{C}_A$.
Moreover, $\Gt_x^n(y) = \Gh_x^{2n}(y)$, for any $n \ge 1$, when it is well-defined,
and the derivative  $\Gt'$ along the $y$ component is piecewise constant in $x$, since the same holds for $\Gh$.

\begin{proposition}
\label{prop:unifcontraction}
For a fixed $x \in \widetilde{\Delta}$, the map $\Gt_x$ is a uniform contraction on $\mathbb{R}^+$. In particular, 
$ 
\widetilde{G}_x'(y) \le g_0'(1) < 1
$ for every $y>0$.
\end{proposition}

\begin{proof}
Given $(x,y) \in J_{s_0s_1}^{q_0q_1} \times \R^+$, there are two cases to consider.  Assume first $q_0=q_1=1$.
From the definitions of $\Gt, \Gh,$ and $\Fh$, we have that $\Gt_x = \Gh_{\Fh(x)}\circ\Gh_x=G_{F(x)}\circ G_x$. Thus, equation \eqref{powerofG} implies that $G_x(y) = g_0(y)+s_0-1 > 1$. Therefore, since $g_0'$ is decreasing, we conclude that $g_0'(G_x(y)) < g_0'(1)$, so $\Gt_x'(y) =  g_0' \left( G_x(y) \right) \cdot g_0'(y) < g_0'(1)$.

Assume now that either $q_0\neq 1$ or $q_1 \neq 1$. Combining
$\Gt_x(y) = \Gh_{\Fh (x)} \circ \Gh_x(y)$
with the chain rule
$ 
\Gh_x'(y) =  \prod\limits_{i=0}^{q-1}g_0'\left( G_x^i(y)\right)
$ 
one obtains
\begin{align*}
\Gt_x'(y) = \left[ \prod\limits_{j=0}^{q_1-1}g_0'\left( G_{\Fh (x)}^j (\Gh_x(y))\right) \right] \cdot \prod\limits_{i=0}^{q_0-1} g_0' \left( G_x^i (y) \right).
\end{align*}
Since either $q_0$ or $q_1$ are not $1$, at least one of the terms in the previous product has to differ from $1$, in the product above there is the appearance of at least one term of the form 
\begin{align*}
g_0'\left (G_{\Fh (x)} (Gh_x (y)) \right) \ \ \ \ \ \text{or} \ \ \ \ \ g_0'(G_x(y)).
\end{align*}
Proceeding as in the first case, each of this two terms is strictly smaller than $g_0'(1)$.
Estimating all the other terms by $1$ we conclude the proof  of the proposition.
\end{proof}

\subsection{$C^1_b$ regular disintegration} \label{sec:Disintegration}

In this subsection we collect some information concerning the disintegration of the $\widetilde P$-invariant probability measure.  
In order to state the main proposition below let us introduce the notation $C_b^1\left(\Dt \times \R^+\right)$ as the set of functions $u : \Dt \times \R^+ \rightarrow \R$ which are bounded, continuously differentiable and $\sup\limits_{(x,y) \in \Dt \times \R^+} \| Du (x,y) \| < +\infty$.

\begin{proposition}
\label{prop:disintegrazione}
The following properties hold:
\begin{enumerate}
\item there exists a unique $\Pt$-invariant probability measure $\nu$ on $\Dt \times \R^+$ which is ergodic and  $\pi_*\nu = \widehat{\mu}$.\\
\item there exists a family $\lbrace \nu_x \rbrace_{x \in \Dt}$ of probability measures on $\Dt \times \R^+$ which is a $C^1$ disintegration of $\nu$ over $\Dt$, that is:
\begin{itemize}
\item[(i)] $\nu_x$ is supported on $\pi^{-1}(x)$, for every $x \in \Dt$.
\item[(ii)] given any uniformly continuous $v \in C_b \left(\Dt \times \R^+\right)$, consider the map $\overline{v}: \Dt \rightarrow  \mathbb{R}$ defined for every $x \in \Dt$ by $\overline{v}(x)=\int\limits_{\Dt \times \R^+} v \, d\nu_x$. Then, $\overline{v} \in L^{\infty}_{\widehat{\mu}} \left( \Dt \right)$.
\item[(iii)] for any uniformly continuous $v \in C_b \left(\Dt \times \R^+\right)$, we have that
\begin{align*}
\int\limits_{\Dt \times \R^+} v(x,y) \, d\nu(x,y) = \int\limits_{\Dt} \overline{v}(x) \, d\widehat{\mu}(x).
\end{align*}
\item[(iv)] there exists $C_R>0$ such that, if $v \in C_b^1\left(\Dt \times \R^+\right)$ has uniformly continuous derivatives, then $\overline{v} \in C_b^1\left( \Dt \right)$ and
\begin{align*}
\| D\overline{v}(x) \|_{C_b^1 \left( \Dt \right)} \le C_R \sup\limits_{ y \in \R^+}|v(x,y)| + C_R \sup\limits_{y \in \R^+} |Dv(x,y)|.
\end{align*}
\end{itemize}
\end{enumerate}
\end{proposition}

Some comments are in order. 
Note that our base space is not compact nor bounded. Nevertheless, we can still apply the same procedure
as in the proof s in \cite{araujo_pacifico_pujals_viana}, the difference being that one needs to carefully work with compactly supported continuous functions in order to bypass the unboundedness of the stable fibres and obtain the desired disintegration of $\nu$. 
We omit the proof  of Proposition \ref{prop:disintegrazione}, since it mirrors the arguments outlined in \cite{araujo_pacifico_pujals_viana,butterleymelbourne} and refer the reader to \cite{Nicola} for a detailed argument.
Finally, we note that, considering the $\PPP$-invariant ergodic measure $m$ on $\R^+ \times \R^+$, since $\Pt$ is obtained by inducing $\PPP$ on $\Dt \times \R^+$ and $m\left( \Dt \times \R^+ \right) < +\infty$, then the probability measure $\dfrac{1}{m\left(\Dt \times \R^+ \right)}\cdot m|_{\Dt \times \R^+}$ is ergodic, preserved by $\Pt$, and projecs on $\widehat{\mu}$ via $\pi_*$. By its uniqueness one concludes that  $\nu = \dfrac{1}{m\left(\Dt \times \R^+ \right)}\cdot m|_{\Dt \times \R^+}$.

\begin{remark}
Observe that, as pointed out in \cite{butterleymelbourne} (Section $5$, Remark $12$), the $C^1$-regularity statement above is equivalent to the relative requirement in the standard assumptions.    
\end{remark}

\section{The roof function} \label{sec:roof function}

In this section we deal with the induced roof function obtained from the original roof function $\rho$ on $\Dt \times \R^+$ and show that it is bounded away from zero, satisfies the UNI condition and has exponential tails. 
Let us define the first and second induced roof functions as
\begin{equation}
\label{eq:roof function1}
    \text{$R(x,y) = \sum\limits_{j=0}^{\tau(x)-1} \rho \circ \PPP^j (x,y)$, for every $(x,y) \in  \bigcup\limits_{s \ge 1} I_s \times \R^+$}
\end{equation}
and 
\begin{equation}
\label{eq:roof function2}
    \text{$\Rh (x,y) = \sum\limits_{j=0}^{\kappa (x) -1} R \circ P^j (x,y)$, for every $(x,y) \in \Delta \times \R^+$,}
\end{equation}
respectively. These admit the following simple representation formulas:

\begin{proposition}
\label{prop:formuletetto}
The following hold:
\begin{align}
&R(x,y) = \rho_0 \ln \left[ \dfrac{x}{y}\cdot \dfrac{G_x(y)}{F(x)}\cdot \dfrac{F'(x)}{G_x'(y)} \right], \text{for every }(x,y) \in  \bigcup\limits_{s \ge 1} I_s \times \R^+. \label{inducedroof function1}\\
&\Rh(x,y) = \rho_0 \ln \left[ \dfrac{x}{y}\cdot \dfrac{\Gh_x(y)}{\Fh(x)}\cdot \dfrac{\Fh'(x)}{\Gh_x'(y)} \right], \text{for every (x,y) }\in \Delta \times \R^+. \label{inducedroof function2}
\end{align}
\end{proposition}

\begin{proof}
We only prove equality \eqref{inducedroof function1} (the proof  of \eqref{inducedroof function2} is analogous). 
Recall that our original roof function is defined for every $(x,y) \in \left( \R^+ \setminus \lbrace 1 \rbrace \right) \times \R^+$, reading
\begin{align*}
\rho (x,y) = \rho_0 \ln \left[ \dfrac{x}{y} \cdot \dfrac{g_x(y)}{f(x)}\cdot \dfrac{f'(x)}{\partial_y g_x(y)}\right].
\end{align*}
Thus, given an arbitrary $s\ge 1$ and $x \in I_s$,
\begin{align*}
R (x,y) &= \sum\limits_{j=0}^{s-1} \rho \circ \PPP^j (x,y) = \sum\limits_{j=0}^{s-1} \rho\left( f^j(x) , g_x^j (y) \right) =  \rho_0 \ln \left[ \prod\limits_{j=0}^{s-1} \dfrac{f^j(x)}{g_x^j(y)} \cdot \dfrac{g_x^{j+1} (y)}{f^{j+1}(x)} \cdot \dfrac{f'(f^j(x))}{\partial_y g_x(g_x^j(y))} \right].
\end{align*}
If $s=1$ there is nothing to prove as in this case $F = f$ and $G_x = g_x$.
If $s \ge 2$, recalling that $F = f^s$ and $G_x = g_x^s$, we observe that
$$
\prod\limits_{j=0}^{s-1} \dfrac{f^j(x)}{f^{j+1}(x)} = \dfrac{x\cdot \prod\limits_{j=1}^{s-1} f^j(x)}{\left[\prod\limits_{j=1}^{s-1}f^j(x) \right] \cdot f^s (x)} = \dfrac{x}{F(x)}
\quad\text{and}\quad 
\prod\limits_{j=0}^{s-1} \dfrac{f^j(x)}{f^{j+1}(x)} = \dfrac{x\cdot \prod\limits_{j=1}^{s-1} f^j(x)}{\left[\prod\limits_{j=1}^{s-1}f^j(x) \right] \cdot f^s (x)} = \dfrac{x}{F(x)},
$$
that
$$
\prod\limits_{j=0}^{s-1} \dfrac{g_x^{j+1}(y)}{g_x^{j}(y)} = \dfrac{\left[\prod\limits_{j=1}^{s-1}g_x^j(y) \right] \cdot g_x^s(y)}{y\cdot \prod\limits_{j=1}^{s-1} g_x^j(y)} = \dfrac{G_x(y)}{y}
\quad\text{and}\quad 
\prod\limits_{j=0}^{s-1}f'(f^j(x)) = F'(x)
$$
and
$$
\prod\limits_{j=0}^{s-1}\dfrac{1}{\partial_yg_x(g_x^j(y))} = \dfrac{1}{G_x'(y)}.
$$
These five expressions combined prove \eqref{inducedroof function1}, as desired. 
\end{proof}

\begin{definition}
\label{def:coomologia}
Recall that given a generic dynamical system $T:X \rightarrow X$, we say that two functions $\phi, \psi: X \rightarrow \R$ are \emph{cohomologous} if there exists a measurable \emph{cohomology function} $u: X \rightarrow \R$ such that $\phi = \psi + u - u\circ T$.
\end{definition}

A crucial fact in the study of the correlation function for suspension flows in \cite{avilaGouezelyoccoz} is that 
the roof function is cohomologous to a roof function that depends only on the unstable direction. However, the roof functions \eqref{inducedroof function1} and \eqref{inducedroof function2} depend on both coordinates $(x,y)$. We proceed to show that the induced roof function is cohomologous to one such roof function
(cf. Proposition~\ref{prop:truccodibowen} for the precise statement).
First we observe the following.

\begin{lemma}
\label{prop:coomologia}
The roof function $\Rh$ is cohomologous to  $\rh: \Delta \times \R^+ \rightarrow \R^+$ defined by
\begin{align*}
\rh (x,y) = \rho_0 \ln \left[ \dfrac{\Fh'(x)}{\Gh_x'(y)} \right],
\quad \text{for every $(x,y) \in \Delta \times \R^+$}.
\end{align*}
\end{lemma}

\begin{proof}
Take $v (x,y) = \rho_0 \ln \left( \dfrac{x}{y} \right)$, for $(x,y) \in \Delta \times \R^+$. Then
\begin{align*}
\Rh(x,y) &= \rho_0 \ln \left[ \dfrac{x}{y}\cdot \dfrac{\Gh_x(y)}{\Fh(x)}\cdot \dfrac{\Fh'(x)}{\Gh_x'(y)} \right] \\
& = \rho_0 \ln \left[ \dfrac{\Fh'(x)}{\Gh_x'(y)} \right] + \rho_0 \ln \left( \dfrac{x}{y} \right) - \rho_0 \ln \left[ \dfrac{\Fh(x)}{\Gh_x(y)} \right]\\[6pt]
&= \rh(x,y) + v(x,y) - v \circ \Ph (x,y)
\end{align*}
for every $(x,y) \in \Delta \times \R^+$. 
\end{proof}

The latter means that there exists a modified cross-section in which the flow can be modeled by a suspension flow with roof function $\hat r$.
Since we consider the map $\Pt$, we  consider the corresponding roof function defined on $\Dt \times \R^+$, that is
$\rt = \rh + \rh \circ \Ph$.
A straightforward computation yields that 
\begin{align}
\rt (x,y) = \rho_0 \ln \left[ \dfrac{ \Ft'(x)}{\Gt_x'(y)}\right],
\quad \text{for every $(x,y) \in \Dt \times \R^+$.}\label{rtilde}
\end{align}

The following notations are intended to simplify some of the upcoming expressions and computations. Consider $s,s_0,s_1, q,q_0,q_1 \ge 1$, with $s,s_0,s_1 \ge 2$, and define:
\begin{itemize}
\item $\Gh_{J_s^q} (y) = g_0^{q-1}\circ g_1^{s-1}\circ g_0 (y)$, with $y \in \R^+$.
\item $\Gt_{J_{s_0s_1}^{q_0q_1}} (y) = \Gh_{J_{s_1}^{q_1}} \circ \Gh_{J_{s_0}^{q_0}} (y)$, with $y \in \R^+$.
\item $\rh_s^q (x,y) = \rho_0 \ln \left[ \dfrac{\Fh'(x)}{\Gh_{J_s^q}'(y)}\right]$, with $(x,y) \in \overline{J_s^q}\times \R^+$.
\item $\rt_{s_0s_1}^{q_0q_1} = \rh_{s_0}^{q_0} + \rh_{s_1}^{q_1} \circ \Ph$, which is well-defined on $\overline{J_{s_0s_1}^{q_0q_1}} \times \R^+$.
\end{itemize}

Lemma \ref{lem:G} implies that, if $(x,y) \in J_s^q \times \R^+$, then $\Gh_x(y) = \Gh_{J_s^q}(y)$. Moreover, $\rh|_{J_s^q\times \R^+} = \rh_s^q$, and $\rt|_{J_{s_0s_1}^{q_0q_1}\times \R^+} = \rt_{s_0s_1}^{q_0q_1}$.\

\begin{definition}
\label{def:suitable}
We will say that $x \in \Dt$ is \emph{suitable} if $\Ft^n (x)$ is well-defined for every $n\ge 1$.\\
A pair $(x,y) \in \Dt \times \R^+$ is \emph{suitable} if $x$ is suitable.
\end{definition}

Note that the set of suitable points in $\Dt$ is dense and has full $\widehat{\mu}$-measure, as it obtained by removing from $\Dt$ the countable set of all the preimages, with respect to $\Ft^n$, of the boundary points of the intervals $J_{s_0s_1}^{q_0q_1}$, for any $n\ge 1$. Similarly, the set of suitable pairs in $\Dt \times \R^+$ is dense and has full $\nu$-measure. 
The following result, often referred to as {Bowen's trick} and whose proof  is given in Appendix~\ref{sec:appendixa}, guarantees that $\tilde r$ is cohomologous to a roof function which is constant along the stable foliation. More precisely:

\begin{proposition}
\label{prop:truccodibowen}
Given $y' \in \R^+$ fixed, the function
\begin{align*}
u(x,y) = \sum\limits_{i=0}^{+\infty} \left[ \rt \left( \Pt^i (x,y) \right) - \rt \left( \Pt^i (x,y') \right) \right] \label{bowen}
\end{align*}
is well-defined for every suitable $(x,y) \in \Dt \times \R^+$ and it satisfies 
\begin{align*}
\rt = \rt (\pi (\cdot , \cdot ) , y' )+u-u\circ\Pt,
\end{align*}
where $\pi$ is the projection on the first component.
In particular, $\rt$ is cohomologous to the roof function $\rt (\pi (\cdot , \cdot ) , y' )$, which is constant along the stable foliation.

\end{proposition}

In view of the proposition, we may fix $y'\in \mathbb R$ and define, for each $x \in \Dt$,
\begin{equation}
    \label{eq:aftercohom}
    r(x) = \rt (x , y' ) = \rho_0 \ln \left[ \dfrac{\Ft' (x) }{\Gt_x' (y')}\right]
    \quad\text{and}\quad
    r_{s_0s_1}^{q_0q_1}(x) = \rt_{s_0s_1}^{q_0q_1} ( x , y'),
\end{equation}
where $s_0,s_1,q_0,q_1 \ge 1, \ s_0,s_1 \ge 2$.
By a slight abuse of notation, we will say that 
$\rt$ is cohomologus to $r$ given by \eqref{eq:aftercohom}, which depends only  on the unstable component.
In particular, we can now look at the suspension flow over $\Pt$ and roof function $r$, which shares the same dynamical properties as the one with roof function $\rt$. 

In the remaining of this section we proceed to show that the roof function $r$ satisfies the standard assumptions.
Let us first fix some notations.
The inverse branch of $\Fh\mid_{J_s^q}$ will be denoted by
\begin{align*}
\phi_s^q = \left[ \Fh |_{J_s^q} \right]^{-1} : \left(g_0(1) , 1 \right)  \rightarrow J_s^q,
\end{align*}
and for each $l \ge 1$ we will denote by $\Phi_l$ the set of inverse branches of $\Fh^l$, which are obtained as compositions of functions of the form $\phi_{s_0}^{q_0} \circ \dots \circ \phi_{s_{l-1}}^{q_{l-1}} :  \left( g_0(1) , 1 \right) \rightarrow \bigcap\limits_{i=0}^{l-1}\Fh^{-i} (J_{s_i}^{q_i} )$, for some parameters $s_0, \dots , s_{l-1} \ge 2$ and $q_0, \dots , q_{l-1}\ge 1$.


The next theorem ensures that the standard assumptions are satisfied by the roof function (we denote by $D$ the derivative with respect to $x$).

\begin{proposition}
\label{thm:rstandard1}

The roof function $r$ satisfies the following properties:
\begin{enumerate}
    \item $\inf\limits_{x \in \Dt} r (x) > 0$;
    \item There exists $C_D>0$ such that, for any $\phi \in \Phi_2$, $\| D(r \circ \phi ) \|_{C^0} \le C_D$.
\end{enumerate}
\end{proposition}

\begin{proof}
For the first property, observe that $\Ft'(x) > [f_0'(g_0(1))]^2 > 1$ and $\Gt_x'(y') < 1$, for any $x \in \Dt$. Hence
\begin{align*}
r(x) = \rho_0 \ln \left[ \dfrac{\Ft'(x)}{\Gt_x'(y')}\right] > 2\rho_0\ln \left[ f_0'(g_0(1)) \right] > 0
\end{align*}
for any $x \in \Dt$, thus proving item (1).
Now, consider an inverse branch of $\Ft$, say $\phi_{s_0s_1}^{q_0q_1}: \left( g_0(1),1\right) \rightarrow J_{s_0s_1}^{q_0q_1}$ of $\Phi_2$.
By  definition, $r \circ \phi_{s_0s_1}^{q_0q_1} = r_{s_0s_1}^{q_0q_1} \circ \phi_{s_0s_1}^{q_0q_1}$, where 
\begin{equation}
    \label{eq:roof functions0s1q0q1}
r_{s_0s_1}^{q_0q_1} = \rho_0 \ln \left[ \dfrac{\Ft'}{\Gt_{J_{s_0s_1}^{q_0q_1}}'(y')}\right].
\end{equation}
Therefore:
\begin{align*}
D \left( r_{s_0s_1}^{q_0q_1} \circ \phi_{s_0s_1}^{q_0q_1} (x) \right) = \left[ Dr_{s_0s_1}^{q_0q_1} \right] \left( \phi_{s_0s_1}^{q_0q_1} (x)\right) \cdot D \phi_{s_0s_1}^{q_0q_1}(x) = \rho_0 \cdot \dfrac{\Ft'' \left( \phi_{s_0s_1}^{q_0q_1}(x) \right)}{\Ft'\left( \phi_{s_0s_1}^{q_0q_1}(x) \right)} \cdot D \phi_{s_0s_1}^{q_0q_1}(x).
\end{align*}
Using the Adler's property for $\Ft$, we find that
\begin{align*}
\dfrac{\Ft'' \left( \phi_{s_0s_1}^{q_0q_1}(x) \right)}{\Ft'\left( \phi_{s_0s_1}^{q_0q_1}(x) \right)} = \dfrac{\Ft'' \left( \phi_{s_0s_1}^{q_0q_1}(x) \right)}{\Ft'\left( \phi_{s_0s_1}^{q_0q_1}(x) \right)^2} \cdot \Ft'\left( \phi_{s_0s_1}^{q_0q_1}(x) \right) \le \widetilde{C}_A \cdot \Ft'\left( \phi_{s_0s_1}^{q_0q_1}(x) \right),
\end{align*}
so that
\begin{align*}
D \left( r_{s_0s_1}^{q_0q_1} \circ \phi_{s_0s_1}^{q_0q_1} (x) \right) \le \rho_0 \!\cdot \!\widetilde{C}_A \!\cdot \!\Ft'\left( \phi_{s_0s_1}^{q_0q_1}(x) \right) \!\cdot \!D \phi_{s_0s_1}^{q_0q_1}(x) = \rho_0\!\cdot\! \widetilde{C}_A \!\cdot\! D\left[ \Ft \circ \phi_{s_0s_1}^{q_0q_1} \right] \!(x) = \rho_0\!\cdot\! \widehat{C}_A.
\end{align*}
The result follows from the  arbitrariness of $s_0,s_1,q_0,q_1$ and $x$.
\end{proof}

We are now in a position to prove that $r$ satisfies the UNI condition.

\begin{proposition}
\label{thm:rstandard2}
The roof function $r$ 
can not be written as $r= \psi +\phi \circ T-\phi$ on $\Dt$, where $\psi: \Dt \rightarrow \R$ is constant on $J_{s_0s_1}^{q_0q_1}$, for each $s_0,s_1,q_0,q_1 \ge 1, \ s_0,s_1 \ge 2$, and $\phi: \Delta \rightarrow \R$ is measurable (UNI).
\end{proposition}

\begin{proof}
It is nowadays well known that the latter is equivalent to show that 
there exist $C_U>0$ such that, for every $n \ge 1,$ there are inverse branches $ \phi , \overline{\phi} \in \Phi_{2n}$ of $\Ft$ such that, if $\psi_{\phi , \overline{\phi}} = r^{(n)} \circ \phi - r^{(n)} \circ \overline{\phi}$ with $r^{(n)} = \sum\limits_{i=0}^{n-1} r \circ \Ft^i$, then
\begin{align}
\label{eq:originalUNI}
\inf\limits_{x \in \left( g_0(1), 1 \right)} | D \psi_{\phi , \overline{\phi}} (x)| \ge C_U
\end{align}
(see e.g. \cite[page 184]{avilaGouezelyoccoz} for the proof  of the equivalence). 

We proceed to estimate the derivative of the map $\psi_{\phi , \overline{\phi}}$ in the special case that each of the maps $\phi$ and $\overline{\phi}$ are compositions of some specific inverse branch of $F$. More precisely, 
take $n \ge 1$ and choose $\phi = [\phi_2^1]^{2n}$, $\overline{\phi} = [ \phi_3^1]^{2n}$ (meaning that they are $2n$ compositions of individual maps $\phi_2^1$ and $\phi_3^1$ with themselves, respectively). Given $0\le i\le n-1$, observe that $\Ft^i\circ \phi = [\phi_2^1]^{2(n-i)}$ and $\Ft^i\circ \overline{\phi} = [\phi_3^1]^{2(n-i)}$. Together with $r|_{J_{s_0s_1}^{q_0q_1}} = r_{s_0s_1}^{q_0q_1} = \rh_{s_0}^{q_0} + \rh_{s_1}^{q_1} \circ \Fh$, this ensures that one can write
\begin{align*}
\psi_{\phi,\overline{\phi}} & = \sum\limits_{i=0}^{n-1} \left[ r \circ \Ft^i \circ \phi - r \circ \Ft^i \circ \overline{\phi} \right] = \sum\limits_{i=0}^{n-1} \left[ r_{22}^{11} \circ [\phi_2^1]^{2(n-i)} - r_{33}^{11} \circ [\phi_3^1]^{2(n-i)} \right]\\
&= \sum\limits_{i=0}^{n-1} \left\{ \left[\rh_2^1\circ [\phi_2^1]^{2(n-i)} - \rh_3^1 \circ [\phi_3^1]^{2(n-i)} \right] + \left[ \rh_2^1\circ [\phi_2^1]^{2(n-i)-1} - \rh_3^1 \circ [\phi_3^1]^{2(n-i)-1} \right]\right\}\\
&=\sum\limits_{l=0}^{2n-1} \left [ \rh_2^1 \circ [\phi_2^1]^{2n-l} - \rh_3^1\circ[\phi_3^1]^{2n-l} \right]
\end{align*}
and, consequently,
$$
D \psi_{\phi,\overline{\phi}} = \sum\limits_{l=0}^{2n-1} D\left[ \rh_2^1 \circ [\phi_2^1]^{2n-l} - \rh_3^1\circ[\phi_3^1]^{2n-l} \right].
$$
Recall that, by Lemma~\ref{prop:coomologia}, $\rh_s^1(x) \!=\! \rho_0\cdot \ln \left[\dfrac{\Fh|_{J_s^1}'(x)}{\Gh'_{J_s^1}(y')}\right]\!=\!\rho_0\cdot \ln \left[\dfrac{f_0'(x)}{\Gh'_{J_s^1}(y')}\right]$, so that its derivative reads as $D\rh_s^1 (x) = \rho_0 \cdot \dfrac{f_0''(x)}{f_0'(x)}$, for some constant $\rho_0>0$.
Hence, by the chain rule, 
\begin{align*}
D\{ \rh_2^1 & \circ [\phi_2^1]^{2n-l}  - \rh_3^1\circ[\phi_3^1]^{2n-l} \} 
\\ & = \rho_0 \cdot \left\{ \dfrac{f_0''\circ [\phi_2^1]^{2n-l}}{f_0'\circ[\phi_2^1]^{2n-l}} \cdot D[\phi_2^1]^{2n-l} - \dfrac{f_0''\circ [\phi_3^1]^{2n-l}}{f_0'\circ[\phi_2^1]^{2n-l}} \cdot D[\phi_3^1]^{2n-l}\right\} \\
&= \rho_0 \cdot \left\{ \dfrac{f_0''\circ [\phi_2^1]^{2n-l}}{f_0'\circ[\phi_2^1]^{2n-l}} \cdot \prod_{j=0}^{2n-l-1}[D\phi_2^1] \circ [\phi_2^1]^j - \dfrac{f_0''\circ [\phi_3^1]^{2n-l}}{f_0'\circ[\phi_2^1]^{2n-l}} \cdot \prod_{j=0}^{2n-l-1}[D\phi_3^1] \circ [\phi_3^1]^j \right\}
\end{align*}
for a generic $0\le l \le 2n-1$.
Observe that $D[\phi_s^1] = D[\Fh|_{J_s^1}^{-1}] = \dfrac{1}{\Fh|_{J_s^1}' \circ \phi_s^1} = \dfrac{1}{f_0' \circ \phi_s^1}$, which implies that
\begin{align*}
[D\phi_s^1] \circ [\phi_s^1]^{2n-l-1} = \dfrac{1}{f_0' \circ \phi_s^1 \circ [\phi_s^1]^{2n-l-1}} = \dfrac{1}{f_0' \circ [\phi_s^1]^{2n-l}}
\end{align*}
for every $s\ge 1$.

Therefore, we can write $D\left[ \rh_2^1 \circ [\phi_2^1]^{2n-l} - \rh_3^1\circ[\phi_3^1]^{2n-l} \right](x)$ for any $x \in \left(g_0(1),1\right)$ as:
\begin{align*}
\rho_0 \cdot \left\{ \dfrac{f_0''\circ [\phi_2^1]^{2n-l}(x)}{\left( f_0'\circ[\phi_2^1]^{2n-l}(x) \right)^2} \cdot \prod_{j=0}^{2n-l-2}[D\phi_2^1] \circ [\phi_2^1]^j(x) - \dfrac{f_0''\circ [\phi_3^1]^{2n-l}(x)}{\left( f_0'\circ[\phi_2^1]^{2n-l} (x)\right)^2} \cdot \prod_{j=0}^{2n-l-2}[D\phi_3^1] \circ [\phi_3^1]^j(x)  \right\}
\end{align*}

(if $l = 2n-1$, then by notation the product that appears above is just equal to $1$).

\begin{figure}[htb]
    \centering
\includegraphics[width=0.3\linewidth]{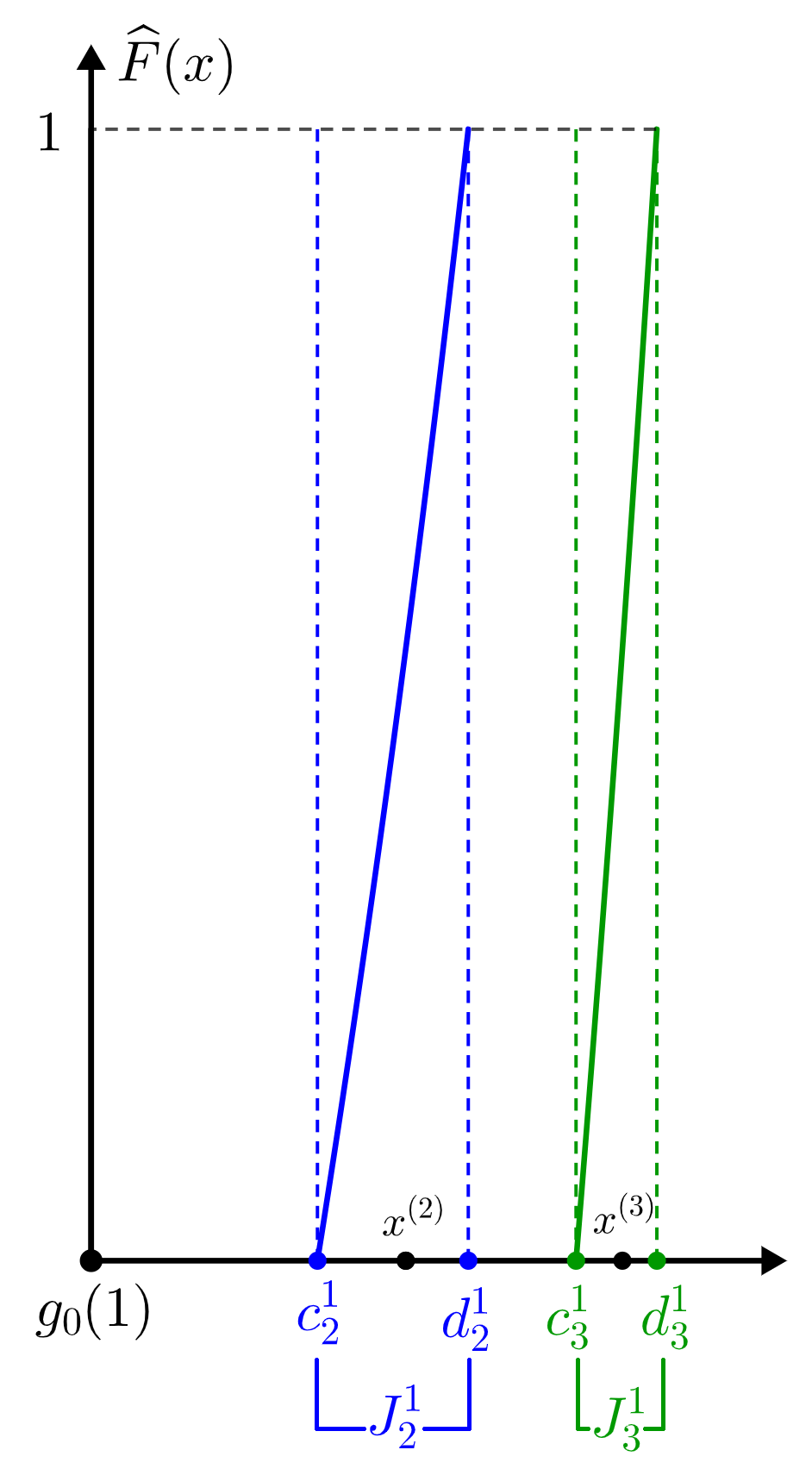}
    \caption{A close-up of intervals $J_2^1$ and $J_3^1$ and the graph of $\Fh$ over them. Here, $x^{(2)} = [\phi_2^1]^{2n-l}(x)$ and $x^{(3)} = [\phi_3^1]^{2n-l}(x)$.}
    \label{fig:intervalsj21j31}
\end{figure}

Now, by construction, as  $ [\phi_2^1]^{2n-l}(x) \in J_2^1$, and $[\phi_3^1]^{2n-l}(x) \in J_3^1$, then $[\phi_2^1]^{2n-l}(x) < [\phi_3^1]^{2n-l}(x)$. Moreover, the assumption \textbf{(A5)} states that the function $(0,1)\ni \xi \mapsto \dfrac{f_0''(\xi)}{(f_0'(\xi))^2}$ is decreasing. In particular, for any $\xi \in J_2^1 = ( c_2^1 , d_2^1 )$ and $\xi' \in J_3^1 = (c_3^1 , d_3^1 )$, we find that
\begin{align*}
\dfrac{f_0''(\xi)}{(f_0'(\xi))^2} > \dfrac{f_0''(d_2^1)}{(f_0'(d_2^1))^2} > \dfrac{f_0''(c_3^1)}{(f_0'(c_3^1))^2} > \dfrac{f_0''(\xi')}{(f_0'(\xi'))^2},
\end{align*}
so that
\begin{align*}
\dfrac{f_0''(\xi)}{(f_0'(\xi))^2}  - \dfrac{f_0''(\xi')}{(f_0'(\xi'))^2} >  \dfrac{f_0''(d_2^1)}{(f_0'(d_2^1))^2} - \dfrac{f_0''(c_3^1)}{(f_0'(c_3^1))^2} 
=: C_U'.
\end{align*}
Note that the constant $C_U'>0$ is independent of $n$ and that 
\begin{align*}
\dfrac{f_0''\circ [\phi_2^1]^{2n-l}(x)}{\left( f_0'\circ[\phi_2^1]^{2n-l}(x) \right)^2} - \dfrac{f_0''\circ [\phi_3^1]^{2n-l}(x)}{\left( f_0'\circ[\phi_3^1]^{2n-l}(x) \right)^2} > C_U'
\end{align*}
and, in particular,
\begin{align*}
\dfrac{f_0''\circ [\phi_2^1]^{2n-l}(x)}{\left( f_0'\circ[\phi_2^1]^{2n-l}(x) \right)^2}> \dfrac{f_0''\circ [\phi_3^1]^{2n-l}(x)}{\left( f_0'\circ[\phi_3^1]^{2n-l}(x) \right)^2}.
\end{align*}
Now, observe that $f_0'(d_2^1) < f_0'(c_3^1)$, so that for any $\xi,\xi' \in \left(g_0(1),1\right)$, we have:
\begin{align*}
D\phi_2^1 ( \xi ) = \dfrac{1}{\Fh' (\phi_2^1(\xi))} = \dfrac{1}{f_0' (\phi_2^1(\xi))} > \dfrac{1}{f_0' (d_2^1)} > \dfrac{1}{f_0' (c_3^1)} > \dfrac{1}{f_0' (\phi_3^1(\xi'))}=\dfrac{1}{\Fh' (\phi_3^1(\xi'))}=D\phi_3^1 ( \xi' ).
\end{align*}
Thus, whenever $0\le l \le 2n-2$, we have that
\begin{align*}
\prod_{j=0}^{2n-l-2}[D\phi_2^1] \circ [\phi_2^1]^j(x) > \prod_{j=0}^{2n-l-2}[D\phi_3^1] \circ [\phi_3^1]^j(x).
\end{align*}
Combining all these estimates,
\begin{align*}
\dfrac{f_0''\circ [\phi_2^1]^{2n-l}(x)}{\left( f_0'\circ[\phi_2^1]^{2n-l}(x) \right)^2} \cdot \prod_{j=0}^{2n-l-2}[D\phi_2^1] \circ [\phi_2^1]^j(x) > \dfrac{f_0''\circ [\phi_3^1]^{2n-l}(x)}{\left( f_0'\circ[\phi_3^1]^{2n-l}(x) \right)^2} \cdot \prod_{j=0}^{2n-l-2}[D\phi_3^1] \circ [\phi_3^1]^j(x),
\end{align*}
and thus $D\left[ \rh_2^1 \circ [\phi_2^1]^{2n-l} - \rh_3^1\circ[\phi_3^1]^{2n-l} \right](x) > 0$.
As for the last term of the sum (i.e. for $l = 2n-1$) we note that we find
\begin{align*}
D\left[ \rh_2^1 \circ \phi_2^1 - \rh_3^1\circ\phi_3^1 \right](x) & = \rho_0 \cdot \left\{\dfrac{f_0''\circ\phi_2^1(x)}{(f_0'\circ\phi_2^1(x))^2} - \dfrac{f_0''\circ\phi_3^1(x)}{(f_0'\circ\phi_3^1(x))^2}\right\} >  \rho_0 \cdot C_U'.
\end{align*}
Altogether we conclude that the statement of the proposition holds with $C_U = \rho_0\cdot C_U'$: 
\begin{align*}
D\psi_{\phi,\overline{\phi}}(x) = \sum\limits_{l=0}^{2n-1} D\left\{ \rh_2^1 \circ [\phi_2^1]^{2n-l}(x) - \rh_3^1\circ[\phi_3^1]^{2n-l}(x) \right\}\ge C_U
\end{align*}
for every $x \in \left(g_0(1),1\right)$.
(we estimate the terms related to $0\le l \le 2n-2$ from below by $0$ and the term for $l=2n-1$ with $C_U$), and one concludes that ~\eqref{eq:originalUNI} holds. This finishes the proof  of the proposition.
\end{proof}

In order for the roof function to satisfy the standard assumptions it remains to prove that it has exponential tails.

\begin{proposition}
\label{thm:rstandard3}
The roof function $r: \Dt \to \mathbb R^+$ has exponential tails: there exists $\sigma>0$ such that $\int e^{\sigma r(x)} \, dx < +\infty$.
\end{proposition}

The proof  of Proposition~\ref{thm:rstandard3} will occupy the remainder of this section. First, let us denote $\left(c_{s_0s_1}^{q_0q_1} , d_{s_0s_1}^{q_0q_1} \right) = J_{s_0s_1}^{q_0q_1}$, for some fixed $s_0,s_1 , q_0,q_1 \ge 1$ with $s_0,s_1 \ge 2$. Using \eqref{eq:roof functions0s1q0q1}, 
\begin{align}
\int\limits_{\Dt} e^{\sigma r(x)} \, dx &= \sum\limits_{\substack{s_0,s_1 \ge 2\\q_0,q_1 \ge 1}} \  \int\limits_{J_{s_0s_1}^{q_0q_1}} e^{\sigma \cdot r_{s_0s_1}^{q_0q_1}(x)} \, dx = \sum\limits_{\substack{s_0,s_1 \ge 2\\q_0,q_1 \ge 1}} \  \int\limits_{J_{s_0s_1}^{q_0q_1}} \left[ \dfrac{\Ft'(x)}{\Gt'_{J_{s_0s_1}^{q_0q_1}}(y')} \right]^{\sigma \cdot \rho_0} \, dx
\nonumber \\
&< \sum\limits_{\substack{s_0,s_1 \ge 2\\q_0,q_1 \ge 1}} |J_{s_0s_1}^{q_0q_1}| \cdot \left[ \dfrac{\Ft'(d_{s_0s_1}^{q_0q_1})}{\Gt'_{J_{s_0s_1}^{q_0q_1}}(y')}\right]^{\sigma \cdot \rho_0},
\label{eqseriesRHS}
\end{align}
where we used that $\Ft'(x) < \Ft'(d_{s_0s_1}^{q_0q_1})$ by monotonicity. 
Therefore, we focus on the convergence of the series in the right hand side above. We need the following auxiliary result, whose proof  will be given in the Appendix \ref{sec:appendixb}.

\begin{lemma}
\label{lemma:ordini}

Fix $s_0,s_1 \ge 2$ and $q_0,q_1 \ge 1$. Then:
\begin{itemize}
\item[$\boldsymbol{(1)}$] $\dfrac{1}{\widetilde{C}} \cdot |J_{s_0s_1}^{q_0q_1}|^{-1} \le \Ft'(d_{s_0s_1}^{q_0q_1} ) \le \widetilde{C} \cdot |J_{s_0s_1}^{q_0q_1}|^{-1}$, for some $\widetilde{C}>0$.
\item[$\boldsymbol{(2)}$] $\dfrac{1}{\widetilde{C}'} \cdot \Ft'(d_{s_0s_1}^{q_0q_1}) \le \dfrac{1}{\Gt'_{J_{s_0s_1}^{q_0q_1}}(y')} \le \widetilde{C}' \cdot \Ft'(d_{s_0s_1}^{q_0q_1})$, for some $\widetilde{C}' >0$.
\end{itemize}
\end{lemma}

As a consequence of 
Lemma \ref{lemma:ordini} implies that the series in \eqref{eqseriesRHS} converges if and only if 
$$\sum\limits_{\substack{s_0,s_1 \ge 2\\q_0,q_1 \ge 1}} \left[ \dfrac{1}{\Ft'(d_{s_0s_1}^{q_0q_1})} \right]^{1-2\sigma \cdot \rho_0} < +\infty.$$
By bounded distortion and the fact that $d_{s_0s_1}^{q_0q_1} < d_{s_0}^{q_0}$,
\begin{align*}
\dfrac{1}{\Ft'(d_{s_0s_1}^{q_0q_1})} = \dfrac{1}{\Fh'(d_{s_1}^{q_1})} \cdot \dfrac{1}{\Fh'(d_{s_0s_1}^{q_0q_1})} \le \widetilde{C}_D \cdot \dfrac{1}{\Fh'(d_{s_1}^{q_1})} \cdot \dfrac{1}{\Fh'(d_{s_0}^{q_0})}.
\end{align*}

\begin{figure}[htb]
    \centering
    \includegraphics[width=1\linewidth]{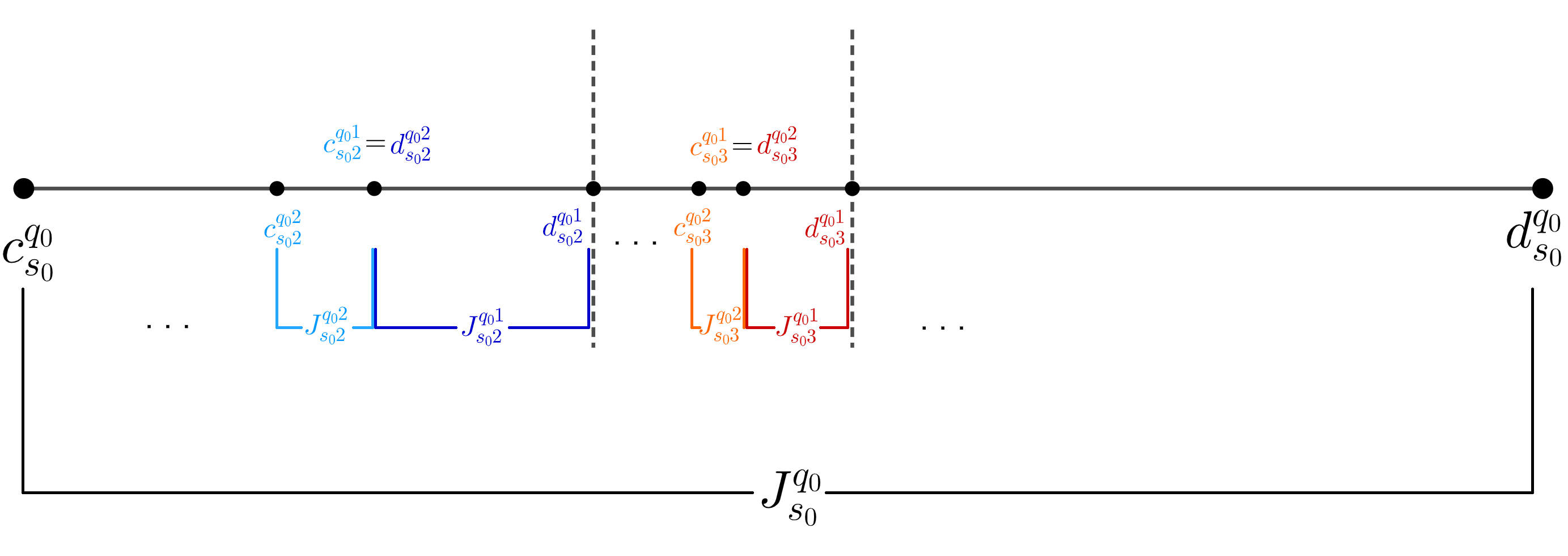}
    \caption{A representation of the interval $J_{s_0}^{q_0}$ and its partition into intervals of the form $J_{s_0s_1}^{q_0q_1}$. Note that this structure is a replica inside of $J_{s_0}^{q_0}$ of the partition of $\left( g_0(1) , 1 \right)$ by the sets $J_s^q$.}
    \label{fig:nestedintervals}
\end{figure}

Hence we need the following auxiliary result which estimates the derivative of $\Fh$ at the points $d_{s_0}^{q_0}$ and $d_{s_1}^{q_1}$.


\begin{lemma}
\label{lemma:ordineminore}
Given any $s \ge 2$ and $q\ge1$, we have that
\begin{align*}
\dfrac{1}{\Fh' (d_s^q)} < C_I^{(1)} \cdot C_I^{(2)} \cdot \omega_s^{(1)} \cdot \omega_q^{(2)}, 
\end{align*}
where these constants are given by
assumption \textbf{(B)}.
\end{lemma}

\begin{proof}

Observe that $d_s^q = \Fh|_{J_s^q}^{-1} (1) = \phi_s^q (1)= g_0 \circ g_1^{s-1} \circ g_0^{q-1} (1)$, so that
\begin{align*}
D\phi_s^q(1) = g_0' (g_1^{s-1} \circ g_0^{q-1}(1)) \cdot \prod\limits_{i=0}^{q-2} g_0' ( g_0^i (1)).
\end{align*}
Now, $g_1^{s-1} \circ g_0^{q-1}(1) > s-1$, and $g_0'$ is decreasing. Hence, $g_0' (g_1^{s-1} \circ g_0^{q-1}(1)) < g_0'(s-1) < C_I^{(1)}\cdot \omega_s^{(1)}$.
Furthermore,
\begin{align*}
\prod\limits_{i=0}^{q-2} g_0' ( g_0^i (1)) = \dfrac{g_0'(1)}{g_0'(g_0^{q-1}(1))} \cdot \prod\limits_{i = 1}^{q-1} g_0'(g_0^i(1)) < C_I^{(2)} \cdot \omega_q^{(2)},
\end{align*}
where we used that $\dfrac{g_0'(1)}{g_0'(g_0^{q-1}(1))} \le 1$, and again assumption \textbf{(B)}.
Thus,
\begin{align*}
\dfrac{1}{\Fh' (d_s^q)} = D\phi_s^q (1) < \left[ C_I^{(1)}\cdot \omega_s^{(1)} \right] \cdot \left[ C_I^{(2)}\cdot \omega_q^{(2)} \right].
\end{align*}
This proves the lemma.
\end{proof}

Now, we are in a position to complete 
the proof  of Proposition~\ref{thm:rstandard3}. In fact, Lemma~\ref{lemma:ordineminore} implies that 
\begin{align*}
&\sum\limits_{\substack{s_0,s_1 \ge 2\\q_0,q_1 \ge 1}} \left[ \dfrac{1}{\Ft'(d_{s_0s_1}^{q_0q_1})} \right]^{1-2\sigma \cdot \rho_0}  < \sum\limits_{\substack{s_0,s_1 \ge 2\\q_0,q_1 \ge 1}} \left\{ \left[ C_I^{(1)} \cdot C_I^{(2)} \right]^2 \cdot \omega_{s_0}^{(1)} \cdot \omega_{s_1}^{(1)} \cdot \omega_{q_0}^{(2)} \cdot \omega_{q_1}^{(2)} \right\}^{1-2\sigma \cdot \rho_0}  \\
    &= \left[ C_I^{(1)} \cdot C_I^{(2)} \right]^{2(1-2\sigma\cdot \rho_0)} \cdot  \sum\limits_{s_0 = 2}^{+\infty} \left[ \omega_{s_0}^{(1)} \right]^{1-2\sigma\cdot \rho_0}\cdot  \sum\limits_{s_1 = 2}^{+\infty} \left[ \omega_{s_1}^{(1)} \right]^{1-2\sigma\cdot \rho_0}\cdot  \sum\limits_{q_0 = 1}^{+\infty} \left[ \omega_{q_0}^{(2)} \right]^{1-2\sigma\cdot \rho_0}\cdot  \sum\limits_{q_1 = 1}^{+\infty} \left[ \omega_{q_1}^{(2)} \right]^{1-2\sigma\cdot \rho_0}.
\end{align*}
Therefore 
choosing $0 < \sigma < \min\limits_{i \in \lbrace 1 , 2 \rbrace} \left( \dfrac{\sigma_i}{2\cdot \rho_0} \right)$
we conclude that all of the series above converge, leading to the exponential tails condition for $r$. \hfill  $\blacksquare$

\section{Exponential Decay of Correlations} \label{sec:expdecay}

This section is devoted to the proof  of Theorem~\ref{thm:main}.

\subsection{Exponential decay for the induced flow} \label{sec:exponentialDOCinducedflow}

We proceed to prove that the suspension flow obtained from the inducing scheme - over the base $\Dt \times \R^+$ with roof function $r$ and measure ${\nu \otimes Leb}/{ \int r \, d\nu}$ - has exponential decay of correlations with respect to its SRB measure. In Section~\ref{sec:base} we proved that $\PPP$ is a $C^2$-piecewise smooth hyperbolic skew-product on $\Dt \times \R^+$, while in Section~\ref{sec:roof function} we verified that the roof function $r$ satisfies its standard assumptions. 
Observe that if the base space $\Dt \times \R^+$ was bounded then one could directly apply Theorem~\ref{thm:AGY}, obtaining Theorem~\ref{thm:main} as a corollary of the latter.
However, this is not the case. 

Here we overcome this obstruction by showing that the unboundedness of the domain does not affect the main the arguments in proof  of Theorem~\ref{thm:AGY}, due to the fact that the unboundedness arises from unbounded stable leaves with strong contracting behaviour under the action of $\Gt$ and the exponential tails property of the roof function $r$.

Let us fix some notations. Given the roof function $r$ defined by ~\eqref{eq:aftercohom}, write  $\Sigma = \Dt \times \R^+$ for the base space, 
$$\Sigma_r = \left\{ [(x,y) , s ] \, \bigg| \, (x,y) \in \Sigma, \ 0 \le s \le r(x) \right\} \bigg / \sim$$ 
for the phase space for the flow (under the equivalence relation $[(x,y),r(x)] \sim\left[\Pt(x,y),0\right]$),
$\Pt_t$ for the suspension flow on $\Sigma_r$, defined by
\begin{align}
\label{eq:defPt}
\Pt_t [(x,y),s] = \left[ \Pt^n (x,y) , s+t-r^{(n)}(x) \right],
\end{align}
where $n$ is the unique natural number such that $r^{(n)} (x) \le t+s < r^{(n+1)}(x)$. Let 
$$
\nu_r = \dfrac{\nu \otimes Leb}{\int\limits_{\Dt \times \R^+}\!\!\!\!r \, d\nu}.
$$
be the unique $\Pt_t$-invariant SRB measure for the flow $\Pt_t$.
It is convenient to define the projection of the flow on the unstable component, hence defining the spaces
$$
\Dt_r = \left\{ (x,t) \, | \, x \in \Dt, \ 0 \le t \le r(x) \right\} / \sim,$$ 
where $(x,r(x))\sim \left(\Ft (x) , 0\right)$,
the suspension semiflow $\Ft_t$ on $\Dt_r$, locally defined for every $t \in [0 , +\infty )$ on any element $(x,s) \in \Dt_r$ as
$\Ft_t (x,s) = \left( x , s+t \right),$
for every $t\ge 0$ (accordingly with the equivalence relation) and the $\Ft_t$-invariant probability measure
$$
\widehat{\mu}_r =   \dfrac{\widehat{\mu} \otimes Leb}{\int\limits_{\Dt} r \, d\widehat{\mu}}.
$$

We need a couple of auxiliary results.

\begin{proposition}
\label{prop:ritorni}

Given a suitable $x \in \Dt, \ a \in [0 , r(x))$, and $t>0$, consider
\begin{align*}
\Psi_t (x,a) = \sup \lbrace n\ge 1 \, | \, a+t> r^{(n)}(x) \rbrace,
\end{align*}
which is the number of returns of $(x,a)\in \Dt_r$ to the base component $\Dt$ under the action of the flow $\Pt_t$ before time $t$. For any $k>1$, there exist $C , \delta>0$ such that
\begin{align}
\int_{\Dt_r} k^{-\Psi_t (x,a)}\, dx \le C \cdot e^{-\delta t}. \label{eqritorni}
\end{align}
for every $t \in [0 , +\infty)$.
\end{proposition}

\begin{proof}
    This is \cite[Lemma 8.1]{avilaGouezelyoccoz}.
\end{proof}

\begin{lemma}
\label{lemma:ritorni}
The following properties hold:
\begin{itemize}
\item[$\boldsymbol{(1)}$] For every $q \ge 1$, we have that $r \in L^q(\widehat{\mu})$.
\item[$\boldsymbol{(2)}$] Consider a suitable $x \in \Dt $, $a \in [0 , r(x) )$ and $t>0$ such that $t - (r(x)-a) > 0$, one has that
\begin{align}
\Psi_{t-(r(x)-a)} \left( \Ft (x) , 0 \right) = \Psi_t (x,a) - 1. \label{eqritorni2}
\end{align}
\end{itemize}
\end{lemma}

\begin{proof}
Fix $q \ge 1$. Since $\widehat{\mu}$ is equivalent to $\Leb$, with density bounded away from zero and infinity, in order to prove item (1) it is enough to show that $r \in L^q (Leb)$. Note that 
\begin{align*}
\int_\Dt r(x)\, dx \le \sum\limits_{n=0}^{+\infty} \int_{A_n} (n+1)^qdx
\end{align*}
where, for each $n\ge 1$, $A_n = \left\{ x \in \Dt \, | \, n < r(x) \le n+1 \right\}$. As there exists $\sigma>0$ such that $\int_\Delta e^{\sigma r(x)}dx <+\infty$ and there exists $N\ge 1$ so that $(n+1)^q \le e^{\sigma n}$ for any $n \ge N$, one gets
\begin{align*}
\int_\Dt r(x) \,dx & \le \sum\limits_{n=0}^{N-1} (n+1)^q Leb(A_n) + \sum\limits_{n=N}^{+\infty} \int_{A_n} e^{\sigma n} dx  \\
& \le \sum\limits_{n=0}^{N-1} (n+1)^q Leb(A_n) +  \int_{\Dt} e^{\sigma r(x)} dx, 
\end{align*}
which is finite, as desired.

Now, given a suitable $x \in \Dt $, $a \in [0 , r(x) )$ and $t>0$ such that $t - (r(x)-a) > 0$, take $N \ge 1$ such that $t-(r(x) -a) > r^{(N)}\left(\Ft (x)\right)$, i.e.
\begin{align*}
a+t > r(x) + r^{(N )} \left(\Ft (x) \right) = r^{(N+1 )} (x).
\end{align*}
In particular
$N+1 \in \lbrace n\ge 1 \, | \, a+t > r^{(n)}(x) \rbrace$
and $N+1 \le \Psi_t (x,a)$.
Taking the supremum over $N$ in the set above, we find that $\Psi_{t-(r(x)-a)} \left( \Ft (x) , 0 \right) \le \Psi_t (x,a)-1$.
On the other hand, consider $k \ge 1$ such that $a+t > r^{(k)}(x)$. Therefore:
\begin{align*}
a+t > r(x) + \sum\limits_{i=1}^{k-1} r \circ \Ft^i (x) = r(x) + r^{(k-1)}\left( \Ft (x) \right),
\end{align*}
which gives $a+t-r(x) > r^{(k-1)}\left(\Ft (x) \right)$, so that
\begin{align*}
k-1 \in \left\{ n\ge 1 \, | \, a+t-r(x) > r^{(n)} \left( \Ft (x) \right) \right\}.
\end{align*}
Thus, $k-1 \le \Psi_{t-(r(x)-a)} \left( \Ft (x) , 0 \right)$, i.e. $k \le \Psi_{t-(r(x)-a)} \left( \Ft (x) , 0 \right) + 1$.
Taking the supremum over $k$ in the set above, we finally find $\Psi_t (x,a) \le \Psi_{t-(r(x)-a)} \left( \Ft (x) , 0 \right) +1$. This proves item (2) and finishes the proof of the lemma.
\end{proof}

The next proposition, 
whose argument is an adaptation of the proof of \cite[Lemma 8.2]{avilaGouezelyoccoz} in our unbounded setting, shows that the mixing rates for the suspension flow can be derived from the mixing rates for the suspension semiflow. 
Let us introduce some notation. Given
a function $V: \Sigma_r \to\mathbb R$,
$t \ge 0$ and 
$(x,a) \in \Dt_r$, let us define $V_t : \Dt_r \rightarrow \R$ by
\begin{align*}
V_t (x,a) = \int_{\R^+} V \circ \widetilde{P}_t \left[ (x,y) , a \right]  \, d\nu_x(y),
\end{align*}
and the projection $\pi_r: \Sigma_r \rightarrow \Delta_r$ by
$\pi_r\left[ (x,y) , a \right]= (x,a)$, where $ \left[ (x,y) , a \right] \in \Sigma_r$.

\begin{proposition}
\label{prop:decadimento}
Let $U, V \in C_{b,*}^1(\Sigma_r)$, and assume that $\int_{\Sigma_r} U \, d\nu_r = 0$. Then, there exists $\delta > 0$ which does not depend on $U,V$, and there exists $C > 0$ such that
\begin{align*}
\left| \int_{\Sigma_r} U \cdot V \circ \widetilde{P}_{2t} \, d\nu_r - \int_{\Sigma_r} U \cdot V_t \circ \Ft_t\circ\pi_r \, d \nu_r \right| \le C \cdot e^{-\delta t},
\end{align*}
for every $t \ge 0$.


\end{proposition}

\begin{proof}
Recall that $\Ft_t \circ \pi_r = \pi_r \circ \Pt_t$. Then, using that $U$ is bounded and that $\Pt_t$ preserves $\nu_r$,
\begin{align*}
&\left| \int_{\Sigma_r} U \cdot V \circ \widetilde{P}_{2t} \, d\nu_r - \int_{\Sigma_r} U \cdot V_t \circ \Ft_t\circ\pi_r \, d \nu_r \right| = \left| \int_{\Sigma_r} U \cdot \left( V \circ \Pt_t - V_t \circ \pi_r \right) \circ \Pt_t \, d\nu_r \right| \\
&\le \lVert U \rVert_{\infty} \int_{\Sigma_r} \left| V \circ \Pt_t - V_t \circ\pi_r \right| \circ\Pt_t \, d\nu_r = \lVert U \rVert_{\infty} \int_{\Sigma_r} \left| V \circ \Pt_t - V_t \circ\pi_r \right| \, d\nu_r\\
& = \lVert U \rVert_{\infty} \int\limits_{\Sigma_r} \left| V \circ \Pt_t \left[(x,y),a \right] - V_t \circ\pi_r \left[(x,y),a \right] \right|  \, d\nu_r[(x,y),a].
\end{align*}
Note that $V_t \circ \pi_r \left[(x,y),a \right] = V_t (x,a) = \int_{\R^+} V \circ \Pt_t \left[(x,y'),a \right] \, d\nu_x (y')$, where $\nu_x$ stands for the disintegration 
of $\nu$ given by Proposition~\ref{prop:disintegrazione}.
On the other hand, since $\nu_x$ is a probability on $\R^+$ for any $x$ for which it is defined, one can also write
\begin{align*}
V \circ \Pt_t \left[(x,y),a \right] = \int_{\R^+} V \circ \Pt_t \left[(x,y),a \right] \, d\nu_x(y').
\end{align*}
Thus,
\begin{align*}
& \int\limits_{\Sigma_r} \left| V \circ \Pt_t \left[(x,y),a \right] - V_t \circ\pi_r \left[(x,y),a \right] \right| \circ\Pt_t \, d\nu_r\left[(x,y),a \right]  \\
&= \int\limits_{\Sigma_r} \left| \int_{\R^+} \left[ V \circ \Pt_t \left[(x,y),a \right] - V \circ \Pt_t \left[(x,y'),a \right] \right] \, d\nu_x(y') \right| \, d\nu_r\left[(x,y),a \right] \\
&\le \int\limits_{\Sigma_r}  \int_{\R^+} \left| V \circ \Pt_t \left[(x,y),a \right] - V \circ \Pt_t \left[(x,y'),a \right] \right| \, d\nu_x(y')  \, d\nu_r\left[(x,y),a \right].
\end{align*}

Recall that $d\nu_r \left[ (x,y) , a \right] = \dfrac{1}{\nu \times Leb \left( \Sigma_r \right)} \cdot d\nu (x,y) da$ and that 
\begin{align*}
\int_{(x,y) \in \Sigma} h(x,y) \, d\nu (x,y) = \int_{x \in \Dt} \left[ \int_{y \in \R^+} h (x,y) \, d\nu_x (y) \right] \, d\widehat{\mu}(x),
\end{align*}
for any bounded and uniformly continuous $h : \Sigma \rightarrow \R$ (see Proposition \ref{prop:disintegrazione}).
In this way, since $V \in C_{b,*}^1(\Sigma_r)$, one can decompose the integrals above as follows:
\begin{align*}
& \int\limits_{\Sigma_r}  \int_{\R^+} \left| V \circ \Pt_t \left[(x,y),a \right] - V \circ \Pt_t \left[(x,y'),a \right] \right| \, d\nu_x(y')  \, d\nu_r\left[(x,y),a \right]  \\
\\
& = \dfrac{1}{\nu \times Leb \left( \Sigma_r \right)} \cdot \int\limits_{\Sigma} \int_0^{r(x)} \int\limits_{\R^+} \left| V \circ \Pt_t \left[(x,y),a \right] - V \circ \Pt_t \left[(x,y'),a \right] \right| \, d\nu_x(y') \, da \,d\nu(x,y)\\
\\
&=\dfrac{1}{\nu \times Leb \left( \Sigma_r \right)} \cdot \int\limits_{\Dt} \int\limits_{\R^+} \int_0^{r(x)} \int\limits_{\R^+} \left| V \circ \Pt_t \left[(x,y),a \right] - V \circ \Pt_t \left[(x,y'),a \right] \right| \, d\nu_x(y')\,da\,d\nu_x(y)\,d\widehat{\mu}(x).
\end{align*}
By a change in the order of integration the previous expressions can be rewritten as
\begin{align*}
& \dfrac{1}{\nu \times Leb \left( \Sigma_r \right)} \cdot \int\limits_{\Dt} \int_0^{r(x)} \left[ \int\limits_{\R^+}  \int\limits_{\R^+} \left| V \circ \Pt_t \left[(x,y),a \right] - V \circ \Pt_t \left[(x,y'),a \right] \right| \, d\nu_x (y') \,d\nu_x (y) \right] da \,d\widehat{\mu}(x)\\
&=\dfrac{1}{\nu \times Leb \left( \Sigma_r \right)} \cdot \int\limits_{\Dt_r} \left[ \int\limits_{\R^+}  \int\limits_{\R^+} \left| V \circ \Pt_t \left[(x,y),a \right] - V \circ \Pt_t \left[(x,y'),a \right] \right| \, d\nu_x (y')\, d\nu_x (y) \right] da \,d\widehat{\mu}_r(x).
\end{align*}
Now we decompose $\Dt_r$ in the two disjoint sets
\begin{align*}
\Dt_r^1 = \lbrace (x,a) \in \Dt_r \, | \, r(x) \ge t+a \rbrace
\quad\text{and}\quad \Dt_r^2 = \lbrace (x,a) \in \Dt_r \, | \, r(x) < t+a \rbrace,
\end{align*}
so that we can split the previous integral as the sum of the following terms
\begin{itemize}
\item[] $I^1 = \int\limits_{\Dt_r^1} \left[ \int\limits_{\R^+}  \int\limits_{\R^+} \left| V \circ \Pt_t \left[(x,y),a \right] - V \circ \Pt_t \left[(x,y'),a \right] \right| \, d\nu_x (y') \,d\nu_x (y) \right] da \,d\widehat{\mu}_r(x)$.\\
\item[] $I^2 = \int\limits_{\Dt_r^2} \left[ \int\limits_{\R^+}  \int\limits_{\R^+} \left| V \circ \Pt_t \left[(x,y),a \right] - V \circ \Pt_t \left[(x,y'),a \right] \right| \, d\nu_x (y') \,d\nu_x (y) \right] da \,d\widehat{\mu}_r(x)$.
\end{itemize}

Observe that
\begin{align*}
I^1 & \le \int\limits_{\Dt_r^1} \left[ \int\limits_{\R^+} \int\limits_{\R^+} 2 \lVert V \rVert_{\infty} \right]\, d\widehat{\mu}_r(x,a) =  2 \lVert V \rVert_{\infty} \cdot \widehat{\mu}_r\left( \Dt_r^1 \right).
\end{align*}
We can assert that
\begin{align*}
(x,a) \in \Dt_r^1 \Leftrightarrow \begin{cases} 0 \le a < r(x);\\
t+a\le r(x). \end{cases} \Leftrightarrow \begin{cases} 0 \le a \le r(x)-t;\\
r(x)\ge t. \end{cases}
\end{align*}
Thus, denoting by $\lbrace  r \ge t \rbrace$ the set $\lbrace x \in \Dt \times \R^+ \, | \, r(x) \ge t \rbrace$ for short, the following holds:
\begin{align*}
\widehat{\mu}_r\left(\Dt^1_r \right) &= \dfrac{1}{\widehat{\mu} \times Leb \left ( \Dt^1_r \right)} \cdot \int\limits_{\lbrace r \ge t \rbrace } \left[ \int_0^{r(x)-t} da \right] d \widehat{\mu}(x) = \dfrac{1}{\widehat{\mu} \times Leb \left ( \Dt^1_r \right)} \cdot \int\limits_{\lbrace r \ge t \rbrace } \left[ r(x)-t \right] d \widehat{\mu}(x) \\
& \le \dfrac{1}{\widehat{\mu} \times Leb \left ( \Dt^1_r \right)} \cdot \int\limits_{\lbrace r \ge t \rbrace }  r(x) \,d \widehat{\mu}(x) = \dfrac{1}{\widehat{\mu} \times Leb \left ( \Dt^1_r \right)} \cdot \int\limits_{\Delta} r(x) \cdot \mathds{1}_{\lbrace r \ge t \rbrace} (x) \, d\widehat{\mu}(x) \\
& \le \dfrac{1}{\widehat{\mu} \times Leb \left ( \Dt^1_r \right)} \lVert r \rVert_{L^2(\widehat{\mu})} \cdot \widehat{\mu} \left( \lbrace r \ge t \rbrace \right)^{\frac{1}{2}} \le \dfrac{1}{\widehat{\mu} \times Leb \left ( \Dt^1_r \right)} \lVert r \rVert_{L^2(\widehat{\mu})} \cdot \widehat{\mu} \left( \lbrace r \ge \lfloor t \rfloor \rbrace\right)^{\frac{1}{2}}.
\end{align*}
On the one hand, Lemma \ref{lemma:ritorni} implies that $\lVert r \rVert_{L^2 \left(\,\widehat{\mu}\,\right)} < +\infty$.
On the other hand, by the exponential tails property, one deduce that
\begin{align*}
\widehat{\mu} \left( \lbrace r \ge \lfloor t \rfloor \rbrace\right) \le \widehat{C} \cdot e^{-\widehat{\delta} \lfloor t \rfloor}= \widehat{C} \cdot e^{\widehat{\delta} (t-\lfloor t \rfloor )} \cdot e^{-\widehat{\delta} t} \le \widehat{C} \cdot e^{\widehat{\delta}} \cdot e^{-\widehat{\delta} t},
\end{align*}
for some constants $\widehat{C}, \widehat{\delta}>0$.
Thus, if we call $C_1 = \dfrac{2\sqrt{e^{\widehat{\delta}} \cdot \widehat{C}}\, \lVert V \rVert_{\infty} \lVert r \rVert_{L^2 \left(\,\widehat{\mu} \, \right)}}{\widehat{\mu} \times Leb \left ( \Dt^1_r \right)} $ and $\delta_1 = \dfrac{1}{2} \widehat{\delta}$, we have that $I^1 \le C_1\cdot e^{-\delta_1 t}$.

Let us turn to the estimate for $I^2$. In this case, we have that $r(x) < t+a$, i.e. $ t > r(x)-a$, since we are working on $\Dt_r^2$. We call $d$ the metric on $\Sigma_r$ given by
\begin{align*}
d\left( [(x,y),a] , [(z,w),b]\right) = \lVert (x,y) - (z,w) \rVert_2 + |a-b|.
\end{align*}
Therefore, also employing the definition of the flow $\Pt_t$ and the mean value theorem,
\begin{align*}
&\left| V \circ \Pt_t \left[ (x,y), a \right] - V \circ \Pt_t \left[(x,y'),a \right] \right| \\
& \le \lVert DV \rVert_{\infty} \cdot d \left( \Pt_t[(x,y),a),\Pt_t[(x,y'),a]\right)  \\
& = \lVert DV \rVert_{\infty} \cdot d \left( \Pt_{t-(r(x)-a)} \left( \Pt_{r(x)-a} [(x,y),a] \right) , \Pt_{t-(r(x)-a)} \left( \Pt_{r(x)-a}[(x,y'),a]\right) \right)\\
&= \lVert DV \rVert_{\infty} \cdot d \left( \Pt_{t-(r(x)-a)} \left( \Pt (x,y) , 0 \right) , \Pt_{t-(r(x)-a)} \left( \Pt (x,y'),0 \right) \right).
\end{align*}
Let us denote by $k = \left[ f_0'(g_0(1)) \right]^2$, which we know is an estimate from below of the expansion coefficient of $\Ft$ (see Proposition \ref{prop:MarkovGibbs} item  \textbf{(2)}). Then, using again the mean value theorem and the definition of $\Psi$ from Proposition \ref{prop:ritorni}, we find that last term written above is smaller or equal than 
\begin{align*}
\lVert DV \rVert_{\infty} \cdot k^{-\Psi_{t-(r(x)-a)}\left( \Ft (x) ,0 \right)} \cdot d \left(\left[\Pt (x,y) , 0 \right] , \left[ \Pt(x,y') , 0 \right] \right).
\end{align*}
Observe that, since the terms with $\Pt$ share the same first component, the distance above reduces to $\left| \Gt_x(y) - \Gt_x (y') \right|$, which we know is always less or equal to $1$.
On the other hand, Lemma \ref{lemma:ritorni} showed that $\Psi_{t-(r(x)-a)}\left( \Ft (x) ,0 \right) = \Psi_t (x,a) -1$. Altogether we deduce that 
\begin{align*}
I^2 & \le \lVert DV \rVert_{\infty} \cdot \int\limits_{\Dt_r^2}k^{-\Psi_t (x,a) +1} \, d\widehat{\mu}_r \le \lVert DV \rVert_{\infty} \cdot k \cdot \int\limits_{\Dt_r^2}k^{-\Psi_t (x,a)} \, d\widehat{\mu}_r \\
&\le \lVert DV \rVert_{\infty} \cdot k \cdot \int\limits_{\Dt_r}k^{-\Psi_t (x,a)} \, d\widehat{\mu}_r \le C_2 \cdot e^{-\delta_2 t},
\end{align*}
for some $C_2 , \delta_2>0$ (see Proposition \ref{prop:ritorni}).
This finishes the proof of the proposition.
\end{proof}

Now we observe that, as a consequence of \cite[Theorem~7.3]{avilaGouezelyoccoz}
that establishes exponential decay of correlations for the suspension expanding semiflow one concludes that
$$
\left|\int_{\Sigma_r} U \cdot V_t \circ \Ft_t\circ\pi_r \, d \nu_r \right|
= 
\left|\int_{\Sigma_r} \overline{U} \cdot V_t \circ \Ft_t \, d \mu_r \right|
\le C e^{-\delta t} \|\overline{U}\|_{B_0} \|V_t\|_{B_1}
$$
for every $t\ge 0$. Then, the exponential decay of correlations for the flow 
$\tilde P_t$ with respect to the probability measure $\nu_r$ is a consequence of the latter together with Theorem~\ref{prop:decadimento}.


\subsection{Exponential decay for the original flow} \label{subsec:exponentialDOCoriginalflow}

In this subsection we complete the proof of 
Theorem~\ref{thm:main}. 
Let us first recall the inducing schemes previously constructed
in order to clarify the procedure.

\begin{small}
\hspace{-2.2cm}$\begin{array}{|c||c|c|c|c|}
\hline
\ & \ & \ & \ & \ \\
\ & \text{\textbf{Original map}} & \text{\textbf{First inducing}} & \text{\textbf{Second inducing}} & \text{\textbf{Third inducing}}\\[4pt]
\hline
\hline
\ & \ & \ & \ & \ \\
\   & \ & \text{1st return time on } (0,1) & \text{1st return time on } \left(g_0(1),1\right)& \hspace{-1.5cm}\text{On } \left(g_0(1),1\right)\\[4pt]
\text{\textbf{Return time}} & - & \hspace{-2cm}\text{of } f: \ \tau & \hspace{-1.9cm}\text{of } F: \ \kappa & \hspace{0cm}\text{of } \Fh: \text{ constant 2}\\
\ & \ & \ & \hspace{-0cm}\text{of } f: \ \theta = \sum\limits_{j=0}^{\kappa-1} \tau \circ F^j & \hspace{0.55cm}\text{of } F: \ \, \widetilde{\kappa}= \kappa \circ \Fh + \kappa\\
\ & \ & \ & \ & \hspace{0.5cm}\text{of } f: \ \  \widetilde{\theta} = \theta\circ \Fh + \theta\\[8pt]
\hline
\ & \ & \ & \ & \ \\
\  & f: & F = f^\tau: & \Fh = F^\kappa = f^\theta:  & \Ft = \Fh^2 = F^{\widetilde{\kappa}} = f^{\widetilde{\theta}}: \\[4pt]
\text{\textbf{First component}} & \left( \R^+ \setminus \lbrace 1 \rbrace \right) \circlearrowleft & (0,1) \circlearrowleft & \left( g_0(1),1\right) \circlearrowleft  & \left( g_0(1),1\right) \circlearrowleft \\
\ & \ & \ & \ & \ \\
\hline
\ & \ & \ & \ & \ \\
\text{\textbf{Second component}} & g_x  & G_x = g_x^\tau  & \Gh_x = G_x^\kappa=g_x^\theta & \Gt_x = \Gh_x^2 = G_x^{\widetilde{\kappa}} = g_x^{\widetilde{\theta}}\\
\ & \ & \ & \ & \ \\
\hline
\ & \ & \ & \ & \ \\
\text{\textbf{Skew product}} & \PPP = (f,g_x) & P=\PPP^\tau & \Ph=P^\kappa=\PPP^\theta & \Pt = \Ph^2 = P^{\widetilde{\kappa}} = \PPP^{\widetilde{\theta}}\\
\ & \ & \ & \ & \ \\
\hline
\end{array}$
\end{small}

The next table also collects the information concerning the suspension (semi)flow, which will be useful below.

\begin{small}
\[
\begin{array}{|c||c|c|c|c|}
\hline
 & \textbf{Original} & \textbf{First inducing} & \textbf{Second inducing} & \textbf{Third inducing} \\
\hline\hline
 & & & & \\[-8pt]
\textbf{Map} 
& \PPP 
& P 
& \Ph 
& \Pt \\
\hline
 & & & & \\[-8pt]
\textbf{Roof function} 
& \rho 
& R 
& \rh 
& r \\
\hline
 & & & & \\[-8pt]
\textbf{Suspension flow} 
& \varphi^t 
& - 
& - 
& \tilde P_t \\
\hline
\end{array}
\]
\end{small}

By construction, one can write (up to cohomology, which makes $r$ to depend only on $x$ while $\rho$ and $\mathcal P$ a priori depend on both $x$ and $y$) 
%
\begin{equation}
\label{eq:comparison:roofs}
r^{(n)} (x) = \sum\limits_{j=0}^{\widetilde{\theta}^{(n)}(x)-1} \rho \circ \PPP^j (x,y), \ \ \ \text{where} \ \ \ \widetilde{\theta}^{(n)} (x) = \sum\limits_{i=0}^{n-1} \widetilde{\theta}\circ \Ft^i,   
\end{equation}
or, alternatively (assuming without loss of generality that $\Rh$ and $\rh$ coincide, and that $\rt$ and $r$ also coincide)
\begin{equation}
\label{eq:comparison:roofs2}
r^{(n)} (x) = \sum\limits_{j=0}^{\widetilde{\kappa}^{(n)}(x)-1} R \circ P^j (x,y), \ \ \ \text{where} \ \ \ \widetilde{\kappa}^{(n)} (x) = \sum\limits_{i=0}^{n-1} \widetilde{\kappa}\circ \Ft^i
\end{equation}
for every $n\ge 1$ and every $x$.
The latter tables make simpler the task to relate the flows $\varphi^t$ and $\tilde P_t$. 
In fact,
the suspension flow $\varphi^t$ over $((\mathbb R^+\setminus\{1\})\times \mathbb R^+, \PPP,m)$
defined by ~\eqref{eq:suspension-t}
and preserving $m_\rho$
is semiconjugate to the 
 suspension flow $\tilde P_t$ on $\Sigma_r$ defined by \eqref{eq:defPt} 
and preserving $\nu_r$,
as $\Pt_t$ is an acceleration of the suspension flow $\varphi^t$, with physical time preserved through the induced roof function.
%
More precisely, 
consider the natural projection
$\Pi:\Sigma_r \rightarrow \Sigma_\rho$
given by
\[
\Pi\bigl([(x,y),s]\bigr)
=
\left[
\left(
\mathcal P^{k}(x,y),
\; s-\sum_{i=0}^{k-1}\rho\circ \mathcal P^{i}(x,y)
\right)
\right],
\]
where $k=k(x,y,s)\ge 0$ is the unique integer such that
\[
\sum_{i=0}^{k-1}\rho\circ \mathcal P^{i}(x,y)
\le s
<
\sum_{i=0}^{k}\rho\circ \mathcal P^{i}(x,y).
\]

\color{red}

\color{black}

The suspension flows $\Pt_t$ and $\varphi^t$ satisfy the semiconjugacy relation
\[
\Pi \circ \Pt_t
=
\varphi^t \circ \Pi
\qquad\text{for all } t\ge0.
\]
and $\Pi_*\nu_r=m_\rho$.

\medskip
Hence, given the bounded observables 
$u,v:\Sigma_\rho \rightarrow \mathbb R$
on the suspension space over $P$, consider their lifts to $\Sigma_r$ by
\[
U=u\circ\Pi
\quad\text{and}\quad
V=v\circ\Pi.
\]
Using the semiconjugacy relation $\Pi\circ\Pt_t=\varphi^t\circ\Pi$ and the fact that
$\Pi_*\nu_r = m_\rho,$
we obtain
\[
\begin{aligned}
\int_{\Sigma_r} (V\circ\Pt_t)\,U \, d\nu_r
=
\int_{\Sigma_r} (v\circ\varphi^t\circ\Pi)(u\circ\Pi)\, d\nu_r 
=
\int_{\Sigma_\rho} (v\circ\varphi^t)\,u \, dm_\rho.
\end{aligned}
\]
Thus the correlations functions for both supension flows coincide, and the original flow has exponential decay of correlation for all observables $u,v$ such that $u\circ \Pi, v\circ \Pi \in C^1_{b,*}(\Sigma_r)$.


\section{Geodesic flow on the modular surface} \label{sec:geodesic}

Here we prove Theorem~\ref{thm:geodesic} on 
 the geodesic flow on the modular surface as a consequence of Theorem~\ref{thm:main}, thus obtaining a purely dynamical proof  of its exponential decay of correlations with respect to the Liouville measure.

From Bonanno et al \cite{bonanno_delvigna_isola}, we know that $\PPP$ preserves the measure $m$ absolutely continuous with respect to Lebesgue measure with density 
$$
\dfrac{dm}{dLeb}(x,y) = \dfrac{1}{(x+y)^2}.
$$

It is simple to check, recalling ~\eqref{eq:deffuncgeo}, that 
$$
f'_0(x) = \dfrac{1}{(1-x)^2} \quad\text{and}\quad g_0'(y) = \dfrac{1}{(1+y)^2}
$$
while $f_1'= g_1' \equiv 1$. Moreover, $\rho(x,y)=\dfrac{1}{2} \ln \left[ \dfrac{x}{y} \cdot \dfrac{g_x(y)}{f(x)}\cdot \dfrac{f'(x)}{\partial_y g_x(y)}\right]$ fits the required form for the roof function roof function.

\begin{proposition}
\label{prop:gfAAdlerB}
Properties \textbf{(A)} and \textbf{(B)} are satisfied.
\end{proposition}

\begin{proof}

It is straightforward to check that $f_0(x) = \dfrac{x}{1-x}$ satisfies conditions \textbf{(A1)}-\textbf{(A4)}, by noting that
\begin{align*}
f_0'(x) = \dfrac{1}{(1-x)^2}, \ \ \ \ \ \ \ \ \text{and} \ \ \ \ \ \ \ \  f_0''(x) = \dfrac{2}{(1-x)^3}.
\end{align*}
As for properties \textbf{(A5)}-\textbf{(A6)}, just observe that 
$\dfrac{f_0''(x)}{f_0'(x)^2} = 2(1-x)$, which is monotone decreasing, hence
$\sup\limits_{x \in (0,1)} \dfrac{f_0''(x)}{f_0'(x)^2} = 2<+\infty$.

Let us prove that conditions \textbf{(B)}
are also satisfied. 
Consider the sequence $\left\{ \omega_n \right\}_{n \ge 1}$ with  $\omega_n= \dfrac{1}{n^2}$ for every ${n \ge 1}$, and define: $\omega_n^{(1)} = \omega_{n-1}$, for $n \ge 2$; and $\omega_n^{(2)} = \omega_{n+1}$, for $n\ge 1$. Note that $\left\{ \omega_n \right\}_{n \ge 1}$ is monotone decreasing  in $(0,1)$, tends to zero as $n$ tends to $+\infty$, and, $\sum\limits_{n=1}^{+\infty} \left( \dfrac{1}{n^2} \right)^{1-\sigma} < +\infty$ for any $\sigma \in 
\left( 0 , \dfrac{1}{2} \right)$. 
Moreover, as
$$
g_0'(n-1) = \dfrac{1}{1+(n-1)^2} \le \dfrac{1}{(n-1)^2} = \omega_n^{(1)}
$$ 
for every $n\ge 2$,
the first condition of \textbf{(B)} is fulfilled. 

In order to prove the second condition, we prove by induction that $g_0^j(1) = \dfrac{1}{j+1}$, for any $j \ge 1$.  In fact, it is clear that $g_0(1) = \dfrac{1}{2}$, hence it holds for $j=1$.
Furthermore, assume that the formula holds for $j \ge 2$ one concludes that 
$$
g_0^{j+1}(1) = g_0(g_0^j(1)) = \dfrac{g_0'(1)}{1+g_0^j(1)} = \dfrac{\dfrac{1}{j+1}}{1+\dfrac{1}{j+1}} = \dfrac{1}{j+2},
$$ 
which proves it. In consequence,
$g_0'(g_0^j(1))=\dfrac{1}{\left(1+\dfrac{1}{j+1}\right)^2} = \left( \dfrac{j+1}{j+2} \right)^2$, which implies that, for any $n \ge 1$:
\begin{align*}
\prod\limits_{j=1}^{n-1} g_0'(g_0^j(1)) = \prod\limits_{j=1}^{n-1} \left( \dfrac{j+1}{j+2} \right)^2 = \left( \dfrac{2}{n+1} \right)^2 = \dfrac{4}{(n+1)^2} =4 \omega_n^{(2)}.
\end{align*}
This proves that the second condition of 
\textbf{(B)} is satisfied, and finishes the proof  of the lemma.
\end{proof}

We will need the following well known result.

\begin{lemma}
\label{lemma:gfergodic}
Let $\varphi^t$ be a suspension semiflow  
over the base map $(X, T)$ preserving a with $T$-invariant measure $\mu$, and with roof function $\tau$. If the $\varphi^t$-invariant measure $\mu \otimes Leb$ is ergodic then $\mu$ is ergodic.
\end{lemma}

\begin{proof}
This lemma can be derived as straighforward consequence of the ergodic theorem, and for that reason we shall omit it.
\end{proof}

\begin{proposition}
\label{prop:gfCD}

Properties \textbf{(C)} and \textbf{(D)} are satisfied.
\end{proposition}

\begin{proof}
We know that $m$ is absolutely continuous with respect to $Leb$, and that it is $\PPP$-invariant. From \cite{bonanno_delvigna_isola}, Appendix B, we have that the suspension flow with base $(\PPP,m)$ and roof function $\rho$ is isomorphic to the geodesic flow on the modular surface, which is known to be ergodic. Hence, the suspension flow is ergodic, and Lemma \ref{lemma:gfergodic} implies that also the base $(\PPP , m)$ is.\\
Now, consider a bounded interval $I=(a,b) \subset \R^+$ with $a >0$. Then, since $dm = \dfrac{1}{(x+y)^2} dLeb$:
\begin{align*}
m (I \times \R^+ ) = \int_a^b \int_0^{+\infty} \dfrac{1}{(x+y)^2} \, dy \, dx= \int_a^b \left[ - \dfrac{1}{x+y} \right]_0^{+\infty} \, dx = \int_a^b \dfrac{1}{x} \, dx = \ln \dfrac{b}{a} <+\infty.
\end{align*}
Therefore, condition \textbf{(C)} is satisfied.\\
Finally the integrability condition \textbf{(D)}, namely that 
$$
\int\limits_{\R^+ \times \R^+} \rho \, dm < +\infty
$$
was proven in \cite[Proposition B.2]{bonanno_delvigna_isola}.
\end{proof}



Since all the required conditions \textbf{(A1) - (A6)}, \textbf{(B)}, and \textbf{(C)} are satisfied, Theorem \ref{thm:main} implies exponential decay of correlations for the suspension flow model of the geodesic flow on the modular surface, with respect to $m_\rho$ and $\CCC_*$ observables. Thus, Theorem \ref{thm:geodesic} is just a consequence of the construction of the regularity class $\widetilde{\CCC}_*$ of observables in \eqref{def:classofobservables2}, through the isomorphism $\Gamma$ defined by the cross-section considered in \cite{bonanno_delvigna_isola}. That is, we finally provided a purely dynamical proof for the decay of correlations of the geodesic flow on the modular surface, with respect to its Liouville measure and the class of observables $\widetilde{\CCC}_*$.
\color{black}

\appendix

\section{proof  of Proposition \ref{prop:truccodibowen}} \label{sec:appendixa}

Our first goal is showing that the series
\begin{align*}
u(x,y) = \sum\limits_{i=0}^{+\infty} \left[ \rt \left( \Pt^i (x,y) \right) - \rt \left( \Pt^i (x,y') \right) \right]
\end{align*}

converges for any suitable $(x,y) \in \Dt \times \R^+$, where $y' \in \R^+$ is fixed. The strategy is straightforward: for $i\ge 2$, we find a bound for $\left| \rt \left( \Pt^i (x,y) \right) - \rt \left( \Pt^i (x,y') \right) \right|$ from above with a term depending on $i$ whose series is finite. We make use of the Mean Value Theorem, some uniform estimates for the derivatives of $\rt$ and the contracting behaviour of $\Gt$.\\
For the next discussion, we can fix $i \ge 2$ and $(x,y) \in \Dt \times \R^+$.\\
Given two sequences $\lbrace s_i \rbrace_{i \ge 1} \subset \N^* \setminus \lbrace 1 \rbrace$ and $\lbrace q_i \rbrace_{i \ge 1} \subset \N^*$, and $l \ge 1$, set:
\begin{itemize}
\item $\Isq{l-1} =  \Gh_{J_{s_{l-1}}^{q_{l-1}}} \circ \dots \circ \Gh_{J_{s_0}^{q_0}} ( \R^+ )$.
\item $I_{s_{-1}}^{q_{-1}} = \R^+$.
\end{itemize}
Moreover, for any interval $I \subset \R^+$, we denote its length by $|I|$.\\
Since we already fixed $x$ and $y$, in order to prevent confusion it is convenient to rename $(\xi , \eta )$ the generic pair of variables in $\Dt \times \R^+$. For consistency, we also call $\partial_\eta$ the partial derivative operator along the second component.\\
Consider $j \in \lbrace 0 , \dots , 2i-1 \rbrace$ and suppose that $\Fh^j (x) \in J_{s_j}^{q_j}$, with $s_j , q_j \ge 1$, $s_j \ge 2$.

\begin{lemma}
\label{lemma:stima1}
\begin{align*}
\left| \rt \left( \Pt^i (x,y) \right) - \rt \left( \Pt^i (x,y') \right) \right| \le \sup\limits_{\eta \in \Isq{2i-1}} \left| \partial_y \rt \left(\Ft^i (x) , \eta \right) \right| \cdot |\Isq{2i-1}|.
\end{align*}
\end{lemma}

\begin{proof}

We apply the Mean Value Theorem to the function $\eta \mapsto \rt \left( \Ft^i(x) , \eta \right)$, finding $\overline{\eta}$ between $\Gt_x^i(y)$ and $\Gt_x^i(y')$ such that
\begin{align*}
\left| \rt \left( \Pt^i (x,y) \right) - \rt \left( \Pt^i (x,y') \right) \right| = \left| \rt \left( \Ft^i (x),\Gt_x^i(y) \right) - \rt \left( \Ft^i (x),\Gt_x^i(y') \right) \right| = \left| \partial_y \rt \left(\Ft^i (x) , \overline{\eta} \right) \right| \cdot \left| \Gt_x^i(y) - \Gt_x^i(y') \right|.
\end{align*}

By the definitions, $\Gt_x^i(y), \Gt_x^i(y')$, and $\overline{\eta}$ all belong to $\Isq{2i-1}$, so that
\begin{align*}
\left| \partial_y \rt \left(\Ft^i (x) , \overline{\eta} \right) \right| \cdot \left| \Gt_x^i(y) - \Gt_x^i(y') \right| \le \sup\limits_{\eta \in \Isq{2i-1}} \left| \partial_y \rt \left(\Ft^i (x) , \eta \right) \right| \cdot |\Isq{2i-1}|.
\end{align*}
\end{proof}

Thus, we now look for estimates for $\partial_\eta \rt$ and $|\Isq{2i-1}|$. Let us start from the first one: it is convenient to focus on $\rh$, first.\\
Consider $(\xi , \eta ) \in \Dt \times \R^+$. The derivative along $\eta$ of $\rh (\xi,\eta ) = \rho_0 \ln \left[\dfrac{\Fh'(\xi)}{\Gh_\xi'(\eta)}\right]$ reads:
\begin{align*}
\partial_\eta \rh (\xi,\eta) = \rho_0 \cdot \left[\dfrac{\Fh'(\xi)}{\Gh_\xi'(\eta)}\right]^{-1} \cdot \left[ -\dfrac{\Fh'(\xi) \cdot \Gh_\xi''(\eta)}{\Gh_\xi'(\eta)^2}\right] = \rho_0\cdot \dfrac{-\Gh_\xi''(\eta)}{\Gh_\xi'(\eta)}.
\end{align*}

Therefore, if $\xi \in J_s^q$, then
\begin{align*}
\partial_\eta \rh (\xi,\eta) = \widehat{C}\cdot \dfrac{-\Gh_{\xi}''(\eta)}{\Gh_{\xi}'(\eta)} = \widehat{C}\cdot \dfrac{-\Gh_{J_s^q}''(\eta)}{\Gh_{J_s^q}'(\eta)},
\end{align*}
meaning that $\partial_\eta \rh|_{J_s^q \times \R^+} = \partial_\eta \rh_s^q|_{J_s^q \times \R^+}$ only depends on $\eta$.\\
Thus, with a slight abuse of notations, for simplicity we omit $\xi$ and write $\partial_\eta \rh_s^q (\eta)$.

\begin{proposition}
\label{prop:stimaderivata}
\begin{align*}
C_T = \sup\limits_{ (\xi,\eta)\in \Delta \times \R^+} \partial_\eta \rh (\xi,\eta) = \sup\limits_{\substack{s \ge 2 \\ q \ge 1}} \left[ \sup_{\eta>0} \partial_\eta \rh_s^q (\eta) \right] < +\infty.
\end{align*}
\end{proposition}

\begin{proof}

The first equivalence is immediate. Let us show that $C_T$ is finite.\\
Fix $s,q \ge 1$, $s\ge 2$, and consider $(\xi,\eta ) \in J_s^q \times \R^+$. Straightforward calculations yield that:\\
\begin{itemize}
\item $\Gh_\xi'(\eta) = \Gh_{J_s^q}'(\eta) = \prod\limits_{i=0}^{q-1}g_0'\left( G_\xi^i(\eta)\right) = \left[ \prod\limits_{i=0}^{q-2}g_0' \left( g_0^i \circ g_1^{s-1} \circ g_0 (\eta) \right) \right]\cdot g_0'(\eta)$.\\
\item $\Gh_\xi''(\eta) = \sum\limits_{l=0}^{q-1} \left\{ \left[ \prod\limits_{i=0}^{l-1} g_0' \left( G_\xi^i (\eta) \right)^2 \right] \cdot g_0''\left( G_\xi^l (\eta) \right) \cdot \left[\prod\limits_{i=l+1}^{q-1} g_0' \left( G_\xi^i(\eta)\right)\right]\right\}$.
\end{itemize}

Let us employ these expressions to estimate the formula $\partial_\eta \rh_s^q(\eta) = \widehat{C}\cdot \dfrac{-\Gh_{J_s^q}''(\eta)}{\Gh_{J_s^q}'(\eta)}$.
\begin{align*}
\dfrac{-\Gh_{\xi}''(\eta)}{\Gh_{\xi}'(\eta)} & = \sum\limits_{l=0}^{q-1} \dfrac{\left[ \prod\limits_{i=0}^{l-1} g_0' \left( G_\xi^i (\eta) \right)^2 \right] \cdot \left[ - g_0''\left( G_\xi^l (\eta) \right)\right] \cdot \left[\prod\limits_{i=l+1}^{q-1} g_0' \left( G_\xi^i(\eta)\right)\right]}{\prod\limits_{i=0}^{q-1}g_0'\left( G_\xi^i(\eta)\right)}= \\
& = \sum\limits_{l=0}^{q-1} \left[ \prod\limits_{i=0}^{l-1} g_0' \left( G_{\xi}^i (\eta )\right) \right] \cdot \left[ - \dfrac{g_0'' \left( G_{\xi}^l (\eta ) \right) }{g_0'\left( G_{\xi}^l (\eta )\right)}\right].
\end{align*}
Observe that $g_0''(\eta) = -\dfrac{f_0''(g_0(\eta))g_0'(\eta)}{\left[f_0'(g_0(\eta))\right]^2}$, so that
\begin{align*}
\left[ - \dfrac{g_0'' \left( G_{\xi}^l (\eta ) \right) }{g_0'\left( G_{\xi}^l (\eta )\right)}\right] = \dfrac{f_0''\left( g_0 \circ G_\xi^l(\eta)\right) \cdot g_0'\left(G_\xi^l(\eta)\right)}{\left[f_0'\left(g_0\circ G_\xi^l (\eta)\right)\right]^2\cdot g_0'\left(G_\xi^l(\eta)\right)} = \dfrac{f_0''\left( g_0 \circ G_\xi^l(\eta)\right)}{\left[f_0'\left(g_0\circ G_\xi^l (\eta)\right)\right]^2} \le C_A,
\end{align*}

where we used the Adler's property of $f_0$. Therefore,
\begin{align*}
\dfrac{-\Gh_{\xi}''(\eta)}{\Gh_{\xi}'(\eta)}  \le C_A \cdot \sum\limits_{l=0}^{q-1} \left[ \prod\limits_{i=0}^{l-1} g_0' \left( G_{\xi}^i (\eta )\right) \right] = C_A \cdot \sum\limits_{l=0}^{q-1} \left[ g_0'(\eta) \cdot \prod\limits_{i=1}^{l-1} g_0' \left( G_{\xi}^i (\eta )\right) \right]<  C_A \cdot \sum\limits_{l=0}^{q-1} \left[\prod\limits_{i=1}^{l-1} g_0' \left( G_{\xi}^i (\eta )\right) \right],
\end{align*}
where in the last inequality we simply used that $g_0'(\eta) < 1$.\\
Since $\xi \in J_s^q$, if $l \le q-1$ and $j \in \lbrace 1 , \dots , l-1 \rbrace$, then $G_\xi^i (\eta) = g_0^{j-1}\circ g_1^{s-1} \circ g_0 (\eta)$.\\
Now, $g_1^{s-1}\circ g_0 (\eta) > s-1 \ge 1$, so that $g_0^{j-1}\circ g_1^{s-1} \circ g_0 (\eta) > g_0^{j-1}(1) > g_0^j(1)$, because $g_0$ is increasing. But $g_0'$ is decreasing, implying that $g_0'\left(G_\xi^i(\eta)\right) = g_0' (g_0^{i-1}\circ g_1^{s-1} \circ g_0 (\eta)) < g_0'(g_0^i(1))$. Therefore:
\begin{align*}
\dfrac{-\Gh_{\xi}''(\eta)}{\Gh_{\xi}'(\eta)} < C_A \cdot \sum\limits_{l=0}^{q-1} \left[\prod\limits_{i=1}^{l-1} g_0' \left( g_0^i(1)\right) \right] = C_A \cdot \left\{ 1 + \sum\limits_{l=1}^{q-1}\left[\prod\limits_{i=1}^{l-1} g_0' \left( g_0^i(1)\right) \right] \right\} .
\end{align*}
Recall from assumption \textbf{(B)} that:
\begin{itemize}
\item $\prod\limits_{i=1}^{l-1} g_0'(g_0^i (1)) < C_I^{(2)} \cdot \omega_l^{(2)}$, with $\omega_l^{(2)} \in (0,1)$.
\item $\sum\limits_{l=1}^{+\infty} \left[ \omega_l^{(2)} \right]^{1-\sigma_2} < +\infty$, 
where $\sigma_2 \in (0,1)$.\\
\end{itemize}
In particular, $\sum\limits_{l=1}^{+\infty} \omega_l^{(2)} < +\infty$, so that, in the end:
\begin{align*}
\partial_\eta \rh_s^q (\eta) = \rho_0 \cdot \dfrac{-\Gh_{\xi}''(\eta)}{\Gh_{\xi}'(\eta)} < \rho_0 \cdot C_A \cdot \left[ 1 + \sum\limits_{l=1}^{q-1} \omega_l^{(2)} \right] < \rho_0 \cdot C_A \cdot \left[ 1 + \sum\limits_{l=1}^{+\infty} \omega_l^{(2)} \right] < +\infty.
\end{align*}
Since the last inequality does not depend on $s,q$ or $\eta$, the proof  is concluded.
\end{proof}

From the proof  of Proposition \ref{prop:stimaderivata}, we learn that $\partial_\eta \rh>0$, i.e. $\rh$ is increasing along the second component.

\begin{corollary}
\label{cor:stimaderivata}
\begin{align*}
\sup\limits_{(\xi,\eta)\in\Dt\times\R^+} \partial_\eta \rt (\xi,\eta) < 2C_T.
\end{align*}
\end{corollary}

\begin{proof}

Take $(\xi , \eta ) \in \Dt \times \R^+$ and recall that $\rt = \rh + \rh \circ \Ph$. Then, Proposition \ref{prop:stimaderivata} implies that:
\begin{align*}
\partial_\eta \rt (\xi , \eta ) & = \partial_\eta\rh (\xi , \eta) + \partial_\eta \left[ \rh \left(\Fh(x),\Gh_\xi (\eta)\right) \right] = \partial_\eta \rh(\xi,\eta) + \partial_\eta\rh\left(\Fh(x),\Gh_\xi(\eta)\right)\cdot\Gh_\xi'(\eta)< 2C_T,
\end{align*}
since $\Gh_\xi'(\eta)<1$.\\
The estimate is independent of $(\xi,\eta)$, thus proving the statement.
\end{proof}

Now we look at the behaviour of $\left| \Isq{2i-1} \right|$.

\begin{lemma}
\label{lemma:lunghezzaintervalli}
\begin{align*}
\left| \Isq{2i-1}\right| \le  g_0' (1)^{i-1}.
\end{align*}
\end{lemma}

\begin{proof}

By definition, $\Isq{2i-1} = \Gt_{\Ft(x)}^{i-1} \left( \Gt_{x} (\R^+) \right) = \Gt_{\Ft(x)}^{i-1} \left( I_{s_0s_1}^{q_0q_1} \right)$.\\
Observe that the Mean Value Theorem provides an $\overline{\eta} \in I_{s_0s_1}^{q_0q_1}$ such that
\begin{align*}
\left| \Gt_{\Ft(x)} (I_{s_0s_1}^{q_0q_1}) \right| = \Gt_{\Ft(x)}'(\overline{\eta}) \cdot \left|I_{s_0s_1}^{q_0q_1} \right| \le g_0'(1) \cdot \left| I_{s_0s_1}^{q_0q_1} \right| \le g_0'(1),
\end{align*}
where we used that $0< \Gt_{J_{s_0s_1}^{q_0q_1}}' (\overline{\eta}) \le g_0'(1)$ and $|I_{s_0s_1}^{q_0q_1}| \le 1$.\\
A simple induction argument proves that $\left| \Gt_{\Ft(x)}^{i-1} (I_{s_0s_1}^{q_0q_1}) \right|   \le g_0'(1)^{i-1}$, which gives the wanted inequality.
\end{proof}

Now we have all the necessary elements for proving Proposition \ref{prop:truccodibowen}.

\begin{proof}[Proof  of Proposition \ref{prop:truccodibowen}]

Consider $x,y,y'$ as before. We show that $$
u(x,u) = \sum\limits_{i=0}^{+\infty} \left[ \rt \left( \Pt^i (x,y) \right) - \rt \left( \Pt^i (x,y') \right) \right]
$$ is absolutely convergent. Recall that $\partial_\eta \rt >0$, and note that:
\begin{itemize}
\item For $i=0$, we have
\begin{align*}
|\rt (x,y) - \rt(x,y') | \le \sup\limits_{\eta>0} \partial_\eta \rt (\eta) \cdot |y-y'| \le 2C_T |y-y'|.
\end{align*}
\item For $i=1$, similarly, we have
\begin{align*}
|\rt \left( \Pt (x,y) \right) - \rt \left( \Pt (x,y')\right) | \le 2C_T \cdot | I_{s_0s_1}^{q_0q_1}| \le 2C_T=\dfrac{2C_T}{g_0'(1)}\cdot g_0'(1).
\end{align*}
\item For $i \ge 2$, we found
\begin{align*}
\left|\rt \left( \Pt^i(x,y)\right) - \rt\left( \Pt^i (x,y')\right) \right| \le 2C_T \cdot |\Isq{2i-1}| \le 2C_T\cdot g_0'(1)^{i-1} = \dfrac{2C_T}{g_0'(1)} \cdot g_0'(1)^i.
\end{align*}
\end{itemize}
Therefore, since $g_0'(1) < 1$,
\begin{align*}
\sum\limits_{i=0}^{+\infty} \left| \rt \left( \Pt^i (x,y) \right) - \rt \left( \Pt^i (x,y') \right) \right| \le 2C_T \cdot\left\{ |y-y'| + \dfrac{2C_T}{g_0'(1)}\cdot\sum\limits_{i=1}^{+\infty} g_0'(1)^j \right\} < +\infty,
\end{align*}
implying that $u(x,y)$ is well-defined on $\Dt \times \R^+$.\\
\\
As for the second part of the Proposition, observe that:
\begin{align*}
 u(x,y) & - u\left(\Pt (x,y)\right)  \\
& = \sum\limits_{i=0}^{+\infty} \left[ \rt \left( \Pt^i (x,y) \right) - \rt \left( \Pt^i (x , y' ) \right) \right] - \sum\limits_{i=0}^{+\infty} \left[ \rt \left( \Pt^{i+1} (x,y) \right) - \rt \left( \Pt^{i+1} (x , y' ) \right) \right] \\
& = \rt (x,y ) - \rt (x,y' ) + \sum\limits_{i=1}^{+\infty} \left[ \rt \left( \Pt^i (x,y) \right) - \rt \left( \Pt^i (x,y') \right)- \rt \left( \Pt^i (x,y) \right) + \rt \left( \Pt^i (x,y') \right) \right]\\
& = \rt (x,y) - \rt (x,y').
\end{align*}

That is: $\rt(x,y) = \rt (x,y' ) + u(x,y) - u \left(\Pt (x,y) \right)$, i.e. $u$ is the cohomology function that makes $(x,y) \mapsto \rt(x,y)$ cohomologous to $(x,y) \mapsto \rt (x , y' )$.
\end{proof}

\section{proof  of Lemma \ref{lemma:ordini}} \label{sec:appendixb}

Observe that bounded distortion for $\Ft$ implies that, for $x > x'$ in some $J_{s_0s_1}^{q_0q_1}$, we have:
\begin{align}
\Ft'(x) \le e^{\widetilde{C}_A \cdot |\Ft (x ) - \Ft (x')|}\cdot \Ft'(x') \le e^{\widetilde{C}_A \cdot \left|\Dt\right|}\cdot \Ft'(x'). \label{boundeddistortionFtilde}
\end{align}
Thus, if we call $\widetilde{C}_D = e^{\widetilde{C}_A \cdot \left|\Dt\right|}$, whenever we invoke ``bounded distortion'' for $\Ft$ with respect to a couple of points $x>x'$ in the same partition element $J_{s_0s_1}^{q_0q_1}$, we simply apply $\Ft'(x) \le \widetilde{C}_D \cdot \Ft'(x')$.\\
It is easy to check that the same holds for $\Fh$ (with the same constants).\\
Let us proceed with the proof  of Lemma \ref{lemma:ordini}.

\begin{proof}[Proof  of Lemma \ref{lemma:ordini}, point $\boldsymbol{(1)}$]

Since $\Ft$ has a $C^2$ extension to $\overline{J_{s_0s_1}^{q_0q_1}}$, by the Integral Mean Value Theorem we find $\xi \in J_{s_0s_1}^{q_0q_1}$ such that
\begin{align*}
|J_{s_0s_1}^{q_0q_1}|^{-1} \cdot \int_{c_{s_0s_1}^{q_0q_1}}^{d_{s_0s_1}^{q_0q_1}} \Ft'(x) dx = \Ft'(\xi ).
\end{align*}
On the other hand, by construction,
\begin{align*}
\int_{c_{s_0s_1}^{q_0q_1}}^{d_{s_0s_1}^{q_0q_1}} \Ft'(x) \, dx = \Ft (d_{s_0s_1}^{q_0q_1} ) - \Ft (c_{s_0s_1}^{q_0q_1}) = 1 - g_0(1).
\end{align*}
Hence, we find: $\Ft'(\xi) = [1-g_0(1)] \cdot |J_{s_0s_1}^{q_0q_1}|^{-1}$.
Now, since $\xi < d_{s_0s_1}^{q_0q_1}$, by monotonicity of $\Ft'$ and bounded distortion:
\begin{itemize}
\item $[1-g_0(1)] \cdot |J_{s_0s_1}^{q_0q_1}|^{-1} = \Ft'(\xi) \le \Ft'(d_{s_0s_1}^{q_0q_1} )$.
\item $\Ft'( d_{s_0s_1}^{q_0q_1}) \le \widetilde{C}_D \cdot \Ft'(\xi) = \widetilde{C}_D \cdot [1-g_0(1)] \cdot |J_{s_0s_1}^{q_0q_1}|^{-1}$.
\end{itemize}
Thus, the statement holds if we set, for instance, $\widetilde{C} = \max \lbrace \widetilde{C}_D \cdot [1-g_0(1)] , [1-g_0(1)]^{-1} \rbrace$.
\end{proof}

Remember that, by notation, if we have a product sign with decreasing index, we consider it equal to $1$.

\begin{proof}[Proof  of Lemma \ref{lemma:ordini}, point $\boldsymbol{(2)}$]

The proof  heavily relies on the following representation formulas:
\begin{itemize}
\item $\Fh'(x) = \left[ \prod\limits_{i=0}^{q-2} f_0' (f_0^i \circ f_i^{s-1} \circ f_0 (x)) \right] \cdot f_0'(x)$, for any $x \in J_s^q$.
\item $\Fh|_{J_s^q} = f_0^{q-1}\circ f_1^{s-1}\circ f_0|_{J_s^q}$. 
\item $d_s^q = \Fh|_{J_s^q}^{-1} (1) = g_0 \circ g_1^{s-1} \circ g_0^{q-1}(1)$.
\item $d_{s_0s_1}^{q_0q_1} = \Ft|_{J_{s_0s_1}^{q_0q_1}}^{-1} (1) =  g_0\circ g_1^{s_0-1}\circ g_0^{q_0-1} \circ g_0 \circ g_0^{s_1-1} \circ g_0^{q_1-1}(1)$.
\end{itemize}

Now:
\begin{align*}
\Ft'(d_{s_0s_1}^{q_0q_1} ) = \Fh' ( \Fh(d_{s_0s_1}^{q_0q_1}))\cdot \Fh'(d_{s_0s_1}^{q_0q_1}) = \Fh'(d_{s_1}^{q_1}) \cdot \Fh'(d_{s_0s_1}^{q_0q_1}).
\end{align*}
Observe that, if $q_1 \ge 2$ and $i \in \lbrace 0 , \dots , q_1 -2 \rbrace$, then:
\begin{align*}
f_0^i \circ f_i^{s_1-1} \circ f_0 (d_{s_1}^{q_1}) & = f_0^i \circ f_i^{s_1-1} \circ f_0 \circ g_0 \circ g_0^{s_1-1} \circ g_0^{q_1-1}(1) = g_0^{q_1-i-1}(1).
\end{align*}
Therefore, $\Fh'(d_{s_1}^{q_1}) = \left[ \prod\limits_{i=0}^{q_1-2} f_0' (g_0^{q_1-i-1}(1)) \right] \cdot f_0'(d_{s_1}^{q_1})$, and similar computations for $\Fh'(d_{s_0s_1}^{q_0q_1})$ lead to:
\begin{align*}
\Ft'(d_{s_0s_1}^{q_0q_1} ) &= \left[ \prod\limits_{i=0}^{q_1-2} f_0' (g_0^{q_1-i-1}(1)) \right] \cdot f_0'(d_{s_1}^{q_1}) \cdot \left[ \prod\limits_{i=0}^{q_0-2} f_0' (g_0^{q_0-i-1}(d_{s_1}^{q_1})) \right] \cdot f_0'(d_{s_0s_1}^{q_0q_1})=\\
& = \left[ \prod\limits_{i=1}^{q_1-1} f_0' (g_0^{i}(1)) \right] \cdot f_0'(d_{s_1}^{q_1}) \cdot \left[ \prod\limits_{i=1}^{q_0-1} f_0' (g_0^{i}(d_{s_1}^{q_1})) \right] \cdot f_0'(d_{s_0s_1}^{q_0q_1}).
\end{align*}

We also have a representation formula for $\Gh'_{J_s^q}$, which implies one for $\Gt_{J_{s_0s_1}^{q_0q_1}}'(y')$:
\begin{align*}
\Gt_{J_{s_0s_1}^{q_0q_1}}'(y') = \left[ \prod\limits_{i=0}^{q_1-2}g_0' \left( g_0^i \circ g_1^{s_1-1} \circ g_0 \left(\Gh_{J_{s_0}^{q_0}}(y')\right) \right) \right]\cdot g_0'\left(\Gh_{J_{s_0}^{q_0}}(y')\right) \cdot \left[ \prod\limits_{i=0}^{q_0-2}g_0' \left( g_0^i \circ g_1^{s_0-1} \circ g_0 (y') \right) \right]\cdot g_0'(y').
\end{align*}

Observe that $\dfrac{1}{g_0'(y)} = f_0'(g_0 (y))$, so that $\dfrac{1}{\Gt_{J_{s_0s_1}^{q_0q_1}}'(y')}$ reads:
\begin{align*}
\left[ \prod\limits_{i=1}^{q_1-1}f_0' \left( g_0^{i} \circ g_1^{s_1-1} \circ g_0 \left(\Gh_{J_{s_0}^{q_0}}(y')\right) \right) \right]\cdot f_0'\left(g_0 \left(\Gh_{J_{s_0}^{q_0}}(y')\right)\right) \cdot \left[ \prod\limits_{i=0}^{q_0-1}f_0' \left( g_0^{i} \circ g_1^{s_0-1} \circ g_0 (y') \right) \right]\cdot f_0'(g_0(y')),
\end{align*}
where we also rearranged the indices.\\
Therefore, we need to compare the following expressions for, respectively, $\Ft'(d_{s_0s_1}^{q_0q_1})$ and $\dfrac{1}{\Gt_{J_{s_0s_1}^{q_0q_1}}'(y')}$:
\begin{itemize}
\item $\left[ \prod\limits_{i=1}^{q_1-1} f_0' (g_0^{i}(1)) \right] \cdot f_0'(d_{s_1}^{q_1}) \cdot \left[ \prod\limits_{i=1}^{q_0-1} f_0' (g_0^{i}(d_{s_1}^{q_1})) \right] \cdot f_0'(d_{s_0s_1}^{q_0q_1})$.
\item $\left[ \prod\limits_{i=1}^{q_1-1}f_0' \left( g_0^{i} \circ g_1^{s_1-1} \circ g_0 \left(\Gh_{J_{s_0}^{q_0}}(y')\right) \right) \right]\cdot f_0'\left(g_0 \left(\Gh_{J_{s_0}^{q_0}}(y')\right)\right) \cdot \left[ \prod\limits_{i=1}^{q_0-1}f_0' \left( g_0^{i} \circ g_1^{s_0-1} \circ g_0 (y') \right) \right]\cdot f_0'(g_0(y'))$.
\end{itemize}

Let us start proving the bound $\dfrac{1}{\Gt'_{J_{s_0s_1}^{q_0q_1}}(y')} \le \widetilde{C}' \cdot \Ft'(d_{s_0s_1}^{q_0q_1})$, for some $\widetilde{C}'\in \R^+$. We consider different cases.\\
\\
\textbf{\underline{Case $\boldsymbol{q_0=q_1=1}$}}\\
\\
We have $\Ft'(d_{s_0s_1}^{11}) = f_0'(d_{s_1}^{1}) f_0'(d_{s_0s_1}^{11})$ and $\dfrac{1}{\Gt'_{J_{s_0s_1}^{11}}(y')} = f_0'\left( g_0\left( \Gh_{J_{s_0}^{1}}(y')\right)\right) f_0'(g_0(y'))$.\\
If $g_0(y') \le d_{s_1}^1$, then, since $f_0',g_0$, and $g_1$ are increasing:
\begin{itemize}
\item $f_0' (g_0(y')) \le  f_0'( d_{s_1}^1)$.
\item $ f_0'\left(g_0 \left(\Gh_{J_{s_0}^{1}}(y')\right)\right) = f_0'(g_0\circ g_1^{s_0-1}\circ g_0 (y')) \le f_0'(g_0\circ g_1^{s_0-1} (d_{s_1}^{1})) = f_0'(d_{s_0s_1}^{11})$.
\end{itemize}
This implies that $\dfrac{1}{\Gt'_{J_{s_0s_1}^{11}}(y')} \le \Ft'(d_{s_0s_1}^{11})$.\\
On the other hand, if $g_0(y') > d_{s_1}^1$, then $f_0'\left(g_0 \left(\Gh_{J_{s_0}^{1}}(y')\right)\right) > f_0'(d_{s_0s_1}^{11})$ and bounded distortion for $f_0$ (see equation \eqref{boundeddistortionf0}) yields:
\begin{itemize}
\item $f_0'(g_0(y'))\le e^{C_A \cdot [f_0 (g_0(y'))-f_0(d_{s_1}^1)]} \cdot f_0'(d_{s_1}^1) \le e^{C_A \cdot (y'-2)}\cdot f_0'(d_{s_1}^1)$,\\
where we used that $f_0 (d_{s_1}^1) = f_0 ( g_0 \circ g_1^{s_1-1}(1)) = g_1^{s_1-1}(1) = s_1 \ge 2$ (which also implies that $y' > 2$).
\item $f_0'\left(g_0 \left(\Gh_{J_{s_0}^{1}}(y')\right)\right) \le e^{C_A \cdot \left[ f_0\left(g_0 \left(\Gh_{J_{s_0}^{1}}(y')\right)\right) - f_0(d_{s_0s_1}^{q_0q_1})\right]} \cdot f_0'(d_{s_0s_1}^{11}) \le e^{C_A (y'-1)} \cdot \Ft'(d_{s_0s_1}^{11})$,\\[4pt]
the last inequality following from:
\begin{align*}
f_0\left(g_0 \left(\Gh_{J_{s_0}^{1}}(y')\right)\right) - f_0(d_{s_0s_1}^{11}) =& \Gh_{J_{s_0}^1}(y') - f_0 ( g_0\circ g_1^{s_0-1} (d_{s_1}^1) = g_1^{s_0-1} \circ g_0 (y') - g_1^{s_0-1} (d_{s_1}^1))=\\
\ &= g_0(y')+s_0-1 - d_{s_1}^1 - (s_0-1) = g_0 (y') - d_{s_1}^1 \le 1.
\end{align*}
\end{itemize}
Thus, $\dfrac{1}{\Gt'_{J_{s_0s_1}^{11}}(y')} \le   e^{C_A (y'-1)} \cdot \Ft'(d_{s_0s_1}^{11})$. In the end, for any value of $y'$: 
\begin{align*}
\dfrac{1}{\Gt'_{J_{s_0s_1}^{11}}(y')} \le \max \lbrace 1 , e^{C_A (y'-1)} \rbrace \cdot \Ft'(d_{s_0s_1}^{11}).
\end{align*}

\textbf{\underline{Case $\boldsymbol{q_0=1, q_1\ge 2}$}}\\
\\
We have to estimate
\begin{align*}
\dfrac{1}{\Gt_{J_{s_0s_1}^{1q_1}}'(y')} =\left[ \prod\limits_{i=1}^{q_1-1}f_0' \left( g_0^{i} \circ g_1^{s_1-1} \circ g_0 \left(\Gh_{J_{s_0}^{1}}(y')\right) \right) \right]\cdot f_0'\left(g_0 \left(\Gh_{J_{s_0}^{1}}(y')\right)\right) \cdot f_0'(g_0(y'))
\end{align*}
from above with
\begin{align*}
\Ft'(d_{s_0s_1}^{1q_1} ) = \left[ \prod\limits_{i=1}^{q_1-1} f_0' (g_0^{i}(1)) \right] \cdot f_0'(d_{s_1}^{q_1}) \cdot f_0'(d_{s_0s_1}^{1q_1}),
\end{align*}
apart from a multiplicative constant. We carefully compare pairs of factors in the two formulas, progressing from left to right with respect to the terms in $\dfrac{1}{\Gt_{J_{s_0s_1}^{1q_1}}'(y')}$ and following the increasing index $i$.

\begin{itemize}
\item We couple $f_0'\left(g_0\circ g_1^{s_1-1}\circ g_0 \left( \Gh_{J_{s_0}^1}(y')\right)\right)$, the term with $i=1$, with $f_0' ( d_{s_1}^{q_1})$, the second-to-last in $\Ft'(d_{s_0s_1}^{1q_1} )$. In particular, we observe that, since $q_1 \ge 2$:
\begin{align*}
\Gh_{J_{s_0}^1}(y') = g_1^{s_0-1}\circ g_0(y') = s_0-1 + g_0(y') > s_0-1 \ge 1 \ge g_0^{q_1-2} (1),
\end{align*}
implying that
\begin{align*}
g_0\circ g_1^{s_1-1}\circ g_0 \left( \Gh_{J_{s_0}^1}(y')\right) > g_0\circ g_1^{s_1-1}\circ g_0 ( g_0^{q_1-2} (1)) =  g_0\circ g_1^{s_1-1}\circ  g_0^{q_1-1} (1) = d_{s_1}^{q_1}.
\end{align*}
Moreover, similarly as in the previous case, we note that
\begin{align*}
f_0\left(g_0\circ g_1^{s_1-1}\circ g_0 \left( \Gh_{J_{s_0}^1}(y')\right)\right) - f_0 ( d_{s_1}^{q_1}) = (s_1-1) + g_0 \left( \Gh_{J_{s_0}^1}(y')\right) - (s_1-1) - g_0^{q_1-1}(1) < 1.
\end{align*}
Thus, by bounded distortion: $f_0'\left(g_0\circ g_1^{s_1-1}\circ g_0 \left( \Gh_{J_{s_0}^1}(y')\right)\right) \le e^{C_A }\cdot f_0' ( d_{s_1}^{q_1})$.
\item For $i \in \lbrace 2, \dots , q_1-1 \rbrace$, compare $f_0' \left( g_0^{i} \circ g_1^{s_1-1} \circ g_0 \left(\Gh_{J_{s_0}^{1}}(y')\right) \right)$ with the term $f_0'(g_0^{i-1}(1))$ in the product appearing in $\Ft' (d_{s_0s_1}^{1q_1})$. Indeed, simply observe that $g_0 \circ g_1^{s_1-1} \circ g_0 \left(\Gh_{J_{s_0}^{1}}(y')\right) < 1$, so that, by increasing monotonicity of $f_0'$ and $g_0$:
\begin{align*}
f_0' \left( g_0^{i} \circ g_1^{s_1-1} \circ g_0 \left(\Gh_{J_{s_0}^{1}}(y')\right) \right) & = f_0' \left( g_0^{i-1}\left( g_0 \circ \circ g_1^{s_1-1} \circ g_0 \left(\Gh_{J_{s_0}^{1}}(y')\right) \right) \right) <  f_0' (g_0^{i-1} (1)).
\end{align*}
Thus, the majority of the terms in the products can be compared by a direct inequality, and the final constant will not depend on $q_1$.
\item Compare $f_0' \left( g_0 \left( \Gh_{J_{s_0}^1}(y') \right) \right)$ with $f_0' ( d_{s_0s_1}^{1q_1})$. If $g_0 \left( \Gh_{J_{s_0}^1}(y') \right) < d_{s_0s_1}^{1q_1}$, the inequality is clear. Otherwise, we use bounded distortion: similarly to the previous case, $f_0 \left( g_0 \left( \Gh_{J_{s_0}^1}(y') \right) \right) - f_0(d_{s_0s_1}^{1q_1})) \le 1$, so that $f_0' \left( g_0 \left( \Gh_{J_{s_0}^1}(y') \right) \right) \le e^{C_A} \cdot f_0' ( d_{s_0s_1}^{1q_1})$.
\item It remains to compare $f_0'(g_0(y'))$ and $f_0'(g_0^{q_1-1}(1))$. If $g_0(y') \le g_0^{q_1-1}(1)$, we are done.\\
If $g_0(y') > g_0^{q_1-1}(1)$, bounded distortion implies:
\begin{align*}
f_0'(g_0(y')) &\le e^{C_A \cdot [f_0 (g_0(y')) - f_0 (g_0^{q_1-1}(1))]} \cdot f_0'(g_0^{q_1-1}(1)) = \\
&= e^{C_A \cdot [y' - g_0^{q_1-2}(1)]} \cdot f_0'(g_0^{q_1-1}(1)) < e^{C_A \cdot y'} \cdot f_0'(g_0^{q_1-1}(1)).
\end{align*}
\end{itemize}
In the end, we found that
\begin{align*}
\dfrac{1}{\Gt_{J_{s_0s_1}^{1q_1}}'(y')} \le \left[ e^{C_A}\cdot e^{C_A} \cdot e^{C_A \cdot y'} \right]\cdot \Ft'(d_{s_0s_1}^{1q_1} )= e^{C_A ( 2+ y' )} \cdot \Ft'(d_{s_0s_1}^{1q_1} ).
\end{align*}

\textbf{\underline{Case $\boldsymbol{q_0\ge 2, q_1 = 1}$}}\\
\\
Here we have:
\begin{align*}
 \dfrac{1}{\Gt_{J_{s_0s_1}^{q_01}}'(y')} &= f_0'\left(g_0 \left(\Gh_{J_{s_0}^{q_0}}(y')\right)\right) \cdot \left[ \prod\limits_{i=1}^{q_0-1}f_0' \left( g_0^{i} \circ g_1^{s_0-1} \circ g_0 (y') \right) \right]\cdot f_0'(g_0(y')).\\
 \Ft'(d_{s_0s_1}^{q_01}) &= f_0'(d_{s_1}^{1}) \cdot \left[ \prod\limits_{i=1}^{q_0-1} f_0' (g_0^{i}(d_{s_1}^{1})) \right] \cdot f_0'(d_{s_0s_1}^{q_01}).
\end{align*}
First, observe that, since $q_0 \ge 2$, then
\begin{align*}
\Gh_{J_{s_0}^{q_0}}(y') = g_0^{q_0-1}\circ g_1^{s_0-1}\circ g_0 (y') = g_0(g_0^{q_0-2}\circ g_1^{s_0-1}\circ g_0 (y') )< 1 < s_1 = g_1^{s_1-1}(1).
\end{align*}
Thus $
f_0'\left( g_0 \left( \Gh_{J_{s_0}^{q_0}}(y') \right) \right) < f_0'\left( g_0 \circ g_1^{s_1-1}(1) \right) = f_0'(d_{s_1}^1)$, and the first factor is estimated.\\
We are left to compare the other factors, which can be wrote as $\dfrac{1}{\Gh'_{J_{s_0}^{q_0}}(y')}$ and $\Fh' (d_{s_0s_1}^{q_01})$. It is actually more convenient to match $\dfrac{1}{\Gh'_{J_{s_0}^{q_0}}(y')}$ with $\Fh' (d_{s_0}^{q_0})$, so we reduce to this case using bounded distortion on $\Fh$, i.e. $\Fh'(d_{s_0}^{q_0}) \le \widetilde{C}_D \cdot \Fh'(d_{s_0s_1}^{q_01})$. Therefore, we compare $\left[ \prod\limits_{i=1}^{q_0-1}f_0' \left( g_0^{i} \circ g_1^{s_0-1} \circ g_0 (y') \right) \right]\cdot f_0'(g_0(y'))$ and
\begin{align*}
\Fh'(d_{s_0}^{q_0}) = \left[ \prod\limits_{i=0}^{q_0-2} f_0' (f_0^i \circ f_1^{s_0-1} \circ f_0 (d_{s_0}^{q_0}))\right] \cdot f_0'(d_{s_0}^{q_0}) = \left[ \prod\limits_{i=1}^{q_0-1} f_0' (g_0^{i}(1)) \right] \cdot f_0'(d_{s_0}^{q_0}),
\end{align*}
where we used that $d_{s_0}^{q_0} = g_0\circ g_1^{s_0-1}\circ g_0^{q_0-1} (1)$. Let us consider $f_0' (g_0 \circ g_1^{s_0-1} \circ g_0 (y'))$, the term with index $i=1$: we match it with $f_0'(d_{s_0}^{q_0})$. The strategy is the same as in the previous cases: whether the argument $g_0 \circ g_1^{s_0-1} \circ g_0 (y')$ is smaller or greater than $d_{s_0}^{q_0}$, bounded distortion implies that
\begin{align*}
f_0' (g_0 \circ g_1^{s_0-1} \circ g_0 (y')) \le e^{C_A} \cdot f_0'(d_{s_0}^{q_0}),
\end{align*}
because
\begin{align*}
|f_0 (g_0 \circ g_1^{s_0-1} \circ g_0 (y')) - f_0(d_{s_0}^{q_0})| = |(s_0-1) + g_0(y') - (s_0-1) - g_0^{q_0-1} (1 )| \le 1.
\end{align*}
Now, note that $g_0 \circ g_1^{s_0-1}\circ g_0 (y') < 1$, so that, for any $i \in \lbrace 2 , \dots , q_0-1 \rbrace$,
\begin{align*}
f_0'(g_0^i \circ g_1^{s_0-1}\circ g_0 (y')) = f_0'(g_0^{i-1}(g_0 \circ g_1^{s_0-1}\circ g_0 (y'))) < f_0'(g_0^{i-1} ( 1 )).
\end{align*}
Lastly, compare $f_0'(g_0(y'))$ and $f_0'(g_0^{q_0-1}(1))$ through bounded distortion: if $g_0(y') \le g_0^{q_0-1}(1)$, the wanted inequality is immediate. Otherwise, we find:
\begin{align*}
f_0'(g_0(y')) &\le e^{C_A \cdot [ f_0 (g_0 (y')) - f_0 (g_0^{q_0-1}(1))]}\cdot f_0'(g_0^{q_0-1}(1))  = e^{C_A \cdot [ y' - g_0^{q_0-2}(1)]}\cdot f_0'(g_0^{q_0-1}(1))<\\
&<e^{C_A \cdot y'} \cdot f_0'(g_0^{q_0-1}(1)).
\end{align*}
We obtained that: $\dfrac{1}{\Gh'_{J_{s_0}^{q_0}}(y')} \le e^{C_A}\cdot e^{C_A\cdot y'} \cdot \Fh'(d_{s_0}^{q_0}) \le e^{C_A}\cdot e^{C_A\cdot y'} \cdot \widetilde{C}_D \cdot \Fh'(d_{s_0s_1}^{q_01})$.\\
Thus, summarizing the entire case, we found that
\begin{align*}
\dfrac{1}{\Gt'_{J_{s_0s_1}^{q_01}}(y')}  \le \left[ e^{C_A}\cdot e^{C_A\cdot y'} \cdot \widetilde{C}_D \right] \cdot \Ft'(d_{s_0s_1}^{q_01}) = \left[ \widetilde{C}_D \cdot e^{C_A (y'+1)}\right] \cdot \Ft'(d_{s_0s_1}^{q_01}).
\end{align*}

\textbf{\underline{Case $\boldsymbol{q_0\ge 2, q_2 \ge 2}$}}\\
\\
Let us compare
\begin{align*}
\frac{1}{\Gt'_{J_{s_0s_1}^{q_0q_1}}(y')} = \frac{1}{\Gh'_{J_{s_1}^{q_1}}\left( \Gh_{J_{s_0}^{q_0}}(y')\right)} \cdot \frac{1}{\Gh'_{J_{s_0}^{q_0}}(y')}
\end{align*}
with
\begin{align*}
\Ft'(d_{s_0s_1}^{q_0q_1}) = \Fh'(d_{s_1}^{q_1}) \cdot \Fh'(d_{s_0s_1}^{q_0q_1}).
\end{align*}
The exact same argument adopted in the second part of the previous case implies that
\begin{align*}
\frac{1}{\Gh'_{J_{s_0}^{q_0}}(y')} \le \left[ \widetilde{C}_D \cdot e^{C_A (y'+1)}\right] \cdot \Fh'(d_{s_0s_1}^{q_0q_1}),
\end{align*}
so that we only need to focus on the first terms. In particular:
\begin{align*}
\frac{1}{\Gh'_{J_{s_1}^{q_1}}\left( \Gh_{J_{s_0}^{q_0}}(y')\right)} & = \left[ \prod\limits_{i=1}^{q_1-1} f_0' \left( g_0^i \circ g_1^{s_1-1}\circ g_0 \left( \Gh_{J_{s_0}^{q_0}}(y')\right)\right) \right] \cdot f_0' \left( g_0 \left( \Gh_{J_{s_0}^{q_0}}(y') \right) \right).\\
\Fh' (d_{s_1}^{q_1} )& = \left[ \prod\limits_{i=1}^{q_1-1} f_0' ( g_0^i (1))\right]\cdot f_0' (d_{s_1}^{q_1}).
\end{align*}
Now, consider $f_0' \left( g_0 \circ g_1^{s_1-1}\circ g_0 \left( \Gh_{J_{s_0}^{q_0}}(y')\right)\right)$: we estimate it from above with  $f_0'(d_{s_1}^{q_1})$. Indeed, using bounded distortion if necessary, i.e. if $g_0 \circ g_1^{s_1-1}\circ g_0 \left( \Gh_{J_{s_0}^{q_0}}(y')\right) > d_{s_1}^{q_1}$, we obtain that
\begin{align*}
f_0' \left( g_0 \circ g_1^{s_1-1}\circ g_0 \left( \Gh_{J_{s_0}^{q_0}}(y')\right)\right) \le e^{C_A}\cdot f_0'(d_{s_1}^{q_1}).
\end{align*}
In the same way that we already saw in the case with $q_0 =1$ and $q_1 \ge 2$, we get, for $i \in \lbrace 2 , \dots , q_1-1 \rbrace$,
\begin{align*}
f_0' \left( g_0^i \circ g_1^{s_1-1}\circ g_0 \left( \Gh_{J_{s_0}^{q_0}}(y')\right)\right) < f_0'(g_0^{i-1}(1)).
\end{align*}
Lastly, note that, being $q_0 \ge 2$, the properties of $\Gh$ imply that $\Gh_{J_{s_0}^{q_0}} (y') <1$. In this way we have that: $\left|f_0\left(g_0 \left(\Gh_{J_{s_0}^{q_0}} (y')\right)\right)- f_0(g_0^{q_1-1}(1))\right|=\left|\Gh_{J_{s_0}^{q_0}} (y') - g_0^{q_1-2}(1)\right| < 1$.
This allows us to employ bounded distortion (if necessary) and find that $f_0'\left(g_0 \left(\Gh_{J_{s_0}^{q_0}} (y')\right)\right) \le e^{C_A}\cdot f_0'(g_0^{q_1-1}(1))$. This means that, in this last case,
\begin{align*}
\frac{1}{\Gt'_{J_{s_0s_1}^{q_0q_1}}(y')} \le \left[ \widetilde{C}_D \cdot e^{C_A (y'+1)} \cdot e^{C_A} \cdot e^{C_A} \right] \cdot \Ft'(d_{s_0s_1}^{q_0q_1})= \left[ \widetilde{C}_D \cdot e^{C_A (y'+3)} \right] \cdot \Ft'(d_{s_0s_1}^{q_0q_1}).
\end{align*}

We can collect all the estimates that we found and finally prove the first part of the statement. In particular, if we call $\widetilde{C}' = \widetilde{C}_D \cdot e^{C_A (y'+3)}$, then we have that $\dfrac{1}{\Gt'_{J_{s_0s_1}^{q_0q_1}}(y')} \le \widetilde{C}' \cdot \Ft'(d_{s_0s_1}^{q_0q_1})$, for any $s_0,s_1,q_0,q_1$.\\
\\
Let us turn to the second estimate, i.e.
\begin{align*}
\dfrac{1}{\widetilde{C}'} \cdot \Ft'(d_{s_0s_1}^{q_0q_1}) \le \dfrac{1}{\Gt'_{J_{s_0s_1}^{q_0q_1}}(y')}.
\end{align*}
The argument is basically the same that we just outlined, but with different matches for the factors appearing in $\Ft'(d_{s_0s_1}^{q_0q_1})$, expressed as
\begin{align*}
\left[ \prod\limits_{i=1}^{q_1-1} f_0' (g_0^{i}(1)) \right] \cdot f_0'(d_{s_1}^{q_1}) \cdot \left[ \prod\limits_{i=1}^{q_0-1} f_0' (g_0^{i}(d_{s_1}^{q_1})) \right] \cdot f_0'(d_{s_0s_1}^{q_0q_1})
\end{align*}
and $\dfrac{1}{\Gt'_{J_{s_0s_1}^{q_0q_1}}(y')}$, i.e.
\begin{align*}
\left[ \prod\limits_{i=1}^{q_1-1}f_0' \left( g_0^{i} \circ g_1^{s_1-1} \circ g_0 \left(\Gh_{J_{s_0}^{q_0}}(y')\right) \right) \right]\cdot f_0'\left(g_0 \left(\Gh_{J_{s_0}^{q_0}}(y')\right)\right) \cdot \left[ \prod\limits_{i=1}^{q_0-1}f_0' \left( g_0^{i} \circ g_1^{s_0-1} \circ g_0 (y') \right) \right]\cdot f_0'(g_0(y')).
\end{align*}
For such a reason, we only sketch the idea.\\
Suppose that $q_1 \ge 2$. Then, for $i \in \lbrace 2 , \dots , q_1-1 \rbrace$, we compare $f_0'(g_0^i(1))$ from $\Ft'(d_{s_0s_1}^{q_0q_1})$ with the term $f_0' \left( g_0^{i} \circ g_1^{s_1-1} \circ g_0 \left(\Gh_{J_{s_0}^{q_0}}(y')\right) \right)$ from $\dfrac{1}{\Gt'_{J_{s_0s_1}^{q_0q_1}}(y')}$.\\
Just note that, since $g_0(1) < 1 \le s_1-1 < g_1^{s_1-1}\circ g_0 \left(\Gh_{J_{s_0}^{q_0}}(y')\right)$, then
\begin{align*}
f_0'(g_0^i(1)) = f_0'(g_0^{i-1}(g_0(1))) < f_0' \left(g_0^{i-1} \left( g_1^{s_1-1}\circ g_0 \left(\Gh_{J_{s_0}^{q_0}}(y')\right) \right)\right),
\end{align*}
which leads to the wanted bounds.\\
If $q_0 =1$, we simply apply bounded distortion to the remaining terms. If $q_0 \ge 2$, we adopt the same scheme as before and estimate the terms appearing in the related product sign.\\
The case $q_0=q_1=1$ is immediately handled with bounded distortion, if necessary.\\
In the end, up to renaming the constant $\widetilde{C}'$, we finally obtain that
\begin{align*}
\dfrac{1}{\widetilde{C}'} \cdot \Ft'(d_{s_0s_1}^{q_0q_1}) \le \dfrac{1}{\Gt'_{J_{s_0s_1}^{q_0q_1}}(y')},
\end{align*}
thus completing the proof .
\end{proof}

\subsection*{Acknowledgments} 
NB and CB are partially supported by the PRIN Grant 2022NTKXCX ``Stochastic properties of dynamical systems'' funded by the Ministry of University and Research, Italy. NB and CB acknowledge the MUR Excellence Department Project awarded to the Department of Mathematics, University of Pisa, CUP I57G22000700001. This research is part of CB's activity within the UMI Group ``DinAmicI'' \texttt{www.dinamici.org} and the Gruppo Nazionale di Fisica Matematica, INdAM, Italy.
PV was partially supported by CIDMA under the
Portuguese Foundation for Science and Technology 
(FCT, https://ror.org/00snfqn58)  
Multi-Annual Financing Program for R\&D Units,
grants UID/4106/2025 and UID/PRR/4106\-/2025.
https://doi.org/\-10.54499/UID/04106/2025.

\end{document}